\documentclass{article}
\usepackage{graphicx} % Required for inserting images
\usepackage{geometry}
\usepackage{amsmath, amsfonts, amssymb,amsthm,nccmath,bbm}
\usepackage[colorlinks,linkcolor=blue,citecolor=red]{hyperref}
\usepackage{verbatim}
\usepackage{appendix}
\usepackage{abstract}
\numberwithin{equation}{section}

\newtheorem{theorem}{Theorem}[section]
\newtheorem{cor}{Corollary}[section]

\newtheorem{lem}{Lemma}[section]
\newtheorem{mydef}{Definition}[section]
\newtheorem{rmk}{Remark}[section]
\newtheorem{cond}{Assumption}

\def\dbx{d\boldsymbol{x}}
\def\bx{\boldsymbol{x}}
\def\dby{d\boldsymbol{y}}
\def\u{\boldsymbol{u}}
\def\dv{\mathrm{div}}
\def\dt{\frac{d}{dt}}
\def\intb{\int_{B(0;a(t))}}
\def\intr{\int_{\mathbb{R}^3}}

\def\intxm{\int_{x_0}^M}
\def\inttxm{\int_0^T\int_{x_0}^M}

\def\ptu{\partial_{\tau}u}
\def\pxu{\partial_xu}

\def\vcons{\eta+\frac{4}{3}\varepsilon}
\def\pvcons{(\eta+\frac{4}{3}\varepsilon)}

\def\sqphi{\sqrt{\varphi}}
\def\ptau{\partial_{\tau}}
\def\px{\partial_x}
\def\pt{\partial_t}
\def\pr{\partial_r}

\def\rhox{\rho^{\xi}}
\def\ux{u^{\xi}}
\def\ax{a^{\xi}}
\def\rhoxj{\rho^{\xi_j}}
\def\uxj{u^{\xi_j}}
\def\axj{a^{\xi_j}}
\def\Linf{L^{\infty}}

\def\rhoin{\rho_{in}}
\def\uin{u_{in}}
\def\rin{r_{in}}

\def\LTwoHOne{L^2(0,T;H^1([0,r_{in}(t)),r^2dr))}
\def\chid{\chi_{\delta}}
\def\One{\mathbbm{1}}
\def\supp{\text{supp}}
\def\Dt{\mathcal{D}_t}

\title{Expanding Solutions to Free Boundary 3D Spherically Symmetric Compressible Navier-Stokes-Poisson Equations near the Lane-Emden Stars}
\author{Han Cao\thanks{Department of Mathematics, University of Southern California, 3620 S. Vermont Ave., Los Angeles, CA 90089-2532, USA. Email: caohan@usc.edu.}}
\date{ }

\begin{document}
    \maketitle
    \begin{abstract}
    We consider the gravitational Navier-Stokes-Poisson equations with the equation of state $P(\rho)=K\rho^{\gamma}$, where $\gamma\in(\frac{6}{5},\frac{4}{3}]$, which models the viscous polytropic gaseous stars. We prove the existence of global weak solutions to the equations with constant viscosity and radially symmetric initial data. For $\gamma=\frac{4}{3}$, we require the initial data having mass less than the mass of the Lane-Emden stars; for $\gamma\in(\frac{6}{5},\frac{4}{3})$, we require that the initial data belong to an invariant set where initial data can be taken near the Lane-Emden stars. For $\gamma\in(\frac{6}{5},\frac{4}{3})$, we show that the invariant set contains some initial data that are not allowed in previous literature. We also prove the support of any strong solution expands to infinity for the Navier-Stokes-Poisson equations with constant viscosity and a class of density-dependent viscosity, which indicates the strong instability of Lane-Emden solutions for the Navier-Stokes-Poisson equations.
\end{abstract}
    \tableofcontents
    \section{Introduction}
The evolution of viscous gaseous stars can be described by the Navier-Stokes-Poisson equations:
\begin{equation}
\begin{cases}
    \partial_t\rho+\mathrm{div}(\rho \boldsymbol{u})=0,\\
    \partial_t(\rho\boldsymbol{u})+\mathrm{div}(\rho\boldsymbol{u}\otimes\boldsymbol{u})+\nabla p=\mathrm{div}T-\rho\nabla\varPhi,\\
    \Delta\varPhi=4\pi\rho,
\end{cases}\label{nsp1}
\end{equation}
where $(\boldsymbol{x},t)\in\mathbb{R}^3\times\mathbb{R}^{+}$, $\rho(\boldsymbol{x},t)\ge0$ is the density, $\boldsymbol{u}(\boldsymbol{x},t)\in\mathbb{R}^3$ is the velocity. $p$ is the pressure satisfying
\begin{equation}\label{Polytropic}
    p=p(\rho)=K\rho^{\gamma}\text{ with }\gamma\in(1,2),
\end{equation}
which is the equation of state for a polytropic gas, where $K>0$ is a constant and $\gamma$ is the adiabatic exponent. For convenience, we set $K=1$. $\varPhi$ is the potential function of the self-gravitational force and the stress tensor $T$ is given by
\begin{equation}\label{StressTen}
    T=\varepsilon\left(\nabla\boldsymbol{u}+\nabla\boldsymbol{u}^T-\frac{2}{3}(\mathrm{div}\boldsymbol{u})I_{3\times3}\right)+\eta(\mathrm{div}\boldsymbol{u})I_{3\times3},
\end{equation}
where $\varepsilon\ge0$ is the shear viscosity and $\eta\ge0$ is the bulk viscosity. Viscosity coefficients can be either constant or density-dependent, i.e., $\varepsilon=\varepsilon(\rho)\ge0,\ \eta=\eta(\rho)\ge0$. One example is $\varepsilon(\rho)=\varepsilon_0\rho^{\alpha},\ \eta(\rho)=\eta_0\rho^{\alpha}$ with $\alpha>0$ \cite{DL,GLX,LXZ2}. On the other hand, if $\varepsilon=0,\ \eta=0$, then (\ref{nsp1}) models the inviscid self-gravitating gaseous stars, called the Euler-Poisson system. Our results mainly focus on the constant viscosity model with at least one of $\varepsilon$ and $\eta$ positive.

One type of special solution is the Lane-Emden solution, that is, the non-rotating steady solution of (\ref{nsp1}). We denote $(\rho_{\mu}(r),0)$ to be such a radially symmetric solution with central density $\rho_{\mu}(0)=\mu$. Then $\rho_{\mu}$ satisfies:
\begin{equation}\label{nsp2}
    \begin{cases}
        \nabla\rho_{\mu}^{\gamma}=-\rho_{\mu}\nabla\varPhi_{\mu},\\
        \Delta\varPhi_{\mu}=4\pi\rho_{\mu}.
    \end{cases}
\end{equation}
By rewriting (\ref{nsp2}), we obtain the following ODE
\begin{equation}\label{ode}
    \pr\rho_{\mu}^{\gamma}=-\frac{4\pi\rho_{\mu}}{r^2}\int_0^r\rho_{\mu}s^2ds.
\end{equation}

The studies of the equation (\ref{ode}) are of great importance in both astrophysics and mathematics. For different values of $\gamma$, the equation is applied to different stellar structures. For example, for $\gamma=\frac{4}{3}$, it is called the ''standard model'' for the sun, which was first proposed by Eddington \cite{Chan}. For $\gamma=\frac{5}{3}$, the model is applied to non-relativistic degenerate electron gas, which is a reasonable approximation to the white dwarfs of small mass \cite{KWW}. For $\gamma\searrow1$, the model is applied to heavier molecules. As for the solutions to (\ref{ode}), we have the following results. When $\gamma\in(\frac{6}{5},2)$, for given finite total mass, there exists at least one compactly supported solution to (\ref{ode}), while every solution is compactly supported and unique when $\gamma\in(\frac{4}{3},2)$. When $\gamma=\frac{6}{5}$, there exists a unique solution with explicit expression, which has infinite support. When $\gamma\in(1,\frac{6}{5})$, there are no solutions with finite mass \cite{LSS}.

When $\gamma\in(\frac{6}{5},2)$, we denote the radius of the support of the Lane-Emden solution by $0<R_{\mu}<+\infty$, the total mass and the total energy by 

\begin{equation}
    M_{\mu}=\int_0^{R_{\mu}}4\pi\rho_{\mu}(r) r^2dr,
\end{equation}

\begin{equation}
    E_{\mu}=\frac{1}{\gamma-1}\int_0^{R_{\mu}}4\pi\rho_{\mu}^{\gamma} r^2dr-\int_0^{R_{\mu}}4\pi\rho_{\mu}r\int_0^r4\pi\rho_{\mu}s^2dsdr.
\end{equation}
By scaling, we can see that if $\rho_1(r)$ is the Lane-Emden solution with central density 1, then $\beta^{\frac{2}{2-\gamma}}\rho_1(\beta r)$ is also a Lane-Emden solution, with central density $\mu=\beta^{\frac{2}{2-\gamma}}$. What's more, we have \cite{Chan}:
\begin{equation}
    M_{\mu}=M_1\mu^{\frac{3\gamma-4}{2}},\ R_{\mu}=R_1\mu^{\frac{\gamma-2}{2}},\ E_{\mu}=E_1\mu^{\frac{5\gamma-6}{2}}.
\end{equation}
We call the case $\gamma=\frac{4}{3}$ mass critical, and the case $\gamma=\frac{6}{5}$ energy critical. These are also the mass critical and energy critical for Euler-Poisson equations \cite{HJ} respectively. For $\gamma>\frac{6}{5}$, $\rho_{\mu}$ satisfies \cite{LSS}
\begin{equation}
    \rho_{\mu}(r)\sim(R_{\mu}-r)^{\frac{1}{\gamma-1}}\text{ for $r$ near $R_{\mu}$},
\end{equation}
which is referred to as the physical vacuum boundary condition \cite{LY}.

In astrophysics, it has been conjectured that the Lane-Emden solutions are stable when $\gamma>\frac{4}{3}$, and unstable when $\gamma<\frac{4}{3}$. In mathematics, there are many recent studies on the stability or instability of the Lane-Emden solutions. The linear instability for $\frac{6}{5}<\gamma<\frac{4}{3}$ and linear stability for $\frac{4}{3}\le\gamma<2$ of $(\rho_{\mu},0)$ are proved in \cite{LSS}. In particular, $\gamma=\frac{4}{3}$ is neutrally stable. In this case, a type of nonlinear instability for the Euler-Poisson system was found by Deng, Liu, Yang and Yao \cite{DLYY}, in the sense that small perturbation of the Lane-Emden solutions will lead to mass escaping to infinity. For $\frac{4}{3}<\gamma<2$, the nonlinear conditional stability of the Lane-Emden solutions for the Euler-Poisson equations was proved by Rein \cite{Rein}, while for $\frac{6}{5}<\gamma<\frac{4}{3}$ and $\gamma=\frac{6}{5}$, the nonlinear instability for the Euler-Poisson equations were proved by Jang \cite{J3,J1}. As for Navier-Stokes-Poisson equations, for $\frac{4}{3}<\gamma<2$, the nonlinear asymptotic stability of the Lane-Emden solutions under radial perturbation was proved by Luo, Xin and Zeng \cite{LXZ1,LXZ2}, while for $\frac{6}{5}<\gamma<\frac{4}{3}$, the nonlinear instability of the Lane-Emden solutions was proved by Jang and Tice \cite{JT}. There are also some studies on the stability of gaseous stars with a more general pressure $p(\rho)$ satisfying $p(\rho)\approx\rho^{\gamma_0}$ near $\rho=0$. Lin and Zeng \cite{LZ} proved the linear stability for the Euler-Poisson equations is entirely determined by the mass–radius curve of the steady state parametrized by the central density, which is referred to as the turning point principle for the Euler-Poisson equations, and they proved the linear stability of the Euler-Poisson equations for $\frac{4}{3}<\gamma_0<2$ and linear instability for $\frac{6}{5}<\gamma_0<\frac{4}{3}$. Lin, Wang and Zhu \cite{LWZ} proved that the linear stability from turning point principle implies nonlinear orbital stability for the Euler-Poisson equations under radial perturbation. The stability is unconditional in the sense that global radial weak solutions near the steady state can be constructed \cite{CHLWW}. Cheng, Lin and Wang showed the turning point principle for the Navier-Stokes-Poisson equations and proved the linear and nonlinear instability of the Navier-Stokes-Poisson equations for $\frac{6}{5}<\gamma_0<\frac{4}{3}$ and the linear and nonlinear stability for $\frac{4}{3}<\gamma_0<2$ in \cite{CLW}. The stability results of \cite{LWZ,LZ, CLW} apply to the model of white dwarf stars proposed by Chandrasekhar.

For the case $\gamma =\frac{4}{3}$, we would also like to mention that there exists a family of expanding/collapsing solutions (Goldreich-Weber solutions) to the Euler-Poisson equations which can be explicitly constructed.
This family of solutions was first discovered by Goldreich and Weber \cite{GW} and the Lane-Emden solution is a special case within the family. Moreover, Goldreich-Weber solutions are explicit solutions that exhibit the nonlinear instability of mass critical Lane-Emden solution found by \cite{DLYY}. In the vacuum free boundary framework, Had\v{z}i\'c and Jang proved the nonlinear stability of the expanding Goldreich-Weber solutions for the Euler-Poisson equations under radial \cite{HJ} and nonradial \cite{HJL} perturbations. While for the Navier-Stokes-Poisson equations with constant viscosity, the instability of self-similar expanding Goldreich-Weber solutions and nonlinear stability of linear expanding Goldreich-Weber solutions were proved in \cite{Liu2019Isentropic} by Liu. For the model of viscous radiation gaseous star, Liu proved the existence of stationary solutions \cite{LXin2} and the stability of a class of linearly expanding homogeneous solutions \cite{LXin}.

The existence and dynamics of solutions to the free boundary value problem for the Navier-Stokes-Poisson equations have been widely studied, while the problem is challenging because of the vacuum boundary, especially when considering the density-dependent viscosity. The local-in-time well-posedness of spherically symmetric strong solutions to the free boundary value problem with constant viscosity for $\frac{6}{5}<\gamma<2$ was established by Jang \cite{J2}. The global existence of spherically symmetric weak solutions was constructed in \cite{G} for $\gamma>\frac{4}{3}$, where the negative gravitational energy can be controlled by the mass and the internal energy. While for $\frac{6}{5}<\gamma\le\frac{4}{3}$, Kong and Li \cite{KL} proved the existence of spherically symmetric global weak solutions for constant viscosity problem and Duan and Li \cite{DL} proved the existence of spherically symmetric global weak solutions for viscosity coefficients $\varepsilon(\rho)=\frac{1}{2}\rho$ and $\eta(\rho)=\frac{1}{3}\rho$. In both choices of viscosity, total mass of the solutions is required to be less than a certain critical mass. As for radially symmetric global strong solutions with $\frac{4}{3}<\gamma<2$, the existence was proved by Luo, Xin and Zeng for both constant viscosity \cite{LXZ1} and density-dependent viscosity \cite{LXZ2} when the initial data is a small radial perturbation of the Lane-Emden solution with the same mass.

There are also a lot of results on the existence of global weak solutions to the Euler-Poisson equations. In \cite{DLYY}, it was proved that the gravitational energy can be controlled by the mass and the internal energy for $\gamma>\frac{4}{3}$; while for $\gamma=\frac{4}{3}$, to control the gravitational energy, the total mass should be less than a critical mass. Chen, Lin, Wang and Yuan \cite{CHWY} proved the global existence of spherically symmetric weak solutions to the Euler-Poisson equations with large initial data for dimension-$n$, where $\gamma>\frac{2n}{n+2}$. In particular, when $\frac{2n}{n+2}<\gamma\le\frac{2(n-1)}{n}$, the total mass is assumed to be less than a certain critical mass. We remark that for $\gamma=\frac{4}{3}$ (or $\gamma=\frac{2(n-2)}{n}$ in dimension-$n$), the critical mass is a constant independent of the initial energy, while for $\frac{6}{5}<\gamma<\frac{4}{3}$ (or $\frac{2n}{n+2}<\gamma<\frac{2(n-1)}{n}$ in dimension-$n$) the critical mass depends on the initial energy. Also see the existence of global radially symmetric weak solutions to the Euler-Poisson equations with general pressure laws in \cite{CHLWW} by Chen, Huang, Li, Wang and Wang. In \cite{CCL}, Cheng, Cheng and Lin proved that the sharp critical mass for $\gamma=\frac{4}{3}$ is the mass of the Lane-Emden solution, which is the Chandrasekhar mass when $p(\rho)=\rho^{\frac{4}{3}}$, and it coincides with the critical mass $M_0$ for the Euler-poisson system proved by Fu and Lin \cite{FL}. A variational method for $\frac{6}{5}<\gamma<\frac{4}{3}$ is also established in \cite{CCL} to control the negative gravitational energy. In both ranges of $\gamma$, they proved the existence of global weak solutions to the Euler-Poisson equations where the initial data can be taken near the Lane-Emden solution.

As for the expanding rate of the free boundary, it was proved that the support of solutions to Euler-Poisson equations \cite{DLYY,CCL}, compressible Navier-Stokes equations \cite{GLX} or Navier-Stokes-Poisson equations \cite{KL} expand at an algebraic rate.

Inspired by \cite{CCL}, we prove the existence of radial weak solutions to the Navier-Stokes-Poisson equations with constant viscosity and $\gamma\in(\frac{6}{5},\frac{4}{3}]$, using a similar variational method to control the negative gravitational energy. We also prove the algebraic expanding rate of the support of strong solutions to the Navier-Stokes-Poisson equations with constant or power-law viscosity with the same range of $\gamma$.

\section{Formulation and main results}
Consider the radially symmetric solutions to (\ref{nsp1}) with constant viscosity where at least one of $\varepsilon$ and $\eta$ are strictly positive and another one non-negative, which describe the motion of a nonrotating star. In this case, let $r=|\boldsymbol{x}|$, we can write:
\begin{equation}\label{VectURHO}
    \rho(\boldsymbol{x},t)=\rho(r,t)\text{ and }\boldsymbol{u}(\boldsymbol{x},t)=u(r,t)\frac{\boldsymbol{x}}{r}.
\end{equation}

We assume that the gaseous star is compactly supported, with a free boundary where the star's density continuously meets the vacuum. Assume $a(t)>0$ to be the radius of the support of $\rho(r,t)$ such that
\begin{equation}
    \rho(r,t)>0\text{ for }r\in[0,a(t))\text{ and }\rho(a(t),t)=0.
\end{equation}
We also impose the kinematic condition
\begin{equation}
    \frac{d}{dt}a(t)=u(a(t),t),
\end{equation}
and the continuity of the effective viscous flux $(pI_{3\times3}-(\eta+\frac{4}{3}\varepsilon)(\dv\u)I_{3\times3})\cdot\nu=0$ at the free boundary \cite{JZ,KL}, where $\nu=\frac{\boldsymbol{x}}{r}$ is the outer normal unit vector of the free surface. Then (\ref{nsp1}) with radially symmetry reads as

\begin{equation}\label{nsp3}
    \begin{cases}
        \partial_t\rho+\frac{1}{r^2}\partial_r(r^2\rho u)=0,\\
        \rho(\partial_t u+u\partial_ru)+\partial_r\rho^{\gamma}=-\frac{4\pi\rho}{r^2}\int_0^r\rho(s,t)s^2ds+\partial_r\left(\frac{\frac{4}{3}\varepsilon+\eta}{r^2}\partial_r(r^2u)\right),
    \end{cases}
\end{equation}
with boundary conditions
\begin{equation}\label{boundary conditions}
    \begin{aligned}
        \sigma:=\rho^{\gamma}-&(\eta+\frac{4}{3}\varepsilon)(\partial_ru+\frac{2u}{r})=0\text{ at }(a(t),t),\\
        \rho(a(t),t)&=0,\ u(a(t),t)=a'(t),\ u(0,t)=0,
    \end{aligned}
\end{equation}
for $(r,t)\in\{(r,t):0\le r\le a(t),0\le t\le T\}.$

Further, it's also convenient to use the Lagrangian coordinates. Let
\begin{equation}
    \begin{cases}
        x(r,t)=4\pi\int_0^r\rho(s,t)s^2ds,\\
        \tau(r,t)=t,
    \end{cases}
\end{equation}
then
\begin{equation*}
    \partial_r=4\pi\rho r^2\partial_x,\ \partial_t=-4\pi\rho ur^2\partial_x+\partial_{\tau},
\end{equation*}
and (\ref{nsp3}) becomes
\begin{equation}\label{nsp4}
    \begin{cases}
        \partial_{\tau}\rho+4\pi\rho^2\partial_x(r^2u)=0,\\
        \partial_{\tau}u+4\pi r^2\partial_x\rho^{\gamma}+\frac{x}{r^2}=16\pi^2r^2\partial_x\left((\frac{4\varepsilon}{3}+\eta)\rho\partial_x(r^2u)\right),
    \end{cases}
\end{equation}
with boundary conditions
\begin{equation}\label{boundary conditions: Lag}
\begin{aligned}
    \sigma=\rho^{\gamma}-(\eta+\frac{4}{3}\varepsilon)(4\pi r^2\rho\partial_x&u+\frac{2u}{r})=0\text{ at }x=M,\\
    u(0,\tau)=0,\ &\rho(M,\tau)=0.
\end{aligned}
\end{equation}

\begin{mydef}[global weak solutions]\label{WeakSolutionRadSymmetry}
    For any given time $T>0$, $(\rho,u,a)$ with $\rho\ge0\ a.e.$ is said to be a global weak solution to the free boundary value problem (\ref{nsp3})-(\ref{boundary conditions}) on $[0,a(t))\times[0,T]$ provided that
    \begin{equation*}
        \rho\in L^{\infty}(0,T;L^1\cap L^{\gamma}([0,a(t)),r^2dr)),
    \end{equation*}
    \begin{equation*}
        \ \sqrt{\rho}u\in L^{\infty}(0,T;L^2([0,a(t)),r^2dr)),
    \end{equation*}
    \begin{equation*}
        \pr u+\frac{2u}{r}\in L^2(0,T;L^2([0,a(t)),r^2dr)),\ a(t)\in H^1([0,T]),
    \end{equation*}
    and the equations are satisfied in the sense of distribution, i.e., for any $0\le t_1<t_2\le T$, and for any $\varphi\in C^1_c([0,a(t))\times[0,T])$,
    \begin{equation}\label{WkSolnM_Rad}
        \int_{0}^{a(t)}\rho\varphi r^2drdt|_{t_1}^{t_2}-\int_{t_1}^{t_2}\int_{0}^{a(t)}(\rho\pt\varphi+\rho u\pr\varphi)r^2drdt=0,
    \end{equation}
    and for any $0\le t_1<t_2\le T$ and any $\psi\in C^1_c([0,a(t))\times[0,T])$ satisfying $|\frac{\psi(r,t)}{r}|\le C$ near $r=0$ for some $C>0$,
    \begin{equation}\label{WkSolnMo_Rad}    
         \begin{aligned}
        &\int_0^{a(t)}\rho u r^2\psi dr|_{t_1}^{t_2}-\int_{t_1}^{t_2} \int_0^{a(t)}\rho u\pt\psi r^2+\rho^\gamma(\pr\psi+\frac{2\psi}{r})r^2drdt+\rho u^2\pr\psi r^2drdt\\
        =&-\int_{t_1}^{t_2} \int_0^{a(t)}4\pi\rho\psi\int_0^r\rho(s,t)s^2dsdrdt-\int_{t_1}^{t_2} \int_0^{a(t)}(\eta-\frac{2}{3}\varepsilon)(\pr u+\frac{2u}{r})(\pr\psi+\frac{2\psi}{r})r^2drdt\\
        &-\int_{t_1}^{t_2} \int_0^{a(t)}2\varepsilon(\pr u\pr\psi+\frac{2u\psi}{r^2})r^2drdt.
    \end{aligned}
    \end{equation}
    and
    \begin{equation}\label{WkSolnP_Rad}
        \pr^2\varPhi+\frac{2}{r}\pr\varPhi=4\pi\rho\ a.e.,
    \end{equation}
    where $\varPhi\rightarrow0$ as $r\rightarrow\infty$ and $\rho=0$ in $\mathbb{R}_{\ge0}\backslash[0,a(t))$ at time $t\in[0,T]$.
    Also,
    \begin{equation}
    \rho(a(t),t)=0\ and\ \sigma(a(t),t)=\left(\rho^{\gamma}-(\eta+\frac{4}{3}\varepsilon)(\pr u+\frac{2u}{r})\right)(a(t),t)=0
    \end{equation}
in the sense of trace.
\end{mydef}
\begin{rmk}\label{WeakSolution3D}
    In fact, if $\rho(r,t),\ u(r,t)$ satisfy (\ref{nsp3})-(\ref{boundary conditions}) in the sense of (\ref{WkSolnM_Rad})-(\ref{WkSolnP_Rad}), then $\rho(\bx,t),\ \boldsymbol{u}(\bx,t)$ given by (\ref{VectURHO}) also satisfy (\ref{nsp1}) in the sense of distribution in 3D space. To be specific, if we let $\Omega_t$ be the support of $\rho$ at time $t\in[0,T]$, then (\ref{WkSolnM_Rad})-(\ref{WkSolnP_Rad}) give: for any given $T>0$ and $0\le t_1<t_2\le T$, and for any $\varphi\in C^1_c({\Omega_t}\times[0,T])$,
    \begin{equation}\label{WkSolnM_3D}
        \int_{\Omega_t}\rho\varphi\dbx|_{t_1}^{t_2}=\int_{t_1}^{t_2}\int_{\Omega_t}\rho\partial_t\varphi+\rho\u\cdot\nabla_x\varphi\dbx dt,
    \end{equation}
    and for any $0\le t_1<t_2\le T$ and any $\boldsymbol{\psi}=(\psi_1,\psi_2,\psi_3)\in (C^1_c({\Omega_t}\times[0,T]))^3$ satisfying $\boldsymbol{\psi}(\bx,t)=0$ on $\partial\Omega_t$,
    \begin{equation}   \label{WkSolnMo_3D} 
        \begin{aligned}                 &\int_{\Omega_t}\rho\u\cdot\boldsymbol{\psi}\dbx|_{t_1}^{t_2}-\int_{t_1}^{t_2}\int_{\Omega_t}\rho\u\cdot\boldsymbol{\psi}+\rho\u\otimes\u:\nabla\boldsymbol{\psi}+\rho^{\gamma}div\boldsymbol{\psi}\dbx dt \\
        =&-\int_{t_1}^{t_2}\int_{\Omega_t}(\eta-\frac{2}{3}\varepsilon)div\u\cdot div\boldsymbol{\psi}+\varepsilon(\nabla\u+\nabla\u^T):\nabla\boldsymbol{\psi}\dbx dt-\int_{t_1}^{t_2}\int_{\Omega_t}\rho\nabla\varPhi\cdot\boldsymbol{\psi}\dbx dt,
        \end{aligned}        
    \end{equation}
    and
    \begin{equation}\label{WkSolnP_3D}
        \Delta\varPhi=4\pi\rho\ a.e.,
    \end{equation}
    where $\varPhi\rightarrow0$ as $|x|\rightarrow\infty$ and $\rho=0$ in $\mathbb{R}^3\backslash\Omega_t$ at time $t$. Also,
    \begin{equation}
        \rho|_{B(0;a(t))}=0\text{ and }(pI_{3\times3}-(\eta+\frac{4}{3}\varepsilon)(\dv\u)I_{3\times3})\cdot\nu=0
    \end{equation}
    in the sense of trace.
\end{rmk}

Suppose $(\rho,u,a)$ is a weak solution to (\ref{nsp3})-(\ref{boundary conditions}), we define the mass
\begin{equation}
    M(\rho(t)):=\int_0^{a(t)}4\pi\rho r^2dr
\end{equation}
and the energy
\begin{equation}\label{eeenergy}
\begin{aligned}
    E(\rho(t),u(t))=&4\pi\int_0^{a(t)}\frac{1}{2}\rho u^2 r^2+\frac{1}{\gamma-1}\rho^{\gamma} r^2dr-4\pi\int_0^{a(t)}4\pi\rho r\int_0^r\rho s^2dsdr\\
    =&4\pi\int_0^{a(t)}\frac{1}{2}\rho u^2 r^2dr+Q(\rho(t))+\frac{4-3\gamma}{\gamma-1}\int_0^{a(t)}4\pi\rho^{\gamma} r^2dr,
\end{aligned}
\end{equation}
where
\begin{equation}\label{FunctionalQ}
    Q(\rho(t)):=3\int_0^{a(t)}4\pi\rho^{\gamma}r^2dr-4\pi\int_0^{a(t)}4\pi\rho r\int_0^r\rho s^2dsdr.
\end{equation}
For any $\mu>0$, define
\begin{equation}\label{FunctionalS}
    S_{\mu}(\rho(t))=\frac{1}{\gamma-1}\int_0^{a(t)}4\pi\rho^{\gamma} r^2dr-4\pi\int_0^{a(t)}4\pi\rho r\int_0^r\rho s^2dsdr-\varPhi(R_{\mu})\int_0^{a(t)}4\pi\rho  r^2dr.
\end{equation}

In order to control the negative part of the energy (\ref{eeenergy}), for $\gamma=\frac{4}{3}$, we require the mass $M(\rho_0)<M_{ch}$, where $M_{ch}$ is the mass of the mass critical Lane-Emden solutions ($\gamma=\frac{4}{3}$). For $\gamma\in(\frac{6}{5},\frac{4}{3})$, we assume the initial data $(\rho_0,u_0)\in\mathcal{I}$, where
\begin{equation}\label{Linmass}
    \medmath{\mathcal{I}=\{(\rho,u):Q(\rho)>0,M(\rho)<\left(\frac{5\gamma-6}{2(\gamma-1)}\right)^{\frac{2(\gamma-1)}{5\gamma-6}}\left(\frac{4-3\gamma}{5\gamma-6}\right)^{\frac{4-3\gamma}{5\gamma-6}}l_1^{\frac{2(\gamma-1)}{5\gamma-6}}\frac{R_1}{M_1}\left(E(\rho,u)\right)^{\frac{3\gamma-4}{5\gamma-6}}\}}
\end{equation}
and $l_1=S_1(\rho_1)$. For convenience, we state
\begin{cond}\label{cond11}

\begin{equation}\label{cond1}
\begin{aligned}
    &M(\rho_0)<M_{ch},\text{ if $\gamma=\frac{4}{3}$};\\
    &(\rho_0,u_0)\in\mathcal{I},\text{ if $\gamma\in(\frac{6}{5},\frac{4}{3})$}.
\end{aligned}
\end{equation}
\end{cond}
Our first main theorem is the following.

\begin{theorem}[Global existence]\label{Thm1: ExistenceOfGlobalWkSoln}Let $T>0$ and $\gamma\in(\frac{6}{5},\frac{4}{3}]$.
    Suppose $(\rho_0(r),u_0(r))$ are initial data satisfying
\begin{equation}\label{conditionsI}
    \begin{aligned}
    &\rho_0\ge0,\ \rho_0\in L^1([0,a_0),r^2dr)\cap L^{\infty}([0,a_0),r^2dr),\ \\&(\rho_0)^k\in H^1([0,a_0),r^2dr)\ (0<k\le\gamma-\frac{1}{2}),\ u_0\in H^1([0,a_0),r^2dr),\\
    &\rho_0(r)>0\text{ for }r\in[0,a_0),\ \rho_0(a_0)=0,\ \partial_ru_0+\frac{2u_0(a_0)}{a_0}=0.
    \end{aligned}
\end{equation}
and Assumption \ref{cond11}. Then there exists a global weak solution $(\rho(r,t),u(r,t),a(t))$ to (\ref{nsp3})-(\ref{boundary conditions}) for $t\in[0,T]$, which satisfies $\rho(r,t)\ge0\ a.e.$ and 
\begin{equation*}
    0\le E(\rho,u)\le E_0:=E(\rho_0,u_0).
\end{equation*}
\end{theorem}
 
 The following theorem gives the higher-order regularities of the solutions. 
\begin{theorem}[Higher order regularities]\label{Thm2}
    Suppose $(\rho,u,a)$ is a weak solution to (\ref{nsp3})-(\ref{boundary conditions}) with initial data $(\rho_0,u_0)$ satisfying the conditions in Theorem \ref{Thm1: ExistenceOfGlobalWkSoln}.
    
    I. (Interior regularities) Assume further that the initial velocity $u_0\in H^2([r_0,a_0))$, then for any $r_{x_0}<r_{x_1}<r_{x_2}<a_0$,
    \begin{equation}\label{Thm2: 1}
        \begin{aligned}
            &(\rho,u)\in C([r_{x_1}(t),r_{x_2}(t)]\times[0,T]);\\
            &\rho\in\Linf(0,T;H^1([r_{x_1}(t),r_{x_2}(t)])),\ u\in\Linf(0,T;H^2([r_{x_1}(t),r_{x_2}(t)]));\\
            &\pt\rho\in\Linf(0,T;L^2([r_{x_1}(t),r_{x_2}(t)]))\cap L^2(0,T;H^1([r_{x_1}(t),r_{x_2}(t)]));\\
            &\pt u\in\Linf(0,T;L^2([r_{x_1}(t),r_{x_2}(t)]))\cap L^2(0,T;H^1([r_{x_1}(t),r_{x_2}(t)])),
        \end{aligned}
    \end{equation}
    where $x_i=M-\int_{r_{x_i}}^{a_0}\rho_04\pi r^2dr$.

    II. (Boundary regularities) For near the free boundary $r=a(t)$, let $\delta>0$ be small enough and $\Omega_{\delta}=(a(t)-\delta,a(t))$, we have
    \begin{equation}\label{Thm2: 2}
    \begin{aligned}
         &\sup_{0\le t\le T}\left(\|\rho^k\|_{H^1(\Omega_{\delta})}+\|u\|_{H^1(\Omega_{\delta})}+\|\sigma\|_{L^2(\Omega_{\delta})}\right)\\
         +&\|\sqrt{\rho}(\pt u+u\pr u)\|_{L^2(0,T;L^2(\Omega_{\delta}))}+\|\sigma\|_{L^2(0,T;H^1(\Omega_{\delta}))}+\|u\|_{L^2(0,T;H^2(\Omega_{\delta}))}+\|a\|_{H^1([0,T])}\le C_1
    \end{aligned}
    \end{equation}
    where $\sigma$ is defined in (\ref{boundary conditions}), and $C_{1}>0$ is a constant depends on
$$\|\rho _{0}\|_{L^{\infty}([0,a_{0}))},\ \|\rho _{0}^{k}\|_{H^{1}([0,a_{0}))},
\ \|u_{0}\|_{H^{1}([0,a_{0}))}.$$

    Furthermore, if $(\rho_0,u_0)$ satisfies $u_0\in H^2([a_0-2\delta,a_0)),\ \rho_0^{-\frac{1}{2}}\pr^2u_0\in L^2([a_0-2\delta,a_0))$, then
    \begin{equation}\label{Thm2: 3}
        \begin{aligned}
            &\sup_{0\le t\le T}\left(\|\sqrt{\rho}(\pt u+u\pr u)\|_{L^2(\Omega_{\delta})}+\|u\|_{H^2(\Omega_{\delta})}+\|\rho^{-\frac{1}{2}}\pr^2u\|_{L^2({\Omega_{\delta}})}+\|\sigma\|_{H^1(\Omega_{\delta})}\right)\\
            +&\|\pr(\pt u+u\pr u)\|_{L^2(0,T;L^2(\Omega_{\delta}))}+\|a\|_{H^2([0,T])}\le C_2,
        \end{aligned}
    \end{equation}
    where $C_2>0$ depends on $\|\rho_0\|_{L^{\infty}([0,a_0))},\ \|\rho_0^k\|_{H^1([0,a_0))},\ \|u_0\|_{H^1([0,a_0))},\ \|\rho_0^{-\frac{1}{2}}\pr^2u_0\|_{L^2(\Omega_{2\delta})}$.
\end{theorem}

The proofs of Theorem \ref{Thm1: ExistenceOfGlobalWkSoln} and \ref{Thm2} rely on an approximate problem and several key uniform estimates. The approximate problem is that we restrict the spatial domain to $[\xi,a(t))$ for any small $\xi>0$, and a similar method as in \cite{G,CK} ensures the existence of solutions to the approximate problem. The uniform estimates we established allow us to take a subsequence convergent to a solution to (\ref{nsp3})-(\ref{boundary conditions}) when $\xi\rightarrow0$. The most standard one is the basic energy estimate. Similar to the approach in \cite{CCL}, we add Assumption \ref{cond11} to the initial data so that the negative gravitational energy can be absorbed by part of the internal energy. The higher space-time integrability of the density ensures the strong convergence of $\rhox$ near the origin, where we adapt the ideas used in \cite{JZ} to deal with compressible Navier-Stokes equations on $\mathbb{R}^n$. As for the region near the free boundary, we are able to obtain some estimates for the derivatives of $(\rho,u)$. But due to the singularity of the origin, based on our techniques, we do not obtain uniform estimates up to the origin. With the higher-order estimates near the free boundary, the strong convergence of $(\rhox,\ux)$ for the region near the free boundary follows from Arzela-Ascoli's theorem. Combining both regions of the domain, we get a solution to (\ref{nsp3})-(\ref{boundary conditions}).

The following theorem shows the algebraic expanding rate of the free boundary, which illustrates the strong instability of the Lane-Emden solutions for Navier-Stokes-Poisson equations. We remark that the strong instability is conditional because the existence of strong solutions (i.e., assuming all derivatives appearing in the estimates are continuous) still remains to be justified.
\begin{theorem}[Algebraic expanding rate of the free boundary]\label{Thm3}Let $T>0$ and $\gamma\in(\frac{6}{5},\frac{4}{3}]$. Suppose $(\rho,u,a)$ is a radially symmetric strong solution to (\ref{nsp3})-(\ref{boundary conditions}) with initial data $(\rho_0,u_0)$ satisfying (\ref{conditionsI}) and Assumption \ref{cond11}. Then
\begin{equation}\label{AlgExpandingRates}
    \begin{aligned}
         \left(\frac{1+t}{\eta+\frac{4}{3}\varepsilon}\right)^{\frac{1}{4}}&\lesssim a(t)\lesssim\left(\frac{t}{\eta+\frac{4}{3}\varepsilon}\right)^{\frac{1}{3}}\text{ for }\gamma=\frac{4}{3};\\
         a(t)\approx&\left(\frac{t}{\eta+\frac{4}{3}\varepsilon}\right)^{\frac{1}{3}}\text{ for }\gamma\in(\frac{6}{5},\frac{4}{3}).
    \end{aligned}
\end{equation}
    
\end{theorem}
To prove Theorem \ref{Thm3}, we use a Virial type argument. That is, we assume the existence of smooth radially symmetric solution and consider the quantity 
$$H(t)=\frac{1}{2}\int_0^{a(t)}4\pi\rho r^2\cdot  r^2dr\text{ for }\gamma\in(\frac{6}{5},\frac{4}{3})$$
or
$$I(t)=\frac{1}{2}\int_0^{a(t)}4\pi\rho(r-(1+t)u)^2 r^2dr+3(1+t)^2\int_0^{a(t)}4\pi\rho^{\frac{4}{3}}r^2dr\text{ for }\gamma=\frac{4}{3}.$$
By analyzing the first and second time derivatives of the functionals $H(t)$ and $I(t)$, we are able to obtain the algebraic expanding rate of the free boundary.
\begin{rmk}
    For $\gamma=\frac{4}{3}$, the Goldreich-Weber solution is a solution to (\ref{nsp3}) with $\eta=0$ and boundary condition $\rho^{\frac{4}{3}}-\frac{4}{3}\varepsilon(\pr u-\frac{u}{r})|_{r=a(t)}=0$, which is referred to as the continuity of normal stress (see, for example, in \cite{JT,LXZ1}) and different from the boundary condition in (\ref{boundary conditions}), and the solution admits a different expanding rate of support than (\ref{AlgExpandingRates}).
\end{rmk}
\begin{rmk}\label{DenDependVisc}
    We consider the following radially symmetric 3D Navier-Stokes-Poisson equations
\begin{equation*}
    \begin{cases}
        \partial_t\rho+\partial_r(\rho u)+\frac{2\rho u}{r}=0,\\
        \partial_t(\rho u)+\partial_r(\rho u^2+\rho^{\gamma})+\frac{2\rho u^2}{r}=\partial_r((\eta+\frac{4}{3}\varepsilon))\rho^{\alpha}(\partial_r u+\frac{2u}{r}))-2\varepsilon\pr\rho^{\alpha}\cdot\frac{2u}{r}-\frac{4\pi\rho}{r^2}\int_0^r\rho s^2 ds
    \end{cases}
\end{equation*}
with $0\le\alpha\le\gamma$, $\gamma\in(\frac{6}{5},\frac{4}{3})$, i.e., taking $\varepsilon(\rho)=\varepsilon\rho^{\alpha},\ \eta(\rho)=\eta\rho^{\alpha}$ with at least one of $\varepsilon,\ \eta$ strictly positive and another non-negative, and the boundary conditions
\begin{equation*}
    \sigma(a(t),t)=0,\ \rho(a(t),t)=0,\ u(0,t)=0,
\end{equation*}
where $\sigma:=\rho^{\gamma}-(\eta+\frac{4}{3}\varepsilon)\rho^{\alpha}(\partial_ru+\frac{2u}{r})$. Assume the initial data $(\rho_0,u_0)$ satisfies (\ref{conditionsI}) and $(\rho_0,u_0)\in\mathcal{I}$. If $(\rho,u,a)$ is a strong solution, we have a similar basic energy estimate
\begin{equation*}
    \begin{aligned}
        \dt E(t)&=\dt\left(\frac{1}{2}\int_0^{a(t)}4\pi\rho u^2 r^2dr+\frac{1}{\gamma-1}\int_0^{a(t)}4\pi\rho^{\gamma}r^2dr-4\pi\int_0^{a(t)}4\pi\rho r\left(\int_0^r\rho s^2ds\right)dr\right)\\
        &=-4\pi\int_0^{a(t)}\eta\rho^{\alpha}(\partial_ru+\frac{2u}{r})^2 r^2+\frac{4}{3}\varepsilon\rho^{\alpha}(\partial_ru-\frac{u}{r})^2 r^2dr.
    \end{aligned}
\end{equation*}
By a similar argument, we can prove
\begin{equation*}
\begin{aligned}
    (\frac{t}{\eta})^{\frac{1}{3(1-\alpha)}}\lesssim a(t)\text{ for }0\le\alpha<\frac{2}{3};\\
    t\lesssim a(t)\text{ for }\frac{2}{3}\le\alpha\le\gamma
\end{aligned}
\end{equation*}
when $\eta\not=0$, and
\begin{equation*}
    t\lesssim a(t)
\end{equation*}
when $\eta=0$. We can see that the expanding rate has a dependency on the coefficient of the bulk viscosity. Especially when $\alpha>0$ is small, whether or not the bulk viscosity vanishes changes the expanding rate. The proof is given in Subsection \ref{Proof of expanding rate for density-dependent viscosity}.

However, while there are many results on the global existence of weak solutions
for the free boundary compressible Navier-Stokes equations or the NSP equations
with density-dependent viscosity \cite{JXZ,GLX,DL}, the density changes
discontinuously on the vacuum free boundary in the solutions. The global
existence of weak solutions satisfying $\rho (a(t),t)=0$ still remains open. Meanwhile, the recent work of Chen, Zhang and Zhu \cite{CZZ2026} proved the global well-posedness of classical solutions to the vacuum free boundary problem for the degenerate compressible Navier-Stokes equations with $\rho (a(t),t)=0$ and with $\gamma\in(\frac{4}{3},\infty)$ for 2D and $\gamma\in(\frac{4}{3},3)$ in 3D for large spherically symmetric data.

\end{rmk}

The rest part is arranged as follows. In section \ref{Variational methods}, we summarize the results of variational methods from \cite{CCL}, which provide a control on the negative part of the basic energy, i.e., the gravitational energy, for $\gamma\in(\frac{6}{5},\frac{4}{3}]$. We also address the differences of the allowed initial data sets in \cite{KL} and \cite{CCL}, that is, the latter one allows a family of initial data near the Lane-Emden solution that are not covered by the former one. In Section \ref{Existence of global weak solutions}, we focus on the existence of global weak solutions. We first establish the basic energy estimate. Then we split the spatial domain into two parts. The first part is near the origin, while the second part is near the free boundary. We establish several uniform higher-order estimates to the approximate solutions in both parts of the domains, which are crucial in the compactness arguments. Then, we take a subsequence limit to get a global solution. In Section \ref{Expanding rate of the free boundary}, we assume the long time existence of strong solutions and establish the algebraic expanding rate of the free boundary.

   \section{Variational methods}
\label{Variational methods}
In this section, we first summarize the results from \cite{CCL}, then we compare the differences between the initial data sets in \cite{CCL} and \cite{KL}.

The following lemma summarizes the results from \cite{CCL}.
\begin{lem}\cite[Theorem 3.1, 4.2]{CCL}\label{LinLemma}
    Suppose $\rho\in L^1(\mathbb{R}^3)\cap L^{\frac{4}{3}}(\mathbb{R}^3)$, then the best constant $C_{min}$ in 
\begin{equation}            
\iint_{\mathbb{R}^3\times\mathbb{R}^3}\frac{\rho(\boldsymbol{x})\rho(\boldsymbol{y})}{|\boldsymbol{x}-\boldsymbol{y}|}\dbx\dby\le C_{min}\|\rho\|^{\frac{2}{3}}_{L^1(\mathbb{R}^3)}\|\rho\|^{\frac{4}{3}}_{L^{\frac{4}{3}}(\mathbb{R}^3)}
\end{equation}
is attained by the Lane-Emden solution $\rho_{\mu}$ for some $\mu>0$, and $ C_{min}=6M_{ch}^{-\frac{2}{3}}$, where $M_{ch}$ is the mass of $\rho_{\mu}$. 

Suppose $\rho\in L^1(\mathbb{R}^3)\cap L^{\gamma}(\mathbb{R}^3)$, where $\gamma\in(\frac{6}{5},\frac{4}{3})$, suppose $(\rho_{\mu},0)$ is a Lane-Emden solution with central density $\mu>0$, then the constraint minimizing problem 
\begin{equation}
    l_{\mu}=\inf_{\rho\in\mathcal{K}}S_{\mu}(\rho),
\end{equation}
where $S_{\mu}$ is defined as 
\begin{equation*}
     S_{\mu}(\rho)=\frac{1}{\gamma-1}\int_{\mathbb{R}^3}\rho^{\gamma} \dbx-\frac{1}{2}\iint_{\mathbb{R}^3\times\mathbb{R}^3}\frac{\rho(\boldsymbol{x})\rho(\boldsymbol{y})}{|\boldsymbol{x}-\boldsymbol{y}|}\dbx\dby-\varPhi(R_{\mu})\int_{\mathbb{R}^3}\rho  \dbx
\end{equation*}
and $\mathcal{K}:=\{\rho\ge0,\rho\not\equiv0:Q(\rho)=0\}$, is attained by $\rho_{\mu}$ (up to a translation).
\end{lem}
We remark that when written in 3D Cartesian coordinates, the gravitational energy is
$$-4\pi\int_0^{a(t)}4\pi\rho r\int_0^r\rho s^2dsdr=-\frac{1}{2}\iint_{\mathbb{R}^3\times\mathbb{R}^3}\frac{\rho(\boldsymbol{x})\rho(\boldsymbol{y})}{|\boldsymbol{x}-\boldsymbol{y}|}\dbx\dby=-\frac{1}{8\pi}\intr|\nabla\varPhi|^2\dbx.$$
Then the following corollary, which is similar to Lemma 4.4, 4.6, 4.7 in \cite{CCL},  shows how Lemma \ref{LinLemma} can be applied to control the negative part of the basic energy.
\begin{cor}\label{cor1}
    For $\gamma=\frac{4}{3}$, suppose that there exists a strong solution $(\rho,\u)=(\rho(|\bx|,t),u(|\bx|,t)\frac{\bx}{|\bx|})$ to (\ref{nsp3})-(\ref{boundary conditions}) with compactly supported radially symmetric initial data $(\rho_0,\u_0)$ such that $\rho_0\in L^1(\mathbb{R}^3)\cap L^{\frac{4}{3}}(\mathbb{R}^3)$, $\sqrt{\rho_0}\u_0\in L^2(\mathbb{R}^3)$ and $M(\rho_0)<M_{ch}$. Let $B(0;a(t))$ be the support of $\rho(\bx,t)$ at time $t\ge0$.  Then
\begin{equation}
    3\intb\rho^{\frac{4}{3}}\dbx-\frac{1}{2}\iint_{\mathbb{R}^3\times\mathbb{R}^3}\frac{\rho(\boldsymbol{x})\rho(\boldsymbol{y})}{|\boldsymbol{x}-\boldsymbol{y}|}\dbx\dby\ge3\left(1-(\frac{M}{M_{ch}})^{\frac{2}{3}}\right)\intb\rho^{\frac{4}{3}}\dbx,
\end{equation}
and
\begin{equation}
    \intb\rho^{\frac{4}{3}}\dbx\le CE_0
\end{equation}
for some constant $C>0$ 

For $\gamma\in(\frac{6}{5},\frac{4}{3})$, suppose the initial data $(\rho_0,\u_0)$ is compactly supported radially symmetric, $\rho_0\in L^1(\mathbb{R}^3)\cap L^{\gamma}(\mathbb{R}^3)$, $\sqrt{\rho_0}\u_0\in L^2(\mathbb{R}^3)$ and $(\rho_0,\u_0)\in\mathcal{I}$, where $\mathcal{I}$ is defined by (\ref{Linmass}), then there exist a $\mu>0$ such that 
$$(\rho_0,\u_0)\in\mathcal{I}_{\mu}:=\{(\rho,\u):Q(\rho)>0,E(\rho,\u)-\varPhi_{\mu}(R_{\mu})M(\rho)<l_{\mu}\}.$$ 
What's more, if there exists a strong solution $(\rho,\u)$ to (\ref{nsp3})-(\ref{boundary conditions}) with initial data $(\rho_0,\u_0)$, then 
$$(\rho(t),\u(t))\in\mathcal{I}_{\mu}$$
for any $t\ge0$, and there exists a constant $\Lambda>0$ such that
\begin{equation}
    Q(\rho(t))>\Lambda.
\end{equation}
\end{cor}
 
 For $\gamma\in(\frac{6}{5},\frac{4}{3})$, we claim that the initial data set $\mathcal{I}$ includes some initial data $(\rho_0,\u_0)$ that are near the Lane-Emden solution, which are not allowed in the initial data set in \cite{KL}. In fact, the approach in \cite{KL} require $M(\rho_0)<M_c$, where
\begin{equation}\label{ZAMPmass}
    M<M_c=\left(\frac{4-3\gamma}{\gamma-1}(\frac{B}{3})^{-\frac{3(\gamma-1)}{4-3\gamma}}\right)^{\frac{4-3\gamma}{5\gamma-6}}(\frac{\tilde{E_0}}{4\pi})^{-\frac{4-3\gamma}{5\gamma-6}}
\end{equation}
with
\begin{equation*}
    \tilde{E_0}=4\pi\int_0^{a_0}\frac{1}{2}\rho_0u_0^2 r^2+\frac{1}{\gamma-1}\rho_0^{\gamma} r^2dr,
\end{equation*}
which leads to
\begin{equation}\label{KLcondition}
\begin{aligned}
    3\int_0^{a_0}\rho^{\gamma}_0r^2dr-4\pi\int_0^{a_0}\rho_0 r\int_0^r\rho_0 s^2dsdr
        \ge C_0M^{\frac{5\gamma-6}{3(\gamma-1)}}\left(\int_0^{a_0}\rho^{\gamma}_0r^2dr\right)^{\frac{1}{3(\gamma-1)}}
\end{aligned}
\end{equation}
for some positive constant $C_0$. We justify (\ref{KLcondition}) in Appendix \ref{Derivation of (KLcondition)}.

From the expressions of (\ref{Linmass}) and (\ref{ZAMPmass}), we can see that for any initial data $(\rho_0,u_0)$ with $\frac{1}{2}\int_0^{a_0}\rho_0u_0^2r^2dr\gg\frac{1}{\gamma-1}\int_0^{a_0}\rho_0^{\gamma}r^2dr$, both upper bounds of mass are small, having the same order.

For initial data with small kinetic energy, we claim that the set (\ref{Linmass}) covers some initial data that are not included in the set (\ref{ZAMPmass}). In fact, suppose $\rho_{\mu}$ is a Lane-Emden solution with central density $\mu>0$. Let $\rho_{\lambda,\mu}=\lambda^3\rho_{\mu}(\lambda r)$, which is the mass preserving scaling. Since
\begin{equation*}
    \frac{d}{d\lambda}S_{\mu}(\rho_{\lambda,\mu})=3\lambda^{3\gamma-4}\intr\rho^{\gamma}_{\mu}\dbx-\frac{1}{2}\iint_{\mathbb{R}^3\times\mathbb{R}^3}\frac{\rho_{\mu}(\bx)\rho_{\mu}(\boldsymbol{y})}{|\boldsymbol{x}-\boldsymbol{y}|}\dbx\dby=\frac{1}{\lambda}Q(\rho_{\lambda,\mu}),
\end{equation*}
and we have $\frac{1}{\lambda}Q(\rho_{\lambda,\mu})>0$ when $\lambda<1$, $\frac{1}{\lambda}Q(\rho_{\lambda,\mu})<0$ when $\lambda>1$. So for $\lambda<1$,
\begin{equation}\label{mono}
    S_{\mu}(\rho_{\lambda,\mu})<S_{\mu}(\rho_{\mu})=l_{\mu}.
\end{equation}
Combine with the fact that $Q(\rho_{\mu})=0$, we have
\begin{equation}
    0<Q(\rho_{\lambda,\mu})<4\pi\frac{4-3\gamma}{\gamma-1}(1-\lambda^{3\gamma-3})\int_0^{R_{\mu}}\rho_{\mu}^{\gamma}r^2dr.
\end{equation}
When $\lambda<1$ is closed to 1, $Q(\rho_{\lambda,\mu})$ can be arbitrarily small. Since the support of $\rho_{\lambda,\mu}$ is $[0,\frac{R_{\mu}}{\lambda}]$, we choose $u_0=u^{\lambda}(r)$ supported on $[0,\frac{R_{\mu}}{\lambda}]$ small enough such that
\begin{equation}
    \begin{aligned}
        E(\rho_{\lambda,\mu},u_0)-\varPhi_{\mu}(R_{\mu})M(\rho_{\lambda,\mu})&=4\pi\int_0^{\frac{R_{\mu}}{\lambda}}\frac{1}{2}\rho_{\lambda,\mu}u_0^2r^2dr+S_{\mu}(\rho_{\lambda,\mu})<S_{\mu}(\rho_{\mu})=l_{\mu},
    \end{aligned}
\end{equation}
we have $(\rho_{\lambda,\mu},u_0)\in\mathcal{I}$. Notice that $(\rho_{\lambda,\mu},u_0)$ doesn't satisfy (\ref{KLcondition}), thus it is not covered by the initial data set required by (\ref{ZAMPmass}).

   \section{Existence of global weak solutions}
\label{Existence of global weak solutions}
In this section, we focus on the construction of radially symmetric global weak solutions to PDEs (\ref{nsp1}). We first consider the following approximate problem:
\begin{equation}
    \begin{cases}\label{apprE1}
        \partial_t\rho+\frac{1}{r^2}\partial_r(r^2\rho u)=0,\\
        \rho(\partial_t u+u\partial_ru)+\partial_r\rho^{\gamma}=-\frac{4\pi\rho}{r^2}\int_0^r\rho(s,t)s^2ds+\partial_r\left(\frac{\frac{4}{3}\varepsilon+\eta}{r^2}\partial_r(r^2u)\right),
    \end{cases}
\end{equation}
for $(r,t)\in\{(r,t):\xi\le r\le a(t),0\le t\le T\}$ with the initial data and boundary condition for any fixed small $\xi>0$,
\begin{equation}\label{apprE2}
    (\rho^{\xi},u^{\xi})(r,0)=(\rho^{\xi}_0,u^{\xi}_0),\ \xi\le r\le a_0,
\end{equation}
\begin{equation}\label{apprE3}
    u^{\xi}(\xi,t)=0,\ \sigma^\xi(a(t),t)=\left((\rho^{\xi})^{\gamma}-(\eta+\frac{4}{3}\varepsilon)(\partial_ru^{\xi}+\frac{2u^{\xi}}{r})\right)(a(t),t)=0,\ t>0,
\end{equation}
\begin{equation}\label{apprE4}
    (a^{\xi})'(t)=u^{\xi}(a^{\xi}(t),t),\ a^{\xi}(0)=a_0.
\end{equation}
where
\begin{equation}\label{apprEcond1}
    \begin{aligned}
    \inf_{r\in[\xi,a_0)}\rhox_0(r)>0,\ &\ux_0(\xi)=0,\\
    \rhox_0\in L^1([\xi,a_0),r^2dr)&\cap L^{\gamma}([\xi,a_0),r^2dr),
    \\(\rhox_0)^k\in H^1([\xi,a_0),r^2dr)\ (0<k\le\gamma&-\frac{1}{2}),\ \ux_0\in H^1([\xi,a_0),r^2dr),
    \end{aligned}
    \end{equation}
\begin{equation}\label{apprEcond2}
\left((\rhox_0)^{\gamma}-(\eta+\frac{4}{3}\varepsilon)(\partial_r\ux_0+\frac{2\ux_0}{r})\right)(a_0)=0.
\end{equation}

If written in Lagrangian coordinates, the problem is changed to:

\begin{equation}\label{apprL1}
    \begin{cases}
        \partial_{\tau}\rho+4\pi\rho^2\partial_x(r^2u)=0,\\
        \partial_{\tau}u+4\pi r^2\partial_x\rho^{\gamma}+\frac{x}{r^2}=16\pi^2r^2\partial_x\left((\frac{4\varepsilon}{3}+\eta)\rho\partial_x(r^2u)\right),
    \end{cases}
\end{equation}
with initial data and boundary conditions
\begin{equation}\label{apprL2}
    (\rho^{\xi},u^{\xi})(x,0)=(\rho_0^{\xi},u_0^{\xi}),\ 0\le x\le M,
\end{equation}
\begin{equation}\label{apprL3}
    u^{\xi}(0,\tau)=0,\ \sigma^\xi(M,\tau)=\left((\rho^{\xi})^{\gamma}-(\eta+\frac{4}{3}\varepsilon)(4\pi r^2\rho^{\xi}\partial_xu^{\xi}+\frac{2u^{\xi}}{r})\right)(M,\tau)=0,
\end{equation}
\begin{equation}\label{apprL4}
    (a^{\xi})'(\tau)=u^{\xi}(M,\tau),\ a^{\xi}(0)=a_0, 
\end{equation}
with
\begin{equation}\label{apprLcond1}
    \begin{aligned}
    \inf_{x\in[0,M]}\rhox_0(x)>0,\ &\ux_0(0)=0,\\
        \rhox_0\in L^{\gamma-1}&([0,M]),
    \\(\rhox_0)^q\in H^1([0,M])\ (\frac{1}{2}<q\le\gamma)&,\ \sqrt{\rhox_0} r^2\partial_x\ux_0,\ \frac{\ux_0}{\sqrt{\rhox_0}r}\in L^2([0,M]),
    \end{aligned}
\end{equation}
\begin{equation}\label{apprLcond2}
    \left((\rhox_0)^{\gamma}-(\eta+\frac{4}{3}\varepsilon)(4\pi r^2\rhox_0\partial_x\ux_0+\frac{2\ux_0}{r})\right)(M)=0.
\end{equation}
where
\begin{equation}
    \begin{cases}
        x(r,t)=4\pi\int_{\xi}^r\rho(s,t)s^2ds=M-4\pi\int_r^{a(t)}\rho(s,t)s^2ds,\\
        \tau(r,t)=t.
    \end{cases}
\end{equation}

\subsection{A priori estimates}\label{A priori estimates}
Assume $T>0$ and $\gamma\in(\frac{6}{5},\frac{4}{3}]$. Let $(\rhox,\ux,a^{\xi})$ be any strong solution to problem (\ref{apprE1})-(\ref{apprE4}), with initial data satisfying (\ref{cond1}) and (\ref{apprEcond1})-(\ref{apprEcond2}). For convenience, we drop the superscript $\xi$ in this subsection.

\subsubsection{Basic energy estimate}
\label{Basic energy estimate}

%%%%%%%%%%%% Basic energy estimate, remove % to get all details %%%%%%%%%%%

\begin{lem}\label{BasicEnergyEst}
Let $T>0$, $\gamma\in(\frac{6}{5},\frac{4}{3}]$, and $(\rho,u,a)$ is a strong solution to problem (\ref{apprE1})-(\ref{apprE4}) with initial data $(\rho_0,u_0)$ satisfying $\rho_0\in L^1([\xi,a_0);r^2dr)\cap L^{\gamma}([\xi,a_0);r^2dr)$, $\sqrt{\rho_0}u_0\in L^2([\xi,a_0);r^2dr)$ and (\ref{cond1}). Let
        \begin{equation}
            E(\rho(t),u(t))=4\pi\int_{\xi}^{a(t)}\frac{1}{2}\rho u^2 r^2+\frac{1}{\gamma-1}\rho^{\gamma} r^2dr-4\pi\int_{\xi}^{a(t)}4\pi\rho r\left(\int_{\xi}^r\rho s^2ds\right)dr,
        \end{equation}
    then for any $t\in[0,T]$, there holds:
        \begin{equation}\label{EnInq}
             E(\rho(t),u(t))+4\pi\int_0^t\int_{\xi}^{a(s)}(\eta+\frac{4}{3}\varepsilon)(\partial_ru+\frac{2u}{r})^2 r^2drds=E(\rho_0,u_0).
        \end{equation}
For the gravitational energy term, it can also be written as follows
\begin{equation}\label{GId}
    4\pi\int_{\xi}^{a(t)}4\pi\rho r\left(\int_{\xi}^r\rho s^2ds\right)dr=\frac{1}{2a(t)}\left(4\pi\int_{\xi}^{a(t)}\rho r^2dr\right)^2+\int_{\xi}^{a(t)}\frac{1}{2r^2}\left(4\pi\int_{\xi}^r\rho s^2ds\right)^2dr.
\end{equation}
\end{lem}
\begin{proof}
    Multiply the momentum equation in (\ref{apprE1}) by $ur^2$ and integrate it over $[\xi,a(t))$ to get
    \begin{equation}\label{BEterm1}
    \begin{aligned}
        \int_{\xi}^{a(t)}\rho(\partial_t u+u\partial_ru)ur^2dr+&\int_{\xi}^{a(t)}\partial_r\rho^{\gamma}\cdot ur^2dr\\
        &=-\int_{\xi}^{a(t)}4\pi\rho u\int_{\xi}^r\rho s^2dsdr+\int_{\xi}^{a(t)}ur^2\partial_r\left(\frac{\frac{4}{3}\varepsilon+\eta}{r^2}\partial_r(r^2u)\right)dr.
    \end{aligned}
    \end{equation}
    For the first term on the left-hand side of (\ref{BEterm1}), we have
    \begin{equation}
        \begin{aligned}
             &\int_{\xi}^{a(t)}\rho(\partial_t u+u\partial_ru)ur^2dr\\
             %=&\int_{\xi}^{a(t)}\frac{1}{2}\partial_t(\rho u^2)r^2-\frac{1}{2}u^2r^2\partial_t\rho+\rho u^2\partial_ru\cdot r^2dr\\
             =&\dt\int_{\xi}^{a(t)}\frac{1}{2}\rho u^2r^2dr-\frac{1}{2}\rho u^3r^2|_{r=a(t)}-\int_{\xi}^{a(t)}\frac{1}{2}u^2r^2\partial_t\rho dr+\int_{\xi}^{a(t)}\rho u^2\partial_ru\cdot r^2dr.
        \end{aligned}
    \end{equation}
    Using the mass equation in (\ref{apprE1}), we have
    \begin{equation}
        \begin{aligned}
            -\int_{\xi}^{a(t)}\frac{1}{2}u^2r^2\partial_t\rho dr&=\int_{\xi}^{a(t)}\frac{1}{2}\partial_r\rho\cdot u^3r^2dr+\int_{\xi}^{a(t)}\frac{1}{2}u^2\rho\cdot(r^2\partial_ru+2ru)dr=I+II.
        \end{aligned}
    \end{equation}
    By integration by part,
    \begin{equation}
    \begin{aligned}
        II%=&-\int_{\xi}^{a(t)}\frac{1}{2}\partial_r(\rho u^2r^2)udr+\frac{1}{2}\rho u^3r^2|_{r=a(t)}+\int_{\xi}^{a(t)}\rho u^3rdr\\
        =&-\int_{\xi}^{a(t)}\frac{1}{2}\partial_r\rho\cdot u^3r^2+\rho u^2\partial_ru\cdot r^2dr+\frac{1}{2}\rho u^3r^2|_{r=a(t)}.
    \end{aligned}
    \end{equation}
    As a result,
    \begin{equation}\label{Left1}
        \int_{\xi}^{a(t)}\rho(\partial_t u+u\partial_ru)ur^2dr=\dt\int_{\xi}^{a(t)}\frac{1}{2}\rho u^2r^2dr.
    \end{equation}
    
     For the second term on the left-hand side of (\ref{BEterm1}), integrating by part and using the mass equation in (\ref{apprE1}), we have
     \begin{equation}\label{Left2}
     \begin{aligned}
          \int_{\xi}^{a(t)}\partial_r\rho^{\gamma}\cdot ur^2dr%=&\frac{\gamma}{\gamma-1}\int_{\xi}^{a(t)}r^2\rho^{\gamma-1}\partial_t\rho dr+\frac{\gamma}{\gamma-1}\rho^{\gamma}ur^2|_{r=a(t)}\\
          =&\dt\int_{\xi}^{a(t)}\frac{1}{\gamma-1}\rho^{\gamma}r^2dr+\rho^{\gamma}ur^2|_{r=a(t)}.
     \end{aligned}
     \end{equation}
     
    For the first term on the right-hand side of (\ref{BEterm1}), by using the mass equation, we have
    \begin{equation}
    \begin{aligned}
        \int_{\xi}^{a(t)}-4\pi\rho u\left(\int_{\xi}^r\rho s^2ds\right)dr=&\int_{\xi}^{a(t)}2\pi(r\partial_t\rho+r\partial_r(\rho u))\left(\int_{\xi}^r\rho s^2ds\right)dr\\
        =&\dt\int_{\xi}^{a(t)}2\pi\rho r\left(\int_{\xi}^r\rho s^2ds\right)dr-\int_{\xi}^{a(t)}2\pi\rho u\left(\int_{\xi}^r\rho s^2ds\right)dr,
    \end{aligned}
    \end{equation}
    which gives
    \begin{equation}\label{right1}
        \begin{aligned}
            \int_{\xi}^{a(t)}-4\pi\rho u\left(\int_{\xi}^r\rho s^2ds\right)dr=\dt\int_{\xi}^{a(t)}4\pi\rho r\left(\int_{\xi}^r\rho s^2ds\right)dr.
        \end{aligned}
    \end{equation}
    
    For the second term on the right-hand side of (\ref{BEterm1}), by integration by part,
    \begin{equation}\label{right2}
        \begin{aligned}
            \int_{\xi}^{a(t)}ur^2&\partial_r\left(\frac{\frac{4}{3}\varepsilon+\eta}{r^2}\partial_r(r^2u)\right)dr\\&=-\int_{\xi}^{a(t)}(\eta+\frac{4}{3}\varepsilon)(\partial_ru+\frac{2u}{r})^2r^2dr+(\frac{4}{3}\varepsilon+\eta)(\partial_ru+\frac{2u}{r})ur^2|_{r=a(t)}.
        \end{aligned}
    \end{equation}
    Combining (\ref{Left1}), (\ref{Left2}), (\ref{right1}), (\ref{right2}), we get

    \begin{equation}
    \begin{aligned}
        &\dt\left(4\pi\int_{\xi}^{a(t)}\frac{1}{2}\rho u^2  r^2+\frac{1}{\gamma-1}\rho^{\gamma}r^2dr-4\pi\int_{\xi}^{a(t)}4\pi\rho r\left(\int_{\xi}^r\rho s^2ds\right)dr\right)\\
        =&-4\pi\int_{\xi}^{a(t)}(\eta+\frac{4}{3}\varepsilon)(\partial_ru+\frac{2u}{r})^2r^2dr,
    \end{aligned}
    \end{equation}
    which gives (\ref{EnInq}).
    
    Also,
    \begin{equation}
        \begin{aligned}
            4\pi\int_{\xi}^{a(t)}\rho r\left(\int_{\xi}^r\rho s^2ds\right)dr=&4\pi\int_{\xi}^{a(t)}\frac{1}{2r}d\left[\left(\int_{\xi}^r\rho s^2ds\right)^2\right]\\
            =&\frac{2\pi}{a(t)}\left(\int_{\xi}^{a(t)}\rho s^2ds\right)^2+\int_{\xi}^{a(t)}\frac{2\pi}{r^2}\left(\int_{\xi}^r\rho s^2ds\right)^2dr,
        \end{aligned}
    \end{equation}
    which gives (\ref{GId}).

    What's more,
    \begin{equation}\label{BoundaryBdd}
        \begin{aligned}
            &4\pi\int_0^T\int_{\xi}^{a(t)}\pvcons(\partial_r u+\frac{2u}{r})^2r^2drdt\\
            =&4\pi\int_0^T\int_{\xi}^{a(t)}\pvcons(r^2(\partial_ru)^2+2u^2)drdt+8\pi\int_0^T\pvcons a(t)|a'(t)|^2dt\le E_0,
        \end{aligned}
    \end{equation}
    which gives $\int_0^Ta(t)|a'(t)|^2dt\le C$.
\end{proof}

\begin{rmk}
When written in 3D Cartesian coordinates, the energy reads:
\begin{equation}
    E(\rho(t),\u(t))=\int_{B(0;a(t))\backslash B(0;\xi)}\frac{1}{2}\rho|\boldsymbol{u}|^2+\frac{1}{\gamma-1}\rho^{\gamma}\dbx-\int_{\mathbb{R}^3}\frac{1}{8\pi}|\nabla\varPhi|^2\dbx,
\end{equation}
    and the energy identity is
        \begin{equation}
             E(\rho(t),\u(t))+\int_0^t\int_{B(0;a(s))\backslash B(0;\xi)}(\eta+\frac{4}{3}\varepsilon)|div\u|^2\dbx ds=E(\rho_0,\u_0).
        \end{equation}
When written in Lagrangian coordinates, the energy becomes
\begin{equation}
    E(\rho(\tau),u(\tau))=\int_0^M\frac{1}{2}u^2+\frac{1}{\gamma-1}\rho^{\gamma-1}-\frac{x}{r}dx,
\end{equation}
and the energy identity is
\begin{equation}\label{BasicEnergyEstLagrangian}
\begin{aligned}
    E(\rho(\tau),u(\tau))
    +\int_0^{\tau}\int_0^M(\eta+\frac{4}{3}\varepsilon)\frac{1}{\rho}(4\pi\rho r^2\partial_xu+\frac{2u}{r})^2dxds=E(\rho_0,u_0).
\end{aligned}
\end{equation}
\end{rmk}
\subsubsection{Further estimates in Eulerian coordinates}
\label{Estimates in Eulerian coordinates}
In this subsection, we establish several further estimates in Eulerian coordinates, which are necessary for the precompactness of the approximate solutions near the origin.

In the following lemma, we prove a higher space-time integrability for $\rho$.
\begin{lem}\label{rhoIntegrability}
    Let $T>0$, $\gamma\in(\frac{6}{5},\frac{4}{3}]$ and $(\rho,u,a)$ be any strong solution to the problem (\ref{apprE1})-(\ref{apprE4}) with initial data $(\rho_0,u_0)$ satisfying $\rho_0\in L^1([\xi,a_0);r^2dr)\cap L^{\gamma}([\xi,a_0);r^2dr)$, $\sqrt{\rho_0}u_0\in L^2([\xi,a_0);r^2dr)$ and (\ref{cond1}). Then there exists a positive constant $C=C(E_0,M,T)$ depending on $E_0,\ M$ and $T$, but independent of $\xi$ such that
    \begin{equation}
        \int_0^T\int_{\xi}^{a(t)}\rho^{2\gamma}r^8drdt\le C(E_0,M,T).
    \end{equation}
\end{lem}
\begin{proof}
    We rewrite the momentum equation in (\ref{apprE1}):
    $$\partial_t(\rho u)+\frac{2}{r}\rho u^2+\partial_r(\rho u^2+\rho^{\gamma})=-\frac{4\pi\rho}{r^2}\int_0^r\rho(s,t)s^2ds+\partial_r\left(\frac{\frac{4}{3}\varepsilon+\eta}{r^2}\partial_r(r^2u)\right).$$
    Change the variable $r$ into $y$, multiply the resulting equation by $y^3$ and integrate it from $r$ to $a(t)$, we get
    \begin{equation}
    \begin{aligned}
        &\int_r^{a(t)}\partial_t(\rho u)y^3 dy+\int_r^{a(t)}2\rho u^2y^2 dy+\int_r^{a(t)}y^3\partial_y(\rho u^2+\rho^{\gamma}) dy\\=&-\int_r^{a(t)}4\pi\rho y\int_{\xi}^y\rho(s,t)s^2 ds dy+\int_r^{a(t)}y^3\partial_y\left(\frac{\frac{4}{3}\varepsilon+\eta}{y^2}\partial_y(y^2u)\right) dy.
    \end{aligned}
    \end{equation}
    We integrate the terms $\int_r^{a(t)}y^3\partial_y(\rho u^2+\rho^{\gamma}) dy$ and $\int_r^{a(t)}y^3\partial_y\left(\frac{\frac{4}{3}\varepsilon+\eta}{y^2}\partial_y(y^2u)\right) dy$ by parts and multiply the resulting equation by $\rho^{\gamma}$ to obtain
    \begin{equation}\label{rhoIntegrability:1}
    \begin{aligned}
        \rho^{2\gamma}r^3=&\left(\frac{4}{3}\varepsilon+\eta\right)\rho^{\gamma}\left(\partial_ru+\frac{2u}{r}\right)r^3-\rho^{\gamma+1}u^2r^3+\rho^{\gamma}\partial_t\int_r^{a(t)}\rho uy^3 dy\\&+\rho^{\gamma}\int_r^{a(t)}2\rho u^2y^2 dy-3\rho^{\gamma}\int_r^{a(t)}(\rho u^2+\rho^{\gamma})y^2 dy\\&+3\rho^{\gamma}\int_r^{a(t)}\left(\frac{4}{3}\varepsilon+\eta\right)\left(\partial_yu+\frac{2u}{y}\right)y^2 dy+\rho^{\gamma}\int_r^{a(t)}4\pi\rho y\int_{\xi}^r\rho(s,t)s^2dsdy.
    \end{aligned}
    \end{equation}
    Using the mass conservation $\partial_t\rho^{\gamma}=-\partial_r(\rho^{\gamma}u)-(\gamma-1)\rho^{\gamma}\partial_ru-\gamma\rho^{\gamma}\frac{2u}{r}=0$, we can rewrite the following term in (\ref{rhoIntegrability:1}) as
    \begin{equation}
    \begin{aligned}
        \rho^{\gamma}\partial_t\int_r^{a(t)}\rho uy^3 dy=&\partial_t\left(\rho^{\gamma}\int_r^{a(t)}\rho uy^3 dy\right)-\pt\rho^{\gamma}\int_r^{a(t)}\rho uy^3 dy\\
        =&\partial_t\left(\rho^{\gamma}\int_r^{a(t)}\rho uy^3 dy\right)+\partial_r\left(\rho^{\gamma}u\int_r^{a(t)}\rho uy^3 dy\right)+\rho^{\gamma+1}u^2r^3
        \\&+(\gamma-1)\rho^{\gamma}\partial_ru\int_r^{a(t)}\rho uy^3 dy+\gamma\rho^{\gamma}\frac{2u}{r}\int_r^{a(t)}\rho uy^3 dy,
    \end{aligned}
    \end{equation}
    then (\ref{rhoIntegrability:1}) reads as
    \begin{equation}\label{rho2gamma}
        \begin{aligned}
        \rho^{2\gamma}r^3=&\left(\frac{4}{3}\varepsilon+\eta\right)\rho^{\gamma}\left(\partial_ru+\frac{2u}{r}\right)r^3+\partial_t\left(\rho^{\gamma}\int_r^{a(t)}\rho uy^3 dy\right)+\partial_r\left(\rho^{\gamma}u\int_r^{a(t)}\rho uy^3 dy\right)
        \\&+(\gamma-1)\rho^{\gamma}\partial_ru\int_r^{a(t)}\rho uy^3 dy+\gamma\rho^{\gamma}\frac{2u}{r}\int_r^{a(t)}\rho uy^3 dy\\&+\rho^{\gamma}\int_r^{a(t)}2\rho u^2y^2 dy-3\rho^{\gamma}\int_r^{a(t)}(\rho u^2+\rho^{\gamma})y^2 dy+\rho^{\gamma}\int_r^{a(t)}4\pi\rho y\int_{\xi}^r\rho(s,t)s^2dsdy\\&+3\rho^{\gamma}\int_r^{a(t)}\left(\frac{4}{3}\varepsilon+\eta\right)\left(\partial_yu+\frac{2u}{y}\right)y^2 dy.
    \end{aligned}
    \end{equation}
Multiply (\ref{rho2gamma}) by $r^5$ and integrate it over $[\xi,a(t))\times[0,T]$, we have
\begin{equation}\label{rhoIntegrability:2}
    \int_0^T\int_{\xi}^{a(t)}\rho^{2\gamma}r^8drdt=\int_0^T\int_{\xi}^{a(t)}\{\text{r.h.s. of (\ref{rho2gamma})}\}r^5drdt=\sum_{i=1}^{9}I_i.
\end{equation}
We estimate each term on the right hand side of (\ref{rhoIntegrability:2}). The controls of $I_1,\ I_4\text{ and }I_5$ are similar and can be estimated together as
\begin{equation}\label{I1}
    \begin{aligned}
        &|I_1+I_4+I_5|\\
        =&\left|\int_0^T\int_{\xi}^{a(t)}\left(\frac{4}{3}\varepsilon+\eta\right)\rho^{\gamma}\left(\partial_ru+\frac{2u}{r}\right)r^8+\left((\gamma-1)\rho^{\gamma}\partial_ru+\gamma\rho^{\gamma}\cdot\frac{2u}{r}\right)r^5\left(\int_r^{a(t)}\rho uy^3dy\right)drdt\right|\\
        \le&2\delta\int_0^T\int_{\xi}^{a(t)}\rho^{2\gamma}r^8drdt+C\delta^{-1}\int_0^T\int_{\xi}^{a(t)}\left(\partial_ru+\frac{2u}{r}\right)^2r^8drdt\\
        &+C\int_0^T\int_{\xi}^{a(t)}\rho^{\gamma}\left(\left|\partial_ru\right|+\left|\frac{2u}{r}\right|\right)r^5drdt\cdot\frac{1}{2}\sup_{0\le t\le T}\int_{\xi}^{a(t)}(\rho+\rho u^2)y^3 dy\\
        \le&2\delta\int_0^T\int_{\xi}^{a(t)}\rho^{2\gamma}r^8drdt+\delta^{-1}C.
    \end{aligned}
\end{equation}
For $I_3$, we integrate by parts with respect to $r$ and estimate similarly as in (\ref{I1}):
\begin{equation}
    \begin{aligned}\label{I3}
        |I_3|=&\left|\int_0^T\int_{\xi}^{a(t)}\partial_r\left(\rho^{\gamma}u\int_r^{a(t)}\rho uy^3 dy\right)r^5drdt\right|\\
        \le&\left|\int_0^T\int_{\xi}^{a(t)}\rho^{\gamma}u\left(\int_r^{a(t)}\rho uy^3 dy\right)5r^4 drdt\right|\\
        \le&\delta\int_0^T\int_{\xi}^{a(t)}\rho^{2\gamma}r^8drdt+\delta^{-1}C\int_0^T\int_{\xi}^{a(t)}u^2\left(\int_{\xi}^{a(t)}(\rho+\rho u^2)y^3 dy\right)drdt\\
        \le&\delta\int_0^T\int_{\xi}^{a(t)}\rho^{2\gamma}r^8drdt+\delta^{-1}C.
    \end{aligned}
\end{equation}
For $I_6,\ I_7\text{ and }I_8$, 
\begin{equation}\label{I67}
    \begin{aligned}
        &|I_6+I_7+I_8|\\
        =&\left|\int_0^T\int_{\xi}^{a(t)}\rho^{\gamma}r^5\left(\int_r^{a(t)}\left[2\rho u^2y^2-3(\rho u^2+\rho^{\gamma})y^2+\int_r^{a(t)}4\pi\rho y\int_{\xi}^r\rho(s,t)s^2ds\right] dy\right)drdt\right|\\
        %\le&C\int_0^T\int_{\xi}^{a(t)}\rho^{\gamma}r^5\left(\int_r^{a(t)}(\rho u^2+\rho^{\gamma})y^2dy\right)drdt\\
        \le& C\int_0^T\int_{\xi}^{a(t)}\rho^{\gamma}y^5drdt\le C,
    \end{aligned}
\end{equation}
where we use the basic energy estimate to control the integral in the interior layer. For $I_9$, we apply H\"older inequality to the interior layer to get
\begin{equation}\label{I9}
    \begin{aligned}
        |I_9|=&\left|\int_0^T\int_{\xi}^{a(t)}\rho^{\gamma}r^5\left(\int_r^{a(t)}\left(\frac{4}{3}\varepsilon+\eta\right)\left(\partial_yu+\frac{2u}{y}\right)3y^2 dy\right)drdt\right|\\
        \le&C\sup_{0\le t\le T}\int_{\xi}^{a(t)}\rho^{\gamma}r^5dr\cdot\left(\int_0^T\int_{\xi}^{a(t)}\left(\partial_yu+\frac{2u}{r}\right)^2(3y^2)^2dydt\right)^{\frac{1}{2}}\le C.
    \end{aligned}
\end{equation}
While for $I_2$, 
\begin{equation}\label{I2}
    \begin{aligned}
        |I_2|=&\left|\int_0^T\int_{\xi}^{a(t)}\partial_t\left(\rho^{\gamma}\int_r^{a(t)}\rho uy^3 dy\right)r^5drdt\right|\\    
        \le&C\sup_{0\le t\le T}\left(\int_{\xi}^{a(t)}\rho^{\gamma}r^5dr\cdot\int_{\xi}^{a(t)}(\rho+\rho u^2)y^3dy\right)\le C.
    \end{aligned}
\end{equation}

Choose $\delta=\frac{1}{6}$, combining (\ref{I1})-(\ref{I2}), we get the result.
\end{proof}
The following lemma gives a lower bound of the particle path $r(x,\tau)$ for each fixed $x\in(0,M]$ and a lower bound of the difference of two paths.
\begin{lem}\label{r_bound}Let $T>0$, $\gamma\in(\frac{6}{5},\frac{4}{3}]$ and $(\rho,u,a)$ be any strong solution to the problem (\ref{apprE1})-(\ref{apprE4}) with initial data $(\rho_0,u_0)$ satisfying $\rho_0\in L^1([\xi,a_0);r^2dr)\cap L^{\gamma}([\xi,a_0);r^2dr)$, $\sqrt{\rho_0}u_0\in L^2([\xi,a_0);r^2dr)$ and (\ref{cond1}). Then
\begin{equation}
    C_{\gamma}E_0^{-\frac{1}{3(\gamma-1)}}x^{\frac{\gamma}{3(\gamma-1)}}\le r(x,\tau)\le a(\tau),\ (x,\tau)\in(0,M]\times[0,T],
\end{equation}
\begin{equation}
    C_{\gamma}E_0^{-\frac{1}{\gamma-1}}(x_2-x_1)^{\frac{\gamma}{\gamma-1}}\le r^3(x_2,\tau)-r^3(x_1,\tau),\ 0< x_1<x_2\le M,\ \tau\in[0,T],
\end{equation}
where $C_{\gamma}>0$ only depends on $\gamma$.
\end{lem}
\begin{proof}Notice that
    \begin{equation}
        \begin{aligned}
            x(r,t)=&4\pi\int_{\xi}^r\rho(s,t)s^2ds\\
            \le&\left(4\pi\int_{\xi}^{r(x,\tau)}\rho^{\gamma}y^2dy\right)^{\frac{1}{\gamma}}\left(4\pi\int_{\xi}^{r(x,\tau)}y^2dy\right)^{\frac{\gamma-1}{\gamma}}\le CE_0^{\frac{1}{\gamma}}r^{\frac{3(\gamma-1)}{\gamma}},
        \end{aligned}
    \end{equation}
    \begin{equation}
        \begin{aligned}
             x_2(r,t)-x_1(r,t)=&4\pi\int_{r(x_1,\tau)}^{r(x_2,\tau)}\rho(s,t)s^2ds\\
            \le&\left(4\pi\int_{r(x_1,\tau)}^{r(x_2,\tau)}\rho^{\gamma}y^2dy\right)^{\frac{1}{\gamma}}\left(4\pi\int_{r(x_1,\tau)}^{r(x_2,\tau)}y^2dy\right)^{\frac{\gamma-1}{\gamma}}\\
            \le& CE_0^{\frac{1}{\gamma}}(r^3(x_2,\tau)-r^3(x_1,\tau))^{\frac{\gamma-1}{\gamma}},
        \end{aligned}
    \end{equation}
    the results just follow. 
\end{proof}
The following lemma provides a pointwise upper bound and lower bound for the density along the particle path.
\begin{lem}\label{rhobound}
    Let $T>0$, $\gamma\in(\frac{6}{5},\frac{4}{3}]$ and $(\rho,u,a)$ be any strong solution to the problem (\ref{apprE1})-(\ref{apprE4}) with initial data satisfying (\ref{cond1}). There exist constants $C=C(x,E_0,M,T)>0$ and $\tilde{C}=\tilde{C}(x,E_0,M,T)>0$ such that
    \begin{equation}
        \rho_0(r(0))\cdot\exp\left(\frac{\tilde{C}}{\frac{4}{3}\varepsilon+\eta}\right)\le\rho(r(t),t)\le\rho_0(r(0))\cdot\exp\left(\frac{C}{\frac{4}{3}\varepsilon+\eta}\right),\ any\ t\in[0,T],
    \end{equation}
    where $r(t)$ is particle path defined as
    \begin{equation}
        \begin{cases}
            \frac{d}{dt}r(t)=u(r(t),t),\\
            r(0)=r,\ r\in(\xi,a_0].
        \end{cases}
    \end{equation}
\end{lem}
\begin{proof}
    The mass equation in (\ref{apprE1}) gives
    \begin{equation}
        \partial_t\rho+u\partial_r\rho+\rho(\partial_ru+\frac{2u}{r})=0,
    \end{equation}
    which implies
    \begin{equation}\label{logrho}
        \partial_t\log\rho+u\partial_r\log\rho+(\partial_ru+\frac{2u}{r})=0.
    \end{equation}
    
    On the other hand,
    \begin{equation}
        \begin{aligned}
            (\partial_t+u\partial_r)\int_{a(t)}^r\rho udy=&-\rho u^2(a(t),t)+\int_{a(t)}^r\partial_t(\rho u)dy+\rho u^2(r,t)\\
            =&-\rho u^2(a(t),t)+\int_{a(t)}^r-\frac{2}{y}\rho u^2dy-\int_{a(t)}^r\partial_y(\rho u^2+\rho^{\gamma})dy\\&-\int_{a(t)}^r\frac{4\pi\rho}{y^2}\int_{\xi}^y\rho(s,t)s^2dsdy+\int_{a(t)}^r\left(\frac{4}{3}\varepsilon+\eta\right)\partial_y(\partial_yu+\frac{2u}{y})dy+\rho u^2(r,t)\\
            =&\int_{a(t)}^r-\frac{2}{y}\rho u^2dy-\left(\rho^{\gamma}-\left(\frac{4}{3}\varepsilon+\eta\right)(\partial_yu+\frac{2u}{y})\right)\big|_{a(t)}^r-\int_{a(t)}^r\frac{4\pi\rho}{y^2}\int_{\xi}^y\rho(s,t)s^2dsdy,
        \end{aligned}
    \end{equation}
    combining with (\ref{logrho}), we get
    \begin{equation}
        \begin{aligned}
            &(\partial_t+u\partial_r)\left(\int_{a(t)}^r\rho udy+\left(\frac{4}{3}\varepsilon+\eta\right)\log\rho\right)\\
            =&\int_r^{a(t)}\frac{2}{y}\rho u^2dy-\rho^{\gamma}+\rho^{\gamma}(a(t),t)-\left(\frac{4}{3}\varepsilon+\eta\right)\left(\partial_ru+\frac{2u}{r}\right)(a(t),t)\\&+\int_r^{a(t)}\frac{4\pi\rho}{y^2}\int_{\xi}^{y}\rho(s,t)s^2dsdy\\
            =&\int_r^{a(t)}\frac{2}{y}\rho u^2dy-\rho^{\gamma}+\int_r^{a(t)}\frac{4\pi\rho}{y^2}\int_{\xi}^{y}\rho(s,t)s^2dsdy,
        \end{aligned}
    \end{equation}
    which yields
    \begin{equation}
        \begin{aligned}
            &\frac{d}{dt}\left(\int_{a(t)}^r\rho udy+\left(\frac{4}{3}\varepsilon+\eta\right)\log\rho\right)+\rho^{\gamma}\\
            =&\int_r^{a(t)}\frac{2}{y}\rho u^2dy+\int_r^{a(t)}\frac{4\pi\rho}{y^2}\int_{\xi}^{y}\rho(s,t)s^2dsdy.
        \end{aligned}
    \end{equation}
    Integrate the above equation from 0 to $t$,
    \begin{equation}
        \begin{aligned}
            &\left(\frac{4}{3}\varepsilon+\eta\right)\log\frac{\rho(r(t),t)}{\rho(r(0))}+\int_0^t\rho^{\gamma}(r(s),s)ds\\
            =&\int_0^t\int_{r(s)}^{a(s)}\frac{2}{y}\rho u^2dyds+\int_0^t\int_{r(\tau)}^{a(\tau)}\frac{4\pi\rho}{y^2}\int_{\xi}^y\rho(s,\tau)s^2dsdyd\tau\\&+\int_{r(t)}^{a(t)}\rho udy-\int_{r_0}^{a_0}\rho udy.
        \end{aligned}
    \end{equation}
    Making use of Lemma \ref{r_bound}, we have the following controls
    \begin{equation}
        \int_{r(t)}^{a(t)}\rho udy\le\frac{1}{r^2(x,t)}\left(\int_{\xi}^{a(t)}\rho u^2y^2dy\right)^{\frac{1}{2}}\left(\int_{\xi}^{a(t)}\rho y^2dy\right)^{\frac{1}{2}}\le CE_0^{\frac{2}{3(\gamma-1)}+\frac{1}{2}}M^{\frac{1}{2}}x^{-\frac{2\gamma}{3(\gamma-1)}},
    \end{equation}
    \begin{equation}
    \begin{aligned}
        \int_0^t\int_{r(s)}^{a(t)}\frac{2}{y}\rho u^2dyds\le&\int_0^t\frac{1}{r^3(x,s)}\int_{r(s)}^{a(t)}\rho u^2y^2dyds\\
        \le&4E_0\int_0^T\frac{1}{r^3(x,s)}ds\le C_{\gamma}TE_0^{1+\frac{1}{\gamma-1}}x^{-\frac{\gamma}{\gamma-1}},
    \end{aligned}
    \end{equation}
    \begin{equation}
        \begin{aligned}
            &\int_0^t\int_{r(\tau)}^{a(\tau)}\frac{4\pi\rho}{y^2}\int_{\xi}^y\rho(s,\tau)s^2dsdyd\tau\\
            \le&\int_0^t\frac{1}{r^3(x,\tau)}\int_{r(\tau)}^{a(\tau)}4\pi\rho y\int_{\xi}^y\rho(s,\tau)s^2dsdyd\tau\le C_{\gamma}TE_0^{1+\frac{1}{\gamma-1}}x^{-\frac{\gamma}{\gamma-1}},
        \end{aligned}
    \end{equation}
    which yield
    \begin{equation}
       \begin{aligned}
       \left(\frac{4}{3}\varepsilon+\eta\right)\log\frac{\rho(r(t),t)}{\rho(r(0))}\le C_1(x,E_0,M,T),
       \end{aligned}
    \end{equation}
    so
    \begin{equation}
        \rho(r(t),t)\le\rho(r(0))\exp\left(\frac{C(x,E_0,M,T)}{\frac{4}{3}\varepsilon+\eta}\right).
    \end{equation}
on the other hand,
\begin{equation}
    \begin{aligned}
        &\left(\frac{4}{3}\varepsilon+\eta\right)\log\frac{\rho(r(t),t)}{\rho(r(0))}\ge-C_1(x,E_0,M,T)-\int_0^t\rho^{\gamma}(r(s),s)ds\\
        \ge&-C_1(x,E_0,M,T)-T\exp\left(\frac{\gamma C(x,E_0,M,T)}{\frac{4}{3}\varepsilon+\eta}\right)\|\rho_0\|_{L^{\infty}}^{\gamma}=:-\tilde{C}(x,E_0,M,T),
    \end{aligned}
\end{equation}
which gives
\begin{equation}
    \rho(r(t),t)\ge\rho(r(0))\exp\left(\frac{\tilde{C}(x,E_0,M,T)}{\frac{4}{3}\varepsilon+\eta}\right).
\end{equation}
\end{proof}
\subsubsection{Further estimates in Lagrangian coordinates}
\label{Estimates in Lagrangian coordinates}
In this subsection, we prove several further weighted $L^2$ estimates for temporal and spatial derivatives of $\rho$ and $u$ away from the origin in Lagrangian coordinates, which are crucial to the convergence of the free boundary.

We introduce the following cutoff functions. Let $0<x_0<x_1<x_2<M$. Let $\phi,\ \zeta\in C^{\infty}[0,M]$ such that
\begin{equation*}
0\le\phi(x)\le1,\ \phi(x)=1\text{ for }x\in[x_1,M],\ \phi(x)=0\text{ for }x\in[0,x_0+\frac{x_1-x_0}{2}],
\end{equation*}
\begin{equation*}
    0\le\phi'(x)\le\frac{C}{x_1-x_0},
\end{equation*}
\begin{equation*}
0\le\zeta(x)\le1,\ \zeta(x)=1\text{ for }x\in[x_2,M],\ \zeta(x)=0\text{ for }x\in[0,x_1+\frac{x_2-x_1}{2}],
\end{equation*}
\begin{equation*}
0\le\zeta'(x)\le\frac{C}{x_2-x_1}.
\end{equation*}
Let $\varphi=\phi^2,\ \psi=\zeta^2$. Then
\begin{equation}
   |\varphi'|=2\phi|\phi'|\le C\sqphi,\ 
   |\psi'|=2\zeta|\zeta'|\le C\sqrt{\psi}.
\end{equation}

We recall the effective viscous flux
\begin{equation}\label{EffVisFlux}
    \sigma=\rho^{\gamma}-\pvcons(4\pi\rho r^2\pxu+\frac{2u}{r}),
\end{equation}
which is an important quantity in the estimates of derivatives of $\rho$ and $u$.

The space-time estimates proved in the following auxiliary lemma will be used frequently in later proofs.
\begin{lem}\label{HigherEnergyEst1}
    Let $T>0$, $\gamma\in(\frac{6}{5},\frac{4}{3}]$ and $(\rho,u,a)$ be any strong solution to the problem (\ref{apprE1})-(\ref{apprE4}) with initial data $(\rho_0,u_0)$ satisfying $\rho_0\in L^1([\xi,a_0);r^2dr)\cap L^{\gamma}([\xi,a_0);r^2dr)$, $\sqrt{\rho_0}u_0\in L^2([\xi,a_0);r^2dr)$ and (\ref{cond1}). Assume further that $\rho_0(x)\in \Linf([x_0,M])$. Then the following estimates hold
    \begin{equation}\label{boundRho} 
        \int_0^T\int_0^M\rho^{-3}(\partial_{\tau}\rho)^2dxd\tau\le C, 
    \end{equation}
    \begin{equation}\label{boundSigma}
            \inttxm\frac{\sigma^2}{\rho} dxd\tau\le C,
    \end{equation}
    \begin{equation}\label{boundU} \int_0^T\|u\|_{L^{\infty}([x_0,M])}^2dt\le C
    \end{equation}
\end{lem}
\begin{proof}For (\ref{boundRho}), from the mass equation in (\ref{nsp4}), we get:
    \begin{equation}
    \begin{aligned}
        \int_0^T\int_0^M\rho^{-3}(\ptau\rho)^2 dxd\tau=\int_0^T\int_0^M16\pi^2\rho\left|r^2\pxu+\frac{u}{2\pi\rho r}\right|^2 dxd\tau\\
        \le C\int_0^T\int_0^M\rho r^4(\pxu)^2+\frac{u^2}{\rho r^2} dxd\tau\le C.
    \end{aligned}
    \end{equation}
For (\ref{boundSigma}),
\begin{equation}
    \begin{aligned}
        \inttxm\frac{\sigma^2}{\rho} dxd\tau=\inttxm\rho^{-1}\left(\rho^{\gamma}-\pvcons4\pi\rho r^2\pxu-\pvcons\frac{2u}{r}\right)^2 dxd\tau\\
        \le C\inttxm\left(\rho^{2\gamma-1}+\rho r^4(\pxu)^2+\frac{u^2}{\rho r^2}\right) dxd\tau\le C.
    \end{aligned}
\end{equation}
For (\ref{boundU}), by Sobolev's inequality,
\begin{equation}
    \begin{aligned}
        \int_0^T\|u\|_{L^{\infty}([x_0,M])}^2dxd\tau\le C\int_0^T\left(\intxm|u| +|\pxu| dx\right)^2d\tau\\
        \le C\int_0^T\intxm u^2+\rho r^4(\pxu)^2+\frac{1}{\rho r^4} dxd\tau\le C.
    \end{aligned}
\end{equation}
\end{proof}
In the following lemma, we prove the spatial weighted $L^2$ estimate for $\px u$, which is the key estimate for the convergence of the free boundary.
\begin{lem}\label{HigherEnergyEst2}
    Let $T>0$, $\gamma\in(\frac{6}{5},\frac{4}{3}]$ and $(\rho,u,a)$ be any strong solution to the problem (\ref{apprE1})-(\ref{apprE4}) with initial data $(\rho_0,u_0)$ satisfying $\rho_0\in L^1([\xi,a_0);r^2dr)\cap L^{\gamma}([\xi,a_0);r^2dr)$, $\sqrt{\rho_0}u_0\in L^2([\xi,a_0);r^2dr)$ and (\ref{cond1}). Assume further that
    \begin{equation}
        \rho_0(x)\in \Linf([x_0,M]),\ (\frac{u_0}{\sqrt{\rho_0}r},\sqrt{\rho_0}r^2\partial_xu_0)\in L^2([x_0,M]),
    \end{equation}
    then
    \begin{equation}\label{HigherEnergyEst2_2}
    \begin{aligned}
        \int_{x_1}^M\frac{\sigma^2}{\rho}dx+\int_{x_1}^M\rho r^4(\partial_xu)^2+\frac{u^2}{\rho r^2}dx+\int_0^T\int_{x_1}^M(\ptu)^2dxd\tau\\
        \le C \int_{x_0}^M\rho_0 r^4(\partial_xu_0)^2+\frac{u_0^2}{\rho_0 r^2}dx+C\|\rho_0\|_{\Linf([x_0,M])}^{2\gamma-1},
    \end{aligned}
    \end{equation}
    where we recall $\sigma=\rho^{\gamma}-\pvcons(4\pi\rho r^2\pxu+\frac{2u}{r})$, and $C=C(x_0,x_1,E_0,M,T)$.
\end{lem}
\begin{proof}
    Multiply the momentum equation in (\ref{apprL1}) by $\partial_{\tau}u\varphi$ and integrate it over $[x_0,M]$,
    \begin{equation}\label{EstDuSq}
        \intxm(\ptu)^2\varphi dx+\intxm4\pi r^2\ptu\partial_{x}\left(\rho^{\gamma}-4\pi\pvcons\rho\partial_x(r^2u)\right)\varphi dx+\intxm\frac{x}{r^2}\ptu\varphi dx=0.
    \end{equation}
    
     For the second term in (\ref{EstDuSq}), we integrate it by parts to get
    \begin{equation}
        \begin{aligned}
            I:=&\intxm4\pi r^2\ptu\partial_{x}\left(\rho^{\gamma}-4\pi\pvcons\rho\partial_x(r^2u)\right)\varphi dx\\
            =&\intxm4\pi\ptau(r^2u)\partial_x\sigma\varphi-8\pi ru^2\partial_x\sigma\varphi dx\\
            =&-\intxm4\pi\partial_x\ptau(r^2u)\sigma\varphi dx-\intxm4\pi\ptau(r^2u)\sigma\varphi'dx+\intxm2u^2(\frac{\ptu}{r}+\frac{x}{r^3})\varphi dx\\
            =&I_1+I_2+I_3,
        \end{aligned}
    \end{equation}
    where $I_1$ can be rewritten in terms of $\sigma=\rho^{\gamma}-\pvcons(4\pi\rho r^2\pxu+\frac{2u}{r})$ as follow 
    \begin{equation}
        \begin{aligned}
            I_1=&\intxm\left(\ptau\left(\frac{\sigma}{\pvcons\rho}\right)\sigma-\frac{\gamma-1}{\vcons}\rho^{\gamma-2}\ptau\rho\sigma\right)\varphi dx\\
            =&\intxm\left(\frac{1}{\vcons}\left(\frac{1}{\sqrt{\rho}}\ptau\frac{\sigma}{\sqrt{\rho}}+\frac{\sigma}{\sqrt{\rho}}\ptau\frac{1}{\sqrt{\rho}}\right)\sigma-\frac{\gamma-1}{\vcons}\rho^{\gamma-2}\ptau\rho\cdot\sigma\right)\varphi dx\\
            =&\intxm\frac{1}{2\pvcons}\ptau(\frac{\sigma^2}{\rho})\varphi dx-\intxm\frac{1}{2\pvcons}\sigma^2\rho^{-2}\ptau\rho\cdot\varphi dx-\intxm\frac{\gamma-1}{\vcons}\rho^{\gamma-2}\ptau\rho\cdot\sigma\varphi dx,
        \end{aligned}
    \end{equation}
    and for $I_2+I_3$, we have
    \begin{equation}
        \begin{aligned}
            I_2+I_3=&-\intxm4\pi\sigma\varphi'(r^2\ptu+2ru^2)dx+\intxm2u^2(\frac{\ptu}{r}+\frac{x}{r^3})\varphi dx,
        \end{aligned}
    \end{equation}
    so (\ref{EstDuSq}) becomes
    \begin{equation}
        \begin{aligned}
            &\intxm(\ptu)^2\varphi dx+\frac{1}{2\pvcons}\ptau\intxm\frac{\sigma^2}{\rho}\varphi dx\\
            =&\frac{1}{2\pvcons}\intxm\sigma^2\rho^{-2}\ptau\rho\cdot\varphi dx+\frac{\gamma-1}{\vcons}\intxm\rho^{\gamma-2}\ptau\rho\cdot\sigma\varphi dx\\
            &+\intxm4\pi\sigma\varphi'(r^2\ptu+2ru^2)dx-\intxm2u^2(\frac{\ptu}{r}+\frac{x}{r^3})\varphi dx-\intxm\frac{x}{r^2}\ptu\varphi dx\\
            \le&\delta\|\sigma\sqphi\|^2_{\Linf([x_0,M])}+\frac{C}{\delta}\left(\intxm\rho^{-3}(\ptau\rho)^2dx\right)\intxm\frac{\sigma^2}{\rho}\varphi dx+C\|\rho^{\gamma}\sqphi\|_{\Linf([x_0,M])}^2\\
            &+\delta\intxm(\ptu)^2\varphi dx+\frac{C}{\delta}\intxm\sigma^2 r^4 dx+C\intxm\sigma^2 dx+C\|ur\sqphi\|_{\Linf([x_0,M])}^2\intxm u^2dx\\
            &+\delta\intxm(\ptu)^2\varphi dx+\frac{C}{\delta}\|\frac{u}{r}\sqphi\|_{\Linf([x_0,M])}^2\intxm u^2dx+\intxm\frac{2x}{r^3}\varphi\cdot u^2dx\\
            &+\delta\intxm(\ptu)^2\varphi dx+\frac{C}{\delta}\intxm\frac{x^2}{r^4}\varphi dx,
        \end{aligned}
    \end{equation}
    which, combined with Lemma \ref{rhobound} and \ref{HigherEnergyEst1}, gives
    \begin{equation}
        \begin{aligned}
            &(1-3\delta)\intxm(\ptu)^2\varphi dx+\frac{1}{2\pvcons}\ptau\intxm\frac{\sigma^2}{\rho}\varphi dx\\
            \le&\frac{C}{\delta}\left(\intxm\rho^{-3}(\ptau\rho)^2dx\right)\intxm\frac{\sigma^2}{\rho}\varphi dx+\delta\|\sigma\sqphi\|_{\Linf([x_0,M])}^2+\frac{C}{\delta}\|u\sqphi\|_{\Linf([x_0,M])}^2\\&+C\intxm\sigma^2dx+C.
        \end{aligned}
    \end{equation}
    For the term $\|\sigma\sqphi\|_{\Linf([x_0,M])}^2$, we estimate as follows
    \begin{equation}
        \begin{aligned}
            \sigma\sqphi=-\int_x^M\partial_y(\sigma\sqphi)dy=-\int_x^M\sigma\partial_y\sqphi dy+\int_x^M(\frac{\ptu}{4\pi r^2}+\frac{x}{4\pi r^4})\sqphi dy,
        \end{aligned}
    \end{equation}
    which gives
    \begin{equation}
        \|\sigma\sqphi\|_{\Linf([x_0,M])}^2\le C_{x_0}\intxm(\ptu)^2\varphi dx+C\intxm\sigma^2dx+C.
    \end{equation}
    As a result, we have
    \begin{equation}
        \begin{aligned}
            &(1-3\delta-C_{x_0}\delta)\intxm(\ptu)^2\varphi dx+\frac{1}{2\pvcons}\ptau\intxm\frac{\sigma^2}{\rho}\varphi dx\\
            \le&\frac{C}{\delta}\left(\intxm\rho^{-3}(\ptau\rho)^2dx\right)\intxm\frac{\sigma^2}{\rho}\varphi dx+\frac{C}{\delta}\|u\sqphi\|_{\Linf([x_0,M])}^2+C\intxm\sigma^2dx+C.
        \end{aligned}
    \end{equation}
    Choose $\delta\in(0,1)$ small enough, by Gronwall's inequality and make use of Lemma \ref{HigherEnergyEst1}, we have
    \begin{equation}
    \begin{aligned}
    \intxm\frac{\sigma^2}{\rho}\varphi dx+\int_0^T\intxm(\ptu)^2\varphi dxd\tau
        \le& C\intxm\rho_0r^4(\partial_xu_0)^2+\frac{u_0^2}{\rho_0r^2}dx+C\|\rho_0\|^{2\gamma-1}_{\Linf([x_0,M])}.
    \end{aligned}
    \end{equation}

    It remains to show the control of $\intxm\left(\rho r^4(\partial_xu)^2+\frac{u^2}{8\pi^2\rho r^2}\right)\varphi dx$. Since
    \begin{equation}
    \begin{aligned}
        \intxm\frac{1}{\rho}(\rho\partial_x(r^2u))^2\varphi dx&=\intxm\left(\rho r^4(\partial_xu)^2+\frac{u^2}{4\pi^2\rho r^2}+\frac{r^2u\partial_xu}{\pi r}\right)\varphi dx\\
        &=\intxm\left(\rho r^4(\partial_xu)^2+\frac{u^2}{8\pi^2\rho r^2}\right)\varphi dx+\frac{ru^2}{2\pi}|_{x=M}-\intxm\frac{u^2r}{2\pi}\varphi'dx,
    \end{aligned}
    \end{equation}
    recalling $\sigma=\rho^{\gamma}-\pvcons(4\pi\rho r^2\pxu+\frac{2u}{r})$, we have
    \begin{equation}
        \begin{aligned}
            &\intxm\left(\rho r^4(\partial_xu)^2+\frac{u^2}{8\pi^2\rho r^2}\right)\varphi dx+\frac{ru^2}{2\pi}|_{x=M}\\\le& C\intxm u^2dx+\intxm\frac{1}{\rho}(\rho\partial_x(r^2u))^2\varphi dx\\
            \le& C\intxm u^2+\frac{\sigma^2}{\rho}\varphi+\rho^{2\gamma-1}\varphi dx\\
            \le& C\intxm\rho_0r^4(\partial_xu_0)^2+\frac{u_0^2}{\rho_0r^2}dx+C\|\rho_0\|^{2\gamma-1}_{\Linf([x_0,M])},
        \end{aligned}
    \end{equation}
    as desired.
\end{proof}
\begin{rmk}\label{dsigmabound}
    With the momentum equation in (\ref{apprL1}), it also follows easily from (\ref{HigherEnergyEst2_2}) that
    \begin{equation}
        \int_0^T\int_{x_1}^M(\px\sigma)^2dxd\tau\le C\intxm\rho_0r^4(\partial_xu_0)^2+\frac{u_0^2}{\rho_0r^2}dx+C\|\rho_0\|^{2\gamma-1}_{\Linf([x_0,M])}.
    \end{equation}
\end{rmk}
The following lemma gives the spatial $H^1$ regularity of the density.
\begin{lem}\label{HigherEnergyEst3}
    Let $T>0$, $\gamma\in(\frac{6}{5},\frac{4}{3}]$ and $(\rho,u,a)$ be any strong solution to the problem (\ref{apprE1})-(\ref{apprE4}) with initial data $(\rho_0,u_0)$ satisfying $\rho_0\in L^1([\xi,a_0);r^2dr)\cap L^{\gamma}([\xi,a_0);r^2dr)$, $\sqrt{\rho_0}u_0\in L^2([\xi,a_0);r^2dr)$ and (\ref{cond1}). Assume further that $\partial_x\rho_0^q\in L^2([x_0,M])$, then
    \begin{equation}
        \int_{x_1}^M(\partial_x\rho^q)^2dx\le C\int_{x_1}^M(\partial_x\rho^q_0)^2dx+C\int_{x_0}^M\rho_0 r^4(\partial_xu_0)^2+\frac{u_0^2}{\rho_0 r^2}dx+C\|\rho_0\|_{\Linf([x_0,M])}^{2\gamma-1},
    \end{equation}
    where $C=C(x_0,x_1,E_0,M,T)$, $\frac{1}{2}<q=k+\frac{1}{2}\le\gamma$.
\end{lem}
\begin{proof}
    Multiply the mass equation in (\ref{apprL1}) by $q\rho^{q-1}$ and take $\partial_x$,
    \begin{equation}
        \partial_{\tau x}\rho^q+4\pi q\partial_x(\rho^{q+1}\partial_x(r^2u))=0,
    \end{equation}
    multiply it by $\partial_x\rho^q$ and integrate it over $[x_1,M]$, we get
    \begin{equation}
        \int_{x_1}^M\partial_x\rho^q\partial_{\tau x}\rho^q+4\pi q(\partial_x\rho^q)^2\rho\partial_x(r^2u)+4\pi q\partial_x\rho^q\cdot\rho^q\partial_x(\rho\partial_x(r^2u))dx=0.
    \end{equation}
    Making use of the momentum equation in (\ref{apprL1}), we get
    \begin{equation}
        \begin{aligned}
            &\frac{1}{2}\ptau\int_{x_1}^M(\partial_x\rho^q)^2dx\\=&-\int_{x_1}^M4\pi q(\partial_x\rho^q)^2\rho\partial_x(r^2u)dx-q\int_{x_1}^M\rho^q\partial_x\rho^q\frac{1}{4\pi \pvcons r^2}(\ptu+4\pi r^2\partial_x\rho^{\gamma}+\frac{x}{r^2})dx.
        \end{aligned}
    \end{equation}
    Using Lemma \ref{HigherEnergyEst1},  \ref{HigherEnergyEst2} and the fact that $\rho r^2\px u=\frac{1}{4\pi(\eta+\frac{4}{3}\varepsilon)}(\rho^{\gamma}-\sigma)-\frac{u}{2\pi r}$, we obtain
    \begin{equation}
    \begin{aligned}
        &\frac{1}{2}\ptau\int_{x_1}^M(\partial_x\rho^q)^2dx\\
        \le&4\pi q\int_{x_1}^M\left|\rho r^2\partial_xu+\frac{u}{2\pi r}\right|(\partial_x\rho^q)^2dx+\frac{q}{4\pi\pvcons}\int_{x_1}^M|\rho^q\partial_x\rho^q|\left(\left|\frac{\ptu}{r^2}\right|+\frac{x}{r^4}\right)+\frac{\gamma}{q}|\rho^{\gamma}(\partial_x\rho^q)^2|dx\\
            \le&C\left(\|\rho\|_{\Linf([x_0,M])}^{\gamma}+\|\sigma\sqphi\|_{\Linf([x_0,M])}+\|u\sqphi\|_{\Linf([x_0,M])}\right)\int_{x_1}^M(\partial_x\rho^q)^2dx\\
            &+C\|\rho\|_{\Linf([x_0,M])}^{2q}\int_{x_1}^M(\partial_x\rho^q)^2dx+\int_{x_1}^M\left(\left|\frac{\ptu}{r^2}\right|+\frac{x}{r^4}\right)^2dx.
        \end{aligned}
    \end{equation}
    By Gronwall's inequality,
    \begin{equation}
        \int_{x_1}^M(\partial_x\rho^q)^2dx\le C\int_{x_1}^M(\partial_x\rho^q_0)^2dx+C\int_{x_0}^M\rho_0 r^4(\partial_xu_0)^2+\frac{u_0^2}{\rho_0 r^2}dx+C\|\rho_0\|_{\Linf([x_0,M])}^{2\gamma-1}.
    \end{equation}
\end{proof}
The following lemma shows, in addition to the initial conditions (\ref{conditionsI}) required by Theorem \ref{Thm1: ExistenceOfGlobalWkSoln}, if we assume further that $\partial_x(\rho_0\partial_x(r^2u_0))\in L^2([x_1,M])$, we will get the spatial $L^2$ estimate for $\ptau u$.
\begin{lem}\label{HigherEnergyEst4}
     Let $T>0$, $\gamma\in(\frac{6}{5},\frac{4}{3}]$ and $(\rho,u,a)$ be any strong solution to the problem (\ref{apprE1})-(\ref{apprE4}) with initial data $(\rho_0,u_0)$ satisfying $\rho_0\in L^1([\xi,a_0);r^2dr)\cap L^{\gamma}([\xi,a_0);r^2dr)$, $\sqrt{\rho_0}u_0\in L^2([\xi,a_0);r^2dr)$ and (\ref{cond1}). Assume further that $\partial_x(\rho_0\partial_x(r^2u_0))\in L^2([x_1,M])$, then
     \begin{equation}
         \begin{aligned}
             &\int_{x_2}^M(\ptu)^2 dx+\int_0^T\int_{x_2}^M\frac{(\ptau\sigma)^2}{\rho}dxd\tau+\int_0^T\int_{x_2}^M4\pi\rho r^4(\partial_{\tau x}u)^2+\frac{(\ptu)^2}{\rho r^2}dxd\tau+\int_0^Ta(t)(a''(t))^2d\tau\\\le& C\int_{x_1}^M(\partial_x(\rho_0\partial_x(r^2u_0)))^2dx+C\int_{x_1}^M(\partial_x\rho^q_0)^2dx+C\int_{x_0}^M\rho_0 r^4(\partial_xu_0)^2+\frac{u_0^2}{\rho_0 r^2}dx+C\|\rho_0\|_{\Linf([x_0,M])}^{2\gamma-1},
         \end{aligned}
     \end{equation}
     where $\sigma=\rho^{\gamma}-\pvcons(4\pi\rho r^2\pxu+\frac{2u}{r})$ and $C=C(x_0,x_1,x_2,E_0,M,T)$.
\end{lem}
\begin{proof}
We apply $\ptau$ to the momentum equation in (\ref{apprL1}) and multiply it by $\ptu\psi$ to get
    \begin{equation}
        \ptu\cdot\partial^2_{\tau}u\psi+4\pi\ptu\cdot\ptau(r^2\px\rho^{\gamma})\psi+\ptu\cdot\ptau\left(\frac{x}{r^2}\right)\psi=\ptu\cdot16\pi^2\ptau\left(r^2\px\left(\pvcons\rho\px(r^2u)\right)\right)\psi,
    \end{equation}
    integrate it over $[x_1,M]$, we get
    \begin{equation}
        \begin{aligned}
            \frac{1}{2}\ptau\int_{x_1}^M(\ptu)^2\psi dx=&-\int_{x_1}^M8\pi ru\ptu\px\left(\rho^{\gamma}-4\pi\pvcons\rho\px(r^2u)\right)\psi dx\\
            &-\int_{x_1}^M4\pi r^2\ptu\cdot\ptau\px\left(\rho^{\gamma}-4\pi\pvcons\rho\px(r^2u)\right)\psi dx\\&+\int_{x_1}^M u\ptu\frac{2x}{r^3}\psi dx\\
            =&I_1+I_2+I_3.
        \end{aligned}
    \end{equation}
     For $I_1$, we use the momentum equation in (\ref{apprL1})
    \begin{equation}
        \begin{aligned}
            I_1=&\int_{x_1}^M8\pi ru\ptu\cdot\frac{1}{4\pi r^2}(\ptu+\frac{x}{r^2})\psi dx\\
            \le&C\left(\|u\|_{\Linf([x_1,M])}+1\right)\int_{x_1}^M(\ptu)^2\psi dx+C\int_{x_1}^M u^2dx\\
            \le&C\left(\|u\|_{\Linf([x_1,M])}+1\right)\int_{x_1}^M(\ptu)^2\psi dx+C.
        \end{aligned}
    \end{equation}
    For $I_3$, we estimate as follows
    \begin{equation}
        I_3=\int_{x_1}^M u\ptu\frac{2x}{r^3}\psi dx\le C\int_{x_1}^M u^2\psi dx+C\int_{x_1}^M(\ptu)^2\psi dx\le C\int_{x_1}^M(\ptu)^2\psi dx+C.
    \end{equation}
    For $I_2$, by integration by parts and rewrite it in terms of $\sigma$, we have
    \begin{equation}
        \begin{aligned}
            I_2=&\int_{x_1}^M4\pi\px(r^2\ptau u)\ptau(\rho^{\gamma}-4\pi\pvcons\partial_x(r^2u))\psi+4\pi r^2\ptau u\ptau\sigma\psi'dx\\
            =&\int_{x_1}^M4\pi(\partial_{\tau x}(r^2u)-\partial_x(2ru^2))\ptau\sigma\psi+4\pi r^2\ptau u\ptau\sigma\psi'dx\\
            =&\int_{x_1}^M4\pi\left[\frac{1}{4\pi\pvcons}\ptau\left(\frac{\rho^{\gamma}-\sigma}{\rho}\right)-\partial_x(2ru^2)\right]\ptau\sigma\psi+4\pi r^2\ptau u\ptau\sigma\psi'dx\\
            =&\int_{x_1}^M4\pi\left[\frac{1}{4\pi\pvcons}\left((\gamma-1)\rho^{\gamma-2}\ptau\rho\ptau\sigma-\frac{1}{\rho}(\ptau \sigma)^2+\frac{1}{\rho^2}\sigma\ptau\rho\ptau\sigma\right)\right]\psi dx\\&
            -\int_{x_1}^M\left(\frac{2u^2}{\rho r^2}\ptau\sigma\psi+16\pi ru\pxu\ptau\sigma\psi\right) dx+\int_{x_1}^M4\pi r^2\ptau u\ptau\sigma\psi'dx.
        \end{aligned}
    \end{equation}
    Making use of the mass equation in (\ref{apprL1}) to replace $\ptau\rho$, we obtain
    \begin{equation}
        \begin{aligned}
            I_2=&\int_{x_1}^M-\frac{1}{\eta+\frac{4}{3}\varepsilon}\frac{(\ptau \sigma)^2}{\rho}\psi dx\\
            &+\int_{x_1}^M\left(\frac{1-\gamma}{\vcons}\rho^{\gamma-\frac{1}{2}}(4\pi\sqrt{\rho} r^2\partial_xu+\frac{2u}{\sqrt{\rho} r})\ptau\sigma+\frac{1}{\eta+\frac{4}{3}\varepsilon}(4\pi r^2\partial_xu+\frac{2u}{\rho r})\sigma\ptau\sigma\right)\psi dx\\&
            -\int_{x_1}^M\left(\frac{2u^2}{\rho r^2}\ptau\sigma\psi+16\pi ru\pxu\ptau\sigma\psi\right) dx+\int_{x_1}^M4\pi r^2\ptau u\ptau\sigma\psi'dx\\
            \le&-\frac{1}{\eta+\frac{4}{3}\varepsilon}\int_{x_1}^M\frac{(\ptau \sigma)^2}{\rho}\psi dx+5\delta\int_{x_1}^M\frac{(\ptau \sigma)^2}{\rho}\psi dx\\
            &+\frac{C}{\delta}(\|\sigma\sqrt{\psi}\|_{\Linf([x_1,M])}+1)\int_{x_1}^M\left(4\pi\sqrt{\rho} r^2\partial_xu+\frac{2u}{\sqrt{\rho} r}\right)^2dx+\frac{C}{\delta}\|u\sqrt{\psi}\|_{\Linf([x_1,M])}^2\int_{x_1}^M\frac{u^2}{\rho r^2}dx
            \\&+\frac{C}{\delta}\|u\sqrt{\psi}\|_{\Linf([x_1,M])}^2\int_{x_1}^M4\pi\rho r^4(\partial_xu)^2dx+\frac{C}{\delta}\|\rho\|_{\Linf([x_1,M])}\int_{x_1}^M(\ptu)^2dx.
        \end{aligned}
    \end{equation}
    Choose $\delta=\frac{1}{10}\frac{1}{\eta+\frac{4}{3}\varepsilon}$, we obtain
    \begin{equation}
    \begin{aligned}
        &\frac{1}{2}\ptau\int_{x_1}^M(\ptu)^2\psi dx+\frac{1}{2(\eta+\frac{4}{3}\varepsilon)}\int_{x_1}^M\frac{(\ptau\sigma)^2}{\rho}\psi dx\\
        \le& C\left(\|u\|_{\Linf([x_1,M])}+\|\rho\|_{\Linf([x_1,M])}+1\right)\int_{x_1}^M(\ptu)^2\psi dx\\
        &+C\left(\|u\sqrt{\psi}\|_{\Linf([x_1,M])}^2+\|\sigma\sqrt{\psi}\|_{\Linf([x_1,M])}^2+1\right)\int_{x_1}^M4\pi\rho r^4(\partial_xu)^2+\frac{u^2}{\rho r^2}dx+C.
    \end{aligned}
    \end{equation}
    By Gronwall's inequality, 
    \begin{equation}
    \begin{aligned}
         &\int_{x_1}^M(\ptu)^2\psi dx+\int_0^T\int_{x_1}^M\frac{(\ptau\sigma)^2}{\rho}\psi dxd\tau\\\le& C\int_{x_1}^M(\partial_x(\rho_0\partial_x(r^2u_0)))^2dx+C\int_{x_1}^M(\partial_x\rho^q_0)^2dx+C\int_{x_0}^M\rho_0 r^4(\partial_xu_0)^2+\frac{u_0^2}{\rho_0 r^2}dx+C\|\rho_0\|_{\Linf([x_0,M])}^{2\gamma-1}.
    \end{aligned}
    \end{equation}

    As for $\int_0^T\int_{x_2}^M4\pi\rho r^4(\partial_{\tau x}u)^2+\frac{(\ptu)^2}{\rho r^2}dxd\tau$, we control it by $\ptau\sigma$. We start from
    \begin{equation}\label{HigherEnergyEst4: 1}
        \begin{aligned}
            &\int_0^T\int_{x_1}^M\frac{1}{\rho}(\ptau \sigma)^2\psi dxd\tau=\int_0^T\int_{x_1}^M\frac{1}{\rho}\left[\ptau(\rho^{\gamma}-4\pi\pvcons\rho\px(r^2u))\right]^2\psi dxd\tau\\
            =&\int_0^T\int_{x_1}^M\left[\gamma\rho^{\gamma-\frac{3}{2}}\ptau\rho-4\pi\rho^{-\frac{1}{2}}\pvcons(\ptau\rho)\px(r^2u)-4\pi\pvcons\rho^{\frac{1}{2}}\ptau\px(r^2u)\right]^2\psi dxd\tau.
        \end{aligned}
    \end{equation}
    The first and second term on the right-hand side of (\ref{HigherEnergyEst4: 1}) can be controlled by the previous estimates. In fact, for the first term, we have
    \begin{equation}
        \int_0^T\int_{x_1}^M\rho^{2\gamma-3}(\ptau\rho)^2dxd\tau\le\|\rho\|^{2\gamma}_{\Linf([x_1,M])}\int_0^T\int_{x_1}^M\rho^{-3}(\ptau\rho)^2dxd\tau\le C\|\rho\|^{2\gamma}_{\Linf([x_1,M])}.
    \end{equation}
    Next, we estimate the second term on the right-hand side of (\ref{HigherEnergyEst4: 1}) as follow
    \begin{equation}
    \begin{aligned}
        \int_0^T\int_{x_1}^M\rho^{-1}(\ptau\rho)^2(\px(r^2u))^2\psi dxd\tau\le&\int_0^T\|\rho^{-1}\ptau\rho\sqrt{\psi}\|^2_{\Linf([x_1,M])}\int_{x_1}^M(\sqrt{\rho} r^2(\pxu)+\frac{u}{2\pi\sqrt{\rho}r})^2dxd\tau\\
        \le&C\int_0^T\|\rho^{-1}\ptau\rho\sqrt{\psi}\|^2_{\Linf([x_1,M])}d\tau,
    \end{aligned}
    \end{equation}
    by Sobolev embedding and the mass equation in (\ref{apprL1}), we have
    \begin{equation}
        \begin{aligned}
            &\int_0^T\|\rho^{-1}\ptau\rho\sqrt{\psi}\|^2_{\Linf([x_1,M])}d\tau\le C\int_0^T\left(\int_{x_1}^M|4\pi\rho\partial_x(r^2u)|+|\partial_x(4\pi\rho\partial_x(r^2u))|dx\right)^2d\tau\\
            \le&C\int_0^T\left(\int_{x_1}^M|4\pi\rho\partial_x(r^2u)|+|\frac{1}{4\pi\pvcons r^2}(\ptu+4\pi r^2\partial_x\rho^{\gamma}+\frac{x}{r^2})|dx\right)^2d\tau\\
            \le&C\int_0^T\int_{x_1}^M\rho r^4(\partial_xu)^2+\frac{u^2}{\rho r^2}+(\ptu)^2+\rho+(\partial_x\rho^{\gamma})^2dxd\tau\\
            \le& C+C\int_{x_1}^M(\partial_x\rho^{\gamma}_0)^2dx+C\int_{x_0}^M\rho_0 r^4(\partial_xu_0)^2+\frac{u_0^2}{\rho_0 r^2}dx+C\|\rho_0\|_{\Linf([x_0,M])}^{2\gamma-1}.
        \end{aligned}
    \end{equation}
    So the second term is controlled. As a result, $\int_0^T\int_{x_1}^M\rho(\partial_{\tau}\partial_x(r^2u))^2\psi dxd\tau$ is also controlled. Since
    \begin{equation}
        \begin{aligned}
        &\int_0^T\int_{x_1}^M\rho(\partial_{\tau}\partial_x(r^2u))^2\psi dxd\tau\\=&\int_0^T\int_{x_1}^M(2\sqrt{\rho}ru\pxu+\sqrt{\rho}r^2\partial_{\tau x}u+\frac{\ptu}{2\pi\sqrt{\rho} r}-\frac{u\ptau\rho}{2\pi\rho^{\frac{3}{2}}r}-\frac{u^2}{2\pi\sqrt{\rho} r^2})^2\psi dxd\tau,
        \end{aligned}
    \end{equation}
    and we have the following estimates
    \begin{equation}
        \int_0^T\int_{x_1}^M\rho r^2u^2(\pxu)^2\psi dxd\tau\le\int_0^T\|\frac{u}{r}\sqrt{\psi}\|_{\Linf([x_1,M])}^2\int_{x_1}^M\rho r^4(\pxu)^2dxd\tau\le C,
    \end{equation}
    \begin{equation}
        \int_0^T\int_{x_1}^M\frac{u^2(\ptau\rho)^2}{\rho^{3}r^2}\psi dxd\tau\le\int_0^T\|\frac{u}{r}\sqrt{\psi}\|_{\Linf([x_1,M])}^2\int_{x_1}^M\rho^{-3}(\ptau\rho)^2dxd\tau\le C,
    \end{equation}
    \begin{equation}
         \int_0^T\int_{x_1}^M\frac{u^4}{\rho r^4}\psi dxd\tau\le\int_0^T\|\frac{u}{r}\sqrt{\psi}\|_{\Linf([x_1,M])}^2\int_{x_1}^M\frac{u^2}{\rho r^2}dxd\tau\le C
    \end{equation}
    we conclude that $\int_0^T\int_{x_1}^M(\sqrt{\rho}r^2\partial_{\tau x}u+\frac{\ptu}{2\pi\sqrt{\rho}r})^2\psi dxd\tau$ is controlled. What's more,
    \begin{equation}
        \begin{aligned}
            &\int_0^T\int_{x_1}^M(\sqrt{\rho}r^2\partial_{\tau x}u+\frac{\ptu}{2\pi\sqrt{\rho}r})^2\psi dxd\tau\\=&\int_0^T\int_{x_1}^M\rho r^4(\partial_{\tau x}u)^2\psi+\frac{(\ptu)^2}{4\pi^2\rho r^2}\psi dxd\tau+\int_0^T\int_{x_1}^M\frac{r}{2\pi}\px((\ptu)^2)\psi dxd\tau\\
            =&\int_0^T\int_{x_1}^M\rho r^4(\partial_{\tau x}u)^2\psi+\frac{(\ptu)^2}{8\pi^2\rho r^2}\psi dxd\tau+\int_0^T\frac{a(t)}{2\pi}(a''(t))^2d\tau-\int_0^T\int_{x_1}^M\frac{1}{2\pi}r(\ptu)^2\psi' dxd\tau,
        \end{aligned}
    \end{equation}
    which gives
    \begin{equation}
        \begin{aligned}
        &\int_0^T\int_{x_2}^M4\pi\rho r^4(\partial_{\tau x}u)^2+\frac{(\ptu)^2}{\rho r^2}dxd\tau+\int_0^Ta(t)(a''(t))^2d\tau\\
            \le& C\int_{x_1}^M(\partial_x(\rho_0\partial_x(r^2u_0)))^2dx+C\int_{x_1}^M(\partial_x\rho^q_0)^2dx+C\int_{x_0}^M\rho_0 r^4(\partial_xu_0)^2+\frac{u_0^2}{\rho_0 r^2}dx+C\|\rho_0\|_{\Linf([x_0,M])}^{2\gamma-1}.
        \end{aligned}
    \end{equation}
\end{proof}
\begin{rmk}\label{HigherEnEst5}
    It follows from Lemma \ref{rhobound} that $\rho$ has a strictly positive lower bound on $[x_0,M-\delta]$ for any $\delta>0$ small enough. As a result, from Lemma \ref{HigherEnergyEst2}-\ref{HigherEnergyEst4}, we have the following interior regularities

    I. If assuming the conditions in (\ref{conditionsI}), which implies $\|\rho_0\|_{H^1([x_0,M-\delta])},\ \|u_0\|_{H^1([x_0,M-\delta])}<\infty$, then
    \begin{equation}
    \medmath{
        \begin{aligned}
            &\sup_{0\le\tau\le T}\left(\|\px u\|_{L^2([x_1,M-2\delta])}^2+\|\px\rho\|_{L^2([x_1,M-2\delta])}^2+\|\ptau\rho\|_{L^2([x_1,M-2\delta])}^2+\|\sigma\|_{L^2([x_1,M-2\delta])}^2\right)\\
            +&\int_0^T\|\ptu\|_{L^2([x_1,M-2\delta])}^2+\|\px^2u\|^2_{L^2([x_1,M-2\delta])}+\|\px\sigma\|_{L^2([x_1,M-2\delta])}^2+\|\partial_{\tau x}\rho\|^2_{L^2([x_1,M-2\delta])}d\tau\le C_1,
        \end{aligned}}
    \end{equation}
Where $C_1>0$ depends on $\|\rho_0\|_{H^1([x_0,M-\delta])},\ \|u_0\|_{H^1([x_0,M-\delta])}$ and the lower bound of $\rho_0$ on $[x_0,M-\delta]$.

    II. If assume further that $\|u_0\|_{H^2([x_1,M-\delta])}<\infty$, then
     \begin{equation}
    \medmath{
        \begin{aligned}
            &\sup_{0\le\tau\le T}\left(\|\ptau u\|_{L^2([x_2,M-2\delta])}^2+\|\px^2u\|_{L^2([x_2,M-2\delta])}^2+\|\partial_x\sigma\|_{L^2([x_2,M-2\delta])}^2\right)\\
            +&\int_0^T\|\ptau\sigma\|_{L^2([x_2,M-2\delta])}^2+\|\partial_{xx}\sigma\|_{L^2([x_2,M-2\delta])}^2+\|\partial_{\tau x}u\|^2_{L^2([x_2,M-2\delta])}d\tau\le C_2,
        \end{aligned}}
    \end{equation}
    Where $C_2>0$ depends on $\|\rho_0\|_{H^1([x_0,M-\delta])},\ \|u_0\|_{H^2([x_1,M-\delta])}$ and the lower bound of $\rho_0$ on $[x_0,M-\delta]$.
\end{rmk}
\subsection{Compactness arguments}
Suppose $(\rho_0,u_0)$ are initial data satisfying (\ref{cond1}) and
\begin{equation}
    \begin{aligned}
    \rho_0\ge0,\ \rho_0\in L^1([0,a_0),&r^2dr)\cap L^{\gamma}([0,a_0),r^2dr),\ \\(\rho_0)^k\in H^1([0,a_0),r^2dr)\ (0<k&\le\gamma-\frac{1}{2}),\ u_0\in H^1([0,a_0),r^2dr),\\
    \rho_0(r)>0\text{ for }r\in(0,a_0),&\ \rho_0(a_0)=0,\ \partial_ru_0+\frac{2u_0(a_0)}{a_0}=0.
    \end{aligned}
\end{equation}

We mollify the initial data using a similar method as in \cite{CHWY} to get $(\rhox_0,\ux_0)$ such that $\rhox_0,\ \ux_0$ are smooth on $[\xi,a_0]$, 
\begin{equation}\label{mollification of ID}
    \begin{aligned}
        \inf_{r\in[\xi,a_0]}\rhox_0(r)&>0,\ \ux_0(\xi)=0,\\
        (\rhox_0)^{\gamma}-(\eta+\frac{4}{3}\varepsilon)&(\partial_r\ux_0+\frac{2\ux_0}{r})(a_0)=0,\\
        (\rhox_0,\ux_0)\rightarrow(\rho_0,u_0)\text{ in } H^1(&[0,a_0),r^2dr),\ \rhox_0(r)\rightarrow\rho_0(r)\ (\xi\rightarrow0)\\
        \int_{\xi}^{a_0}\rhox_0 r^2dr&=\int_0^{a_0}\rho_0 r^2dr.
    \end{aligned}
\end{equation}
Then the problem (\ref{apprE1})-(\ref{apprE4}) with initial data $(\rhox_0,\ux_0)$ has a global weak solution in the sense of distribution. In fact, we can apply a finite difference argument to obtain the existence of a local solution $(\rhox,\ux)$ for some short time $T^*>0$. By a continuity argument, we can extend the solution to any finite time $T>0$, as in \cite{CK,G}. Under Lagrangian coordinates, the following regularities hold for a solution $(\rhox(x,\tau),\ux(x,\tau))$:
\begin{equation}\label{ReguApprSoln}
    \begin{aligned}
        &C^{-1}\rhox_0(x)\le\rhox(\tau,x)\le C\rhox_0(x),\\
        &|\rhox(\cdot,x)-\rhox(\cdot,y)|\le C|x-y|^{\frac{1}{2}},\ |\rhox(s,\cdot)-\rhox(t,\cdot)|\le C|s-t|^{\frac{1}{2}},\\
        &\rhox\in C([0,T];L^2[0,M])\cap L^{\infty}(0,T;H^1([0,M])),\\
        &\ux\in L^{\infty}\cap H^1_{loc}([0,M)\times[0,T]),\\
        &\sqrt{\rhox}\px(\rhox)^{\gamma-1},\ptu^{\xi},\sqrt{\rhox}\pxu^{\xi},\sqrt{\rhox}\sigma^{\xi}\in L^2(0,T;L^2([0,M])),\\
        &\sqrt{\rhox}\partial_x((r^{\xi})^2 \ux)\in L^1(0,T;L^{\infty}([0,M]))\cap L^{\infty}(0,T; L^2([0,M])),
    \end{aligned}
\end{equation}
where the norms of $\rhox$ and $\ux$ in the functional spaces in (\ref{ReguApprSoln}) and $C>0$ may not be uniformly bounded in $\xi$. But the a priori estimates in Lemma \ref{rhoIntegrability}-\ref{HigherEnergyEst3} hold for the solution $(\rhox,\ux)$, and are uniform in $\xi$, which will allow us to take $\xi\rightarrow0$ in the following subsections.

Extend $(\rhox,\ux)$ by setting $(\rhox(r,t),\ux(r,t))=(0,0)$ when $0\le r\le\xi$, and we still denote the resulting functions by $(\rhox,\ux)$.

\subsubsection{Convergence near the free boundary}\label{Convergence near the free boundary}
As for the region near the free boundary, since the solutions enjoy better regularities, we will have a stronger sense of convergence via Arzela-Ascoli's theorem.

Fix $0<x_0<x_1<x_2$ as in Lemma \ref{HigherEnergyEst2}-\ref{HigherEnergyEst4}. From (\ref{BoundaryBdd}), we know $\ax(t)$ is uniformly bounded in $H^1([0,T])$. Since $ H^1([0,T])\hookrightarrow\hookrightarrow C^{0,\alpha}([0,T])$ for $0<\alpha<\frac{1}{2}$, there exists a sequence of $\{\xi_j\}$ with $\xi_j\rightarrow0$ such that
\begin{equation}
    \axj(\cdot)\rightarrow a(\cdot)\text{ in }C^{0,\alpha}([0,T]).
\end{equation}
As for $\{r^{\xi_j}_{x_1}\}$, from (\ref{BoundaryBdd}), Lemma \ref{rhobound} and Lemma \ref{HigherEnergyEst2}, we have
\begin{equation}\label{EstAwayFromBdry}
    \int_{x_1}^{M-\delta}(\uxj)^2+(\px\uxj)^2dx+\int_0^T\int_{x_1}^{M-\delta}(\ptau\uxj)^2dxd\tau\le C
    \end{equation}
    for any $\delta>0$ that is small enough, and the constant $C>0$ also depends on the lower bound of $\rho_0$ on $[x_0,M-\delta]$, which is finite. Then $\uxj(x_1,\tau)\in L^2([0,T])$ is well defined. As a result, $r^{\xi_j}(x_1,\tau)\in H^1([0,T])\hookrightarrow\hookrightarrow C^{0,\alpha}([0,T])\ (0\le\alpha<\frac{1}{2})$, which implies there exists a subsequence
\begin{equation}
    r^{\xi_j}_{x_1}(\cdot)\rightarrow r_{x_1,b}(\cdot)\text{ in }C^{\alpha}([0,T]).
\end{equation}

Next, we focus on the convergence of $(\rhoxj,\uxj)$ on the region $[x_1,M]\times[0,T]$. From Lemma \ref{HigherEnergyEst2}-\ref{HigherEnergyEst3} and (\ref{EstAwayFromBdry}), we have, up to a subsequence,
\begin{equation}\label{ConvInLag}
    \begin{aligned}
        &\uxj\rightharpoonup u_b\text{ weak* in }\Linf(0,T;W^{1,1}([x_1,M])),\\
         &\uxj\rightharpoonup u_b\text{ weak* in }\Linf(0,T;H^1([x_1,M-\delta]))\text{ for any $\delta>0$ small enough},\\
         &\ptau\uxj\rightharpoonup \ptau u_b\text{ weakly in }L^2(0,T;L^2([x_1,M])),\\
        &\rhoxj\rightharpoonup\rho_b\text{ weak* in }\Linf(0,T;H^1([x_1,M])),\\
        &\ptau\rhoxj\rightharpoonup\ptau\rho_b\text{ weakly in }L^2(0,T;L^2([x_1,M])).
    \end{aligned}
\end{equation}

The following Lemmas show the strong convergence of $(\rhoxj,\uxj)$ on $[x_1,M]$.
\begin{lem}\label{StrConvOfrho}
    There exists a subsequence of $\{\rhoxj\}$, still denoted by $\{\rhoxj\}$, such that $\rhoxj\rightarrow\rho_b$ in $C([x_1,M]\times[0,T])$.
\end{lem}
\begin{proof}
    From (\ref{ConvInLag}), $\{\rhoxj\}$ is bounded in $H^{1}(0,T;L^2([x_1,M]))\cap L^{\infty}(0,T;H^1([x_1,M]))$. By Sobolev embedding, $\rhoxj\in C(0,T;L^2([x_1,M]))\cap L^{\infty}(0,T;C^{\frac{1}{2}}([x_1,M]))$. Thus we have $\{\rhoxj\}$ is uniformly bounded and
    \begin{equation}
        |\rhoxj(x',\tau)-\rhoxj(x'',\tau)|\le C|x'-x''|^{\frac{1}{2}}
    \end{equation}
    for any $\tau\in[0,T]$.
    
    On the other hand, $H^1([x_1,M])\hookrightarrow\hookrightarrow C^{0,\alpha}([x_1,M])\hookrightarrow L^2([x_1,M])$ for $0<\alpha<\frac{1}{2}$, for any $\delta>0$, there exists $C_\delta>0$ such that
    \begin{equation}
    \begin{aligned}
        &\|\rhoxj(\cdot,\tau_1)-\rhoxj(\cdot,\tau_2)\|_{\Linf([x_1,M])}\\
        \le&\delta\|\rhoxj(\cdot,\tau_1)-\rhoxj(\cdot,\tau_2)\|_{H^1([x_1,M])}+C_{\delta}\|\rhoxj(\cdot,\tau_1)-\rhoxj(\cdot,\tau_2)\|_{L^2([x_1,M])}\\
        \le&2\delta\|\rhoxj\|_{\Linf(0,T;H^1([x_1,M]))}+C_{\delta}\|\partial_{\tau}\rhoxj\|_{L^2([0,T]\times[x_1,M]))}\cdot|\tau_1-\tau_2|^{\frac{1}{2}}\\
        \le&C\delta+C_{\delta}\cdot|\tau_1-\tau_2|^{\frac{1}{2}}.
    \end{aligned}
    \end{equation}
    So $\{\rhoxj\}$ is equicontinuous. By Arzela-Ascoli's theorem, there exists a subsequence of $\{\rhoxj\}$ convergent to $\rho_b$ in $C([x_1,M]\times[0,T])$.
\end{proof}
\begin{lem}\label{StrConvOfu}
    There exists a subsequence of $\{\uxj\}$, still denoted by $\{\uxj\}$, such that $\uxj\rightarrow u_b$ in $C(0,T;L^p([x_1,M]))\ \text{, where } 1\le p<\infty$.
\end{lem}
\begin{proof}
    It follows from Lemma \ref{HigherEnergyEst2} that $\{\uxj\}$ is bounded in $\Linf(0,T;W^{1,1}([x_1,M]))$ and 
    \begin{equation}
        \int_0^T\int_{x_1}^M(\uxj)^2+(\ptau\uxj)^2dxd\tau\le C,
    \end{equation}
    for any $\delta>0$, there exists $C_{\delta}>0$ such that
    \begin{equation}
        \begin{aligned}
            &\|\uxj(\cdot,\tau_1)-\uxj(\cdot,\tau_2)\|_{L^p([x_1,M])}\\
            \le&\delta\|\uxj(\cdot,\tau_1)-\uxj(\cdot,\tau_2)\|_{W^{1,1}([x_1,M])}+C_{\delta}\|\uxj(\cdot,\tau_1)-\uxj(\cdot,\tau_2)\|_{L^2([x_1,M])}\\
            \le&2\delta\|\uxj\|_{\Linf(0,T;W^{1,1}([x_1,M]))}+C_{\delta}\|\ptau\uxj\|_{L^2([0,T]\times[x_1,M])}\cdot|\tau_1-\tau_2|^{\frac{1}{2}}\\
            \le&C\delta+C_{\delta}\cdot|\tau_1-\tau_2|^{\frac{1}{2}},
        \end{aligned}
    \end{equation}
    where $1\le p<\infty$. Thus $\{\uxj\}$ is equicontinuous in $\tau$. On the other hand, by compact embedding, for any $\tau\in[0,T]$, $\{\uxj(\cdot,\tau)\}$ has a convergent subsequence in $L^p([x_1,M])$. By Arzela-Ascoli's theorem, there exists a subsequence of $\{\uxj\}$ convergent to $u_b$ in $C([0,T];L^p([x_1,M]))$.
\end{proof}
\begin{rmk}\label{StrConvOfu: Int}
   Following a similar argument as in Lemma \ref{StrConvOfrho}, we will have
    \begin{equation}\label{UniformConvOfU}
        \uxj\rightarrow u_b\text{ in }C([x_1,M-\delta]\times[0,T]).
    \end{equation}
\end{rmk}

Next, we derive the convergence of $r^{\xi_j}(x,\tau)$. From the relation $\px r^{\xi_j}=\frac{1}{4\pi\rhoxj(r^{\xi_j})^2}$, we have
\begin{equation}
    r^{\xi_j}(x,0)=\left( (r^{\xi_j})^3(x_1,0)+3\int_{x_1}^x\frac{dy}{4\pi\rhoxj(y,0)}\right)^{\frac{1}{3}}
\end{equation}
for $\tau=0$. So $r^{\xi_j}(\cdot,0)$ is uniformly bounded in $ W^{1,1}([x_1,M])\cap H^1([x_1,M-\delta])$. As a result, up to a subsequence,
\begin{equation}
\begin{aligned}
    &r^{\xi_j}(x,0)\rightarrow r_b(x,0)\text{ pointwisely},\\
    &r^{\xi_j}(x,0)\rightarrow r_b(x,0)\text{ in }C^{0,\alpha}([x_1,M-\delta])\ (0\le\alpha<\frac{1}{2})\text{ for any small enough }\delta>0.
\end{aligned}
\end{equation}
Define $r_b(x,\tau)$ by 
\begin{equation}
    r_b(x,\tau)=r_b(x,0)+\int_0^\tau u_b(x,s)ds,
\end{equation}
we have $r^{\xi_j}(x,\tau)\rightarrow r_b(x,\tau)$ in $C^1([0,T];L^p([x_1,M]))$. Also, by (\ref{UniformConvOfU}), if we fix any $x<M$, we have $r_b(x,\tau)\in C^1([0,T])$. What's more, it follows from (\ref{ConvInLag}) that
\begin{equation}
\begin{aligned}
    &\px r_b\in L^{\infty}(0,T;L^1([x_1,M]));\\
    &\partial_{\tau x} r_b\in L^{\infty}(0,T;L^1([x_1,M]))\cap L^{\infty}(0,T;L^2_{loc}([x_1,M)));\\
    &\partial_{\tau\tau} r_b\in L^{2}(0,T;L^2([x_1,M])).
\end{aligned}
\end{equation}

As for the convergence of $\sigma^{\xi_j}$, it follows from Lemma \ref{HigherEnergyEst2} and Remark \ref{dsigmabound} that
    \begin{equation}
        \int_0^T\int_{x_1}^M(\sigma^{\xi_j})^2+(\partial_x\sigma^{\xi_j})^2dxd\tau\le C,
    \end{equation}
    which gives
    \begin{equation}
        \begin{aligned}
            \sigma^{\xi_j}\rightharpoonup\sigma_b\text{ weakly in } L^2(0,T;H^1([x_1,M])),
            \begin{comment}
                \\\sigma^{\xi_j}\rightharpoonup\sigma_b\text{ weakly in }& L^2(0,T;H^1([x_1,M-\delta]))\text{ for any }\delta>0\text{ that is small enough},
            \end{comment}
        \end{aligned}
    \end{equation}
where $\sigma_b=\rho_b^{\gamma}-\pvcons(4\pi\rho_b r_b^2\pxu_b+\frac{2u_b}{r_b})$.

\subsubsection{Convergence near the origin}
\label{Convergence near the origin}
As for the region near the origin, the solutions have lower regularities. We establish the compactness argument by adapting the idea from \cite{JZ}. The difference is that the spatial domain changes in different time $t\in[0,T]$.

Since $r^{\xi_j}_{x_1}(\cdot)\rightarrow r_{x_1,b}(\cdot)$ and $a^{\xi_j}(\cdot)\rightarrow a(\cdot)$ in $C^{\alpha}([0,T])$ for $0<\alpha<\frac{1}{2}$, we can choose a trajectory path $r_{in}(t)$ such that
\begin{equation*}
    r^{\xi_j}_{x_1}(t)<r_{in}(t)<a^{\xi_j}(t),\ r_{x_1,b}(t)<r_{in}(t)<a(t)
\end{equation*}
for all $\xi_j$ small enough (i.e., for all $j>N$, where $N$ is a large enough number). Then it follows from the basic energy estimate (\ref{BasicEnergyEst}) and Lemma \ref{rhobound} that, up to a subsequence,
\begin{equation}\label{weakstarlim}
    \begin{aligned}
        &\rhoxj\rightharpoonup\rho_{in}\text{ weak* in }\Linf(0,T;L^{\gamma}([0,r_{in}(t)),r^2dr)),\\
        &\rhoxj\rightharpoonup\rho_{in}\text{ weakly in }L^{2\gamma}(0,T;L^{2\gamma}_{loc}((0,r_{in}(t)),r^2dr)),\\
        &\uxj\rightharpoonup\uin\text{ weakly in }L^2(0,T;H^1([0,r_{in}(t)),r^2dr)).
    \end{aligned}
\end{equation}
By lower semicontinuity, 
\begin{equation}
    \begin{aligned}
        &\rho_{in}\in\Linf(0,T;L^{\gamma}([0,r_{in}(t)),r^2dr))\cap L^{2\gamma}(0,T;L^{2\gamma}_{loc}((0,r_{in}(t)),r^2dr)),\\
        &\uin\in L^2(0,T;H^1([0,r_{in}(t)),r^2dr)),\\
        &\frac{\uin}{r}\in L^2(0,T;L^2([0,r_{in}(t)),r^2dr)).
    \end{aligned}
\end{equation}

We need to show that $(\rhoin,\uin)$ satisfies equations (\ref{nsp3}) in the weak sense on the set $\Omega_{in}:=\{(r,t):0\le t\le T,\ 0\le r\le\rin(t)\}$. We denote the weak limit of $f(\rhoxj)$ by $\overline{f(\rho)}$ as $\xi_j\rightarrow0$.

In the proofs of the following several lemmas, we will apply a generalized version of the Aubin-Lions lemma, that is, Lemma \ref{App1} in the Appendix. For any function $f(r,t)$ defined on the domain $[0,\rin(t))\times[0,T]$, we use the flow map $$\varphi_{-t}: f(r,t)\mapsto f(\frac{y}{r_{in,0}}\rin(t),t).$$
By the fact that $\rin(t)\in C^1([0,T])$ and $\rin(t)$ having a strictly positive lower bound, the flow map $\varphi_t$ satisfies the conditions in Lemma \ref{App1} for any Sobolev space over $[0,\rin(t))$. By direct computation,
$$\Dt f(r,t)=\pt f(r,t)+\frac{u(\rin(t),t)}{\rin(t)}r\pr f(r,t).$$

The main difficulty to prove $(\rhoxj,\uxj)$ indeed converges to a solution on $[0,a(t))\times[0,T]$ is the pressure term. Convexity only implies $\overline{\rho^{\theta}}\le\rhoin^{\theta}$ for $0<\theta<1$ and $\overline{\rho^{\theta}}\ge\rhoin^{\theta}$ for $\theta>1$. We will prove $\overline{\rho^{\theta}}=\rhoin^{\theta}$ for $0<\theta\le1$. The first step to approach this conclusion is the following lemma.
\begin{lem}\label{wkconvergenceLemma1}
    For any $0<\theta<\gamma$, we have
    \begin{equation}
        \overline{\rho^{\theta+\gamma}}-\pvcons\overline{\rho^{\theta}\partial_ru}=\overline{\rho^{\theta}}\overline{\rho^{\gamma}}-\pvcons\overline{\rho^{\theta}}\partial_r\uin.
    \end{equation}
\end{lem}
\begin{proof}
From the mass equation in (\ref{nsp3}), we get
    \begin{equation}\label{rhotheta}
        \partial_t(\rhoxj)^{\theta}+\partial_r(\uxj(\rhoxj)^{\theta})+\frac{2\theta}{r}\uxj\cdot(\rhoxj)^{\theta}=(1-\theta)(\rhoxj)^{\theta}\partial_r\uxj,
    \end{equation}
where $0<\theta\le\gamma$. Multiply the momentum equation in (\ref{nsp3}) by $\phi\in C_c((0,\rin(t))\times[0,T])$, integrate it from $r$ to $a(t)$, we get
\begin{equation}\label{DtIntrhou}
    \begin{aligned}
        &\pt\int_r^{a(t)}\rhoxj \uxj\phi dy-\int_r^{a(t)}\rhoxj \uxj\pt\phi dy+\int_r^{a(t)}\frac{2}{y}\rhoxj (\uxj)^2\phi dy-\int_r^{a(t)}(\rhoxj (\uxj)^2+(\rhoxj)^{\gamma})\partial_y\phi dy\\
        &+\int_r^{a(t)}\pvcons(\partial_y \uxj+\frac{2\uxj}{y})\partial_y\phi dy+\int_r^{a(t)}\frac{4\pi\rhoxj}{y^2}\phi\int_0^y\rhoxj(s,t)s^2dsdy\\
        =&\rhoxj (\uxj)^2\phi+(\rhoxj)^{\gamma}\phi-\pvcons(\pr \uxj+\frac{2\uxj}{y})\phi.
    \end{aligned}
\end{equation}
Following a similar procedure as in Lemma \ref{rhoIntegrability}, we get
\begin{equation}\label{rhoGammaTheta}
    \medmath{\begin{aligned}
        &(\rhoxj)^{\gamma+\theta}\phi-\pvcons(\partial_r\uxj+\frac{2\uxj}{r})(\rhoxj)^{\theta}\phi\\
        =&\partial_t\left((\rhoxj)^{\theta}\int_r^{a(t)}\rhoxj \uxj\phi dy\right)+\partial_r\left(\uxj(\rhoxj)^{\theta}\int_r^{a(t)}\rhoxj \uxj\phi dy\right)\\&+\left(\frac{2\theta}{r}\uxj(\rhoxj)^{\theta}+(\theta-1)(\rhoxj)^{\theta}\partial_r\uxj\right)\int_r^{a(t)}\rhoxj \uxj\phi dy\\
        &-(\rhoxj)^{\theta}\int_r^{a(t)}\rhoxj \uxj\partial_t\phi dy+(\rhoxj)^{\theta}\int_r^{a(t)}\frac{2}{y}\rhoxj (\uxj)^2\phi dy-(\rhoxj)^{\theta}\int_r^{a(t)}(\rhoxj (\uxj)^2+(\rhoxj)^{\gamma})\partial_y\phi dy\\
        &+(\rhoxj)^{\theta}\int_r^{a(t)}\pvcons(\partial_y\uxj+\frac{2\uxj}{y})\partial_y\phi dy+(\rhoxj)^{\theta}\int_r^{a(t)}\frac{4\pi\rhoxj}{y^2}\phi\int_0^r\rhoxj(s,t)s^2dsdy.
    \end{aligned}}
\end{equation}
We need to take the weak limit of each term on the right-hand side of (\ref{rhoGammaTheta}).

    We first establish the weak convergence of $\rhoxj\uxj$ and $\rhoxj(\uxj)^2$. By basic energy estimate (\ref{BasicEnergyEst}) and H\"older inequality,
    \begin{equation}
        \int_0^{a(t)}(\rhoxj\uxj)^{\frac{2\gamma}{\gamma+1}}r^2dr\le\left(\int_0^{a(t)}(\rhoxj)^{\gamma}r^2dr\right)^{\frac{1}{\gamma+1}}\left(\int_0^{a(t)}\rhoxj(\uxj)^2r^2dr \right)^{\frac{\gamma}{\gamma+1}}\le C,
    \end{equation}
    which, combining with the mass equation in (\ref{nsp3}), gives
    \begin{equation}
        \Dt\rhoxj\in\Linf(0,T;W^{-1,\frac{2\gamma}{\gamma+1}}((s_1,r_{in}(t))))
    \end{equation}
    for any $0<s_1<s_0$, where $s_0>0$ and $s_0<\min_{0\le t\le T}\{r^{\xi_j}_{x_1}(t),r_{x_j,b}(t):j\ge N\}$ for a large enough $N$. By (\ref{weakstarlim}), we have $\rhoxj\rightharpoonup\rhoin$ weakly in $L^2(0,T;L^2((s_1,r_{in}(t))))$ and $\uxj\rightharpoonup\uin$ weakly in $L^2(0,T;L^2((s_1,r_{in}(t))))$. By Lemma \ref{App2}, combining with the uniform boundedness of $\uxj$ in $\LTwoHOne$, we conclude
    \begin{equation}\label{ConvOfRhoU}
        \rhoxj\uxj\rightharpoonup\rhoin\uin\text{ in }D'((s_1,\rin(t))\times(0,T)).
    \end{equation}
Since $\rhoxj\uxj\in L^2(0,T;L^{\gamma}((s_1,r_{in}(t))))$, the convergence holds weakly in $L^2(0,T;L^{\gamma}((s_1,r_{in}(t))))$. Similarly, from the momentum equation in (\ref{nsp3}), we get
\begin{equation}
\begin{aligned}
    \Dt(\rhoxj\uxj)\in&L^2(0,T;W^{-1,2}(((s_1,r_{in}(t))))+\Linf(0,T;W^{-1,1}((s_1,r_{in}(t))))\\
    &+L^{\infty}(0,T;W^{-1,\frac{2\gamma}{\gamma+1}}(((s_1,r_{in}(t))))\\
    &\subseteq L^1(0,T;W^{-2,1}((s_1,r_{in}(t)))),
\end{aligned}
\end{equation}
by Lemma \ref{App2},
\begin{equation}\label{rhouSq}
    \rhoxj(\uxj)^2\rightharpoonup\rhoin(\uin)^2\text{ in }D'((s_1,\rin(t))\times(0,T)).
\end{equation}
Since
\begin{equation}
    r(\uxj)^2(r,t)=\int_{\xi_j}^r\partial_y(y(\uxj(y,t))^2)dy\le C\int_{\xi_j}^r((\partial_y\uxj)^2+\frac{2(\uxj)^2}{y^2})y^2dy=:A^2(t),
\end{equation}
where $A(t)\in L^2([0,T])$, by Young's inequality (with powers $(2\gamma+1,\frac{2\gamma}{2\gamma+1})$),
\begin{equation}
    \int_0^T\int_{s_1}^{r_{in}(t)}(\rhoxj)^{\frac{4\gamma}{2\gamma+1}}(\uxj)^{\frac{8\gamma}{2\gamma+1}}drdt\le C\int_0^T\int_{s_1}^{r_{in}(t)}(\rhoxj)^{2\gamma}+A^2(t)\rhoxj(\uxj)^2drdt\le C.
\end{equation}
So the convergence of $\rhoxj(\uxj)^2$ holds in $L^{\frac{4\gamma}{2\gamma+1}}((s_1,r_{in}(t))\times[0,T])$.

Next, we deal with the terms containing $\int_r^{a(t)}\rhoxj\uxj\phi dy$. From (\ref{DtIntrhou}), we know
\begin{equation}
\begin{aligned}
    \Dt&\int_r^{a(t)}\rhoxj\uxj\phi dy\\&\in\Linf((s_1,r_{in}(t))\times[0,T])+L^2((s_1,r_{in}(t))\times[0,T])+\Linf(0,T;L^1((s_1,r_{in}(t)))),
\end{aligned}
\end{equation}
by Aubin-Lions lemma,
\begin{equation}\label{Strong Cov of IntRhoU}
    \int_r^{a(t)}\rhoxj\uxj\phi dy\rightarrow\int_r^{a(t)}\rhoin\uin\phi dy\text{ strongly in }L^2(0,T;L^p((s_1,r_{in}(t))))\ (p\le\infty).
\end{equation}
Also, by H\"older inequality,
\begin{equation}
\begin{aligned}
    (\rhoxj)^{\theta}\in& L^{\frac{2\gamma}{\theta}}(0,T;L^{\frac{2\gamma}{\theta}}_{loc}((0,r_{in}(t)))),\\
     (\rhoxj)^{\gamma}\in& L^2((0,T;L^2_{loc}(0,r_{in}(t)))),\\
    (\rhoxj)^{\theta}\partial_r\uxj\in& L^{\frac{2\gamma}{\theta+\gamma}}(0,T;L^{\frac{2\gamma}{\theta+\gamma}}_{loc}((0,r_{in}(t)))),
\end{aligned}
\end{equation}
combined with (\ref{Strong Cov of IntRhoU}), we have
    \begin{equation}\label{wkconvergenceLemma1: part 1}
    \begin{aligned}
        (\rhoxj)^{\theta}\int_r^{a(t)}\rhoxj\uxj\phi dy&\rightharpoonup\overline{\rho^{\theta}}\int_r^{a(t)}\rho_{in} u_{in}\phi dy\text{ in }D'((s_1,r_{in}(t))\times(0,T)),\\
        (\rhoxj)^{\gamma}\int_r^{a(t)}\rhoxj\uxj\pt\phi dy&\rightharpoonup\overline{\rho^{\gamma}}\int_r^{a(t)}\rho_{in} u_{in}\pt\phi dy\text{ in }D'((s_1,r_{in}(t))\times(0,T)),\\
        (\rhoxj)^{\theta}\pr\uxj\int_r^{a(t)}\rhoxj\uxj\phi dy&\rightharpoonup\overline{\rho^{\theta}\pr u}\int_r^{a(t)}\rho_{in} u_{in}\phi dy\text{ in }D'((s_1,r_{in}(t))\times(0,T)).
    \end{aligned}
\end{equation}

Next, for the terms involving $\int_r^{a(t)}\frac{1}{y}\rhoxj(\uxj)^2\phi dy$, by (\ref{rhouSq}) and dominated convergence theorem,
\begin{equation}
\begin{aligned}
    \int_r^{a(t)}\frac{1}{y}&\rhoxj(\uxj)^2\phi dy\rightharpoonup\int_r^{a(t)}\frac{1}{y}\rho_{in} u^2_{in}\phi dy\\
    &\text{ weakly/weak* in }L^p((s_1,r_{in}(t))\times(0,T))\ (1<p<\infty/p=\infty).
\end{aligned}
\end{equation}
By (\ref{rhotheta}), we know
\begin{equation}
    \Dt(\rhoxj)^{\theta}\in L^{\frac{2\gamma}{\theta+\gamma}}(0,T;W^{-1,\frac{2\gamma}{\theta+\gamma}}((s_1,r_{in}(t)))),
\end{equation}
thus
\begin{equation}\label{ConvOfEverything}
    \begin{aligned}
        (\rhoxj)^{\theta}\int_r^{a(t)}\frac{1}{y}\rhoxj(\uxj)^2\phi dy
        \rightharpoonup\overline{\rho^{\theta}}\int_r^{a(t)}\frac{1}{y}\rho_{in} u^2_{in}\phi dy\text{ in }D'((s_1,r_{in}(t))\times(0,T)).
    \end{aligned}
\end{equation}

The weak convergence of the following terms can be obtained similarly:
\begin{equation}\label{ConvOfEverything2}
    \begin{aligned}
        &(\rhoxj)^{\theta}\int_r^{a(t)}\rhoxj(\uxj)^2\partial_y\phi dy,\ (\rhoxj)^{\theta}\int_r^{a(t)}(\rhoxj)^{\gamma}\partial_y\phi dy,\\
        &(\rhoxj)^{\theta}\int_r^{a(t)}\partial_y\uxj\partial_y\phi dy,\ (\rhoxj)^{\theta}\int_r^{a(t)}\frac{\uxj}{y}\partial_y\phi dy,\ (\rhoxj)^{\theta}\frac{\uxj}{r}\\
        \rightharpoonup&\ \overline{\rho^{\theta}}\int_r^{a(t)}\rho_{in} u^2_{in}\partial_y\phi dy,\ \overline{\rho^{\theta}}\int_r^{a(t)}\ \overline{\rho^{\gamma}}\partial_y\phi dy,\\
        &\overline{\rho^{\theta}}\int_r^{a(t)}\partial_yu_{in}\partial_y\phi dy,\ 
 \overline{\rho^{\theta}}\int_r^{a(t)}\frac{u_{in}}{y}\partial_y\phi dy,\ \overline{\rho^{\theta}}\frac{u_{in}}{r}\text{ in }D'((s_1,r_{in}(t))\times(0,T)).
    \end{aligned}
\end{equation}

Next, by (\ref{rhoGammaTheta}), 
\begin{equation}
    \Dt\left((\rhoxj)^{\theta}\int_r^{a(t)}\rhoxj\uxj\phi dy\right)\in L^1((s_1,r_{in}(t))\times[0,T])+L^1(0,T;W^{-1,1}((s_1,r_{in}(t)))),
\end{equation}
then by Lemma \ref{App2} and (\ref{wkconvergenceLemma1: part 1}),
\begin{equation}
    \begin{aligned}
        \uxj(\rhoxj)^{\theta}\int_r^{a(t)}\rhoxj\uxj\phi dy&\rightharpoonup u_{in}\overline{\rho^{\theta}}\int_r^{a(t)}\rho_{in} u_{in}\phi dy\text{ in }D'((s_1,r_{in}(t))\times(0,T)),\\
        \frac{1}{r}\uxj(\rhoxj)^{\theta}\int_r^{a(t)}\rhoxj\uxj\phi dy&\rightharpoonup\frac{1}{r}u_{in}\overline{\rho^{\theta}}\int_r^{a(t)}\rho_{in} u_{in}\phi dy\text{ in }D'((s_1,r_{in}(t))\times(0,T)).    
\end{aligned}
\end{equation}

At last, we derive the weak convergence of the gravitational term, that is, the last term on the right-hand side of (\ref{rhoGammaTheta}). From the mass equation in (\ref{nsp3}), we have $\Dt\int_0^r\rhoxj y^2dy\in\Linf(0,T;L^{\frac{2\gamma}{\gamma+1}}((s_1,r_{in}(t))))$. Also, $\int_0^r\rhoxj y^2dy\in L^{2\gamma}(0,T;W^{1,2\gamma}((s_1,r_{in}(t))))$, by Aubin-Lions lemma,
\begin{equation}
    \int_0^r\rhoxj y^2dy\rightarrow\int_0^r\rho_{in} y^2dy\text{ in }L^{2\gamma}(0,T;L^{2\gamma}((s_1,r_{in}(t)))).
\end{equation}
From the basic energy estimate in Lemma \ref{BasicEnergyEst}, $\int_0^{a(t)}\frac{1}{r^2}(\int_0^r\rhoxj s^2ds)^2dr$ is uniformly bounded, thus
\begin{equation}
    \frac{1}{r}\int_0^r\rhoxj s^2ds\rightharpoonup\frac{1}{r}\int_0^r\rho_{in} s^2ds\text{ weak* in }\Linf(0,T;L^2((0,r_{in}(t)))),
\end{equation}
thus the convergence is weakly in $L^p(0,T;L^2((0,r_{in}(t))))$ for $1\le p<\infty$. Since
\begin{equation}
    \frac{1}{r+h}\int_0^{r+h}\rhoxj s^2ds-\frac{1}{r}\int_0^r\rhoxj s^2ds=\frac{1}{r+h}\int_r^{r+h}\rhoxj s^2dx+(\frac{1}{r+h}-\frac{1}{r})\int_0^r\rhoxj s^2ds,
\end{equation}
by Sobolev embedding, $W^{1,2\gamma}((s_1,r_{in}(t)))\hookrightarrow C^{1-\frac{1}{2\gamma}}((s_1,r_{in}(t)))$, we have
\begin{equation}
    \left|\int_r^{r+h}\rhoxj s^2ds\right|\le C|h|^{1-\frac{1}{2\gamma}}.
\end{equation}
Then
\begin{equation}
\begin{aligned}
    \|\frac{1}{r+h}\int_0^{r+h}\rhoxj s^2ds-\frac{1}{r}\int_0^r\rhoxj s^2ds\|_{L^p(0,T;L^2((s_1,r_{in}(t))))}\rightarrow0\text{ as }|h|\rightarrow0\text{ uniformly in }\xi.
\end{aligned}
\end{equation}
By Lemma \ref{App2},
\begin{equation}\label{Weak conv of rho times mass}
    \frac{\rhoxj}{r}\cdot\frac{1}{r}\int_0^r\rhoxj s^2ds\rightharpoonup\frac{\rho_{in}}{r}\cdot\frac{1}{r}\int_0^r\rho_{in} s^2ds\text{ in }D'((s_1,r_{in}(t))\times(0,T)).
\end{equation}
Since $\frac{\rhoxj}{r}\cdot\frac{1}{r}\int_0^r\rhoxj s^2ds\in L^{\frac{2\gamma p}{2\gamma+p}}(0,T;L^{\frac{2\gamma}{\gamma+1}}((s_1,r_{in}(t))))$, the weak convergence also holds in the same space. So
\begin{equation}
    \int_r^{a(t)}\frac{4\pi\rhoxj\phi}{y^2}\int_0^y\rhoxj s^2dsdy\text{ converges weakly in }L^{\frac{2\gamma p}{2\gamma+p}}((s_1,r_{in}(t))\times(0,T)).
\end{equation}
Take $p=\frac{2\gamma}{2\gamma-(1+\theta)}$, then $\frac{2\gamma p}{2\gamma+p}=\frac{2\gamma}{2\gamma-2\theta}$. So
\begin{equation}
    (\rhoxj)^{\theta}
    \int_r^{a(t)}\frac{4\pi\rhoxj\phi}{y^2}\int_0^y\rhoxj s^2dsdy\rightharpoonup\overline{\rho^{\theta}}
    \int_r^{a(t)}\frac{4\pi\rho_{in}\phi}{y^2}\int_0^y\rho_{in} s^2dsdy\text{ in }D'((s_1,r_{in}(t))\times(0,T)).
\end{equation}

As a result, if we take $\xi_j\rightarrow0$ in (\ref{rhoGammaTheta}), we have
\begin{equation}\label{rhoGammaThetaWk}
    \medmath{\begin{aligned}
        &\overline{\rho^{\gamma+\theta}}\phi-\pvcons(\overline{\rho^{\theta}\pr u}+\overline{\rho^{\theta}}\frac{2u_{in}}{r})\phi\\
        =&\partial_t\left(\overline{\rho^{\theta}}\int_r^{a(t)}\rho_{in} u_{in}\phi dy\right)+\partial_r\left(u_{in}\overline{\rho^{\theta}}\int_r^{a(t)}\rho_{in} u_{in}\phi dy\right)+\left(\frac{2\theta}{r}u_{in}\overline{\rho^{\theta}}+(\theta-1)\overline{\rho^{\theta}\partial_ru}\right)\int_r^{a(t)}\rho_{in} u_{in}\phi dy\\
        &-\overline{\rho^{\theta}}\int_r^{a(t)}\rho_{in} u_{in}\partial_t\phi dy+\overline{\rho^{\theta}}\int_r^{a(t)}\frac{2}{y}\rho_{in} u^2_{in}\phi dy-\overline{\rho^{\theta}}\int_r^{a(t)}(\rho_{in} u^2_{in}+\overline{\rho^{\gamma}})\partial_y\phi dy\\
        &+\overline{\rho^{\theta}}\int_r^{a(t)}\pvcons(\partial_yu_{in}+\frac{2u_{in}}{y})\partial_y\phi dy+\overline{\rho^{\theta}}\int_r^{a(t)}\frac{4\pi\rho_{in}}{y^2}\phi\int_0^r\rho_{in}(s,t)s^2dsdy.
    \end{aligned}}
\end{equation}

On the other hand, with the help of (\ref{ConvOfRhoU}), (\ref{rhouSq}), (\ref{ConvOfEverything}) and (\ref{Weak conv of rho times mass}), if taking $\xi_j\rightarrow0$ in the equation (\ref{apprE1}) and (\ref{rhotheta}) (in the sense of distribution), we get
\begin{equation}\label{WkLimEq1}
    \begin{cases}
        \pt\rho_{in}+\pr(\rho_{in} u_{in})+\frac{2\rho_{in} u_{in}}{r}=0,\\
        \pt(\rho_{in} u_{in})+\pr(\rho_{in} u^2_{in}+\overline{\rho^{\gamma}})+\frac{2}{r}\rho_{in} u^2_{in}=-\frac{4\pi\rho_{in}}{r^2}\int_0^r\rho_{in} s^2ds+\pr\left(\frac{\eta+\frac{4}{3}\varepsilon}{r^2}\pr(r^2u_{in})\right),
    \end{cases}
\end{equation}
and
\begin{equation}\label{WkLimEq2}
    \partial_t\overline{\rho^{\theta}}+\partial_r(u_{in}\overline{\rho^{\theta}})+\frac{2\theta}{r}u_{in}\cdot\overline{\rho^{\theta}}=(1-\theta)\overline{\rho^{\theta}\partial_ru},
\end{equation}
where the convergence is in the sense of distribution. By a similar computation, we obtain
\begin{equation}\label{rhoGammaThetaWkWk}
    \medmath{\begin{aligned}
        &\overline{\rho^{\gamma}}\overline{\rho^{\theta}}\phi-\pvcons(\partial_r\uin+\frac{2\uin}{r})\overline{\rho^{\theta}}\phi\\
        =&\partial_t\left(\overline{\rho^{\theta}}\int_r^{a(t)}\rho_{in} u_{in}\phi dy\right)+\partial_r\left(u_{in}\overline{\rho^{\theta}}\int_r^{a(t)}\rho_{in} u_{in}\phi dy\right)+\left(\frac{2\theta}{r}u_{in}\overline{\rho^{\theta}}+(\theta-1)\overline{\rho^{\theta}\partial_ru}\right)\int_r^{a(t)}\rho_{in} u_{in}\phi dy\\
        &-\overline{\rho^{\theta}}\int_r^{a(t)}\rho_{in} u_{in}\partial_t\phi dy+\overline{\rho^{\theta}}\int_r^{a(t)}\frac{2}{y}\rho_{in} u^2_{in}\phi dy-\overline{\rho^{\theta}}\int_r^{a(t)}(\rho_{in} u^2_{in}+\overline{\rho^{\gamma}})\partial_y\phi dy\\
        &+\overline{\rho^{\theta}}\int_r^{a(t)}\pvcons(\partial_yu_{in}+\frac{2u_{in}}{y})\partial_y\phi dy+\overline{\rho^{\theta}}\int_r^{a(t)}\frac{4\pi\rho_{in}}{y^2}\phi\int_0^r\rho_{in}(s,t)s^2dsdy.
    \end{aligned}}
\end{equation}
Comparing (\ref{rhoGammaThetaWk}) and (\ref{rhoGammaThetaWkWk}), we get the conclusion.
\end{proof}
Although we do not establish the $L^2$ integrability of the density, the following lemma is sufficient for proving the conclusion for the weak limit of $(\rhoxj)^{\theta}$.
\begin{lem}\label{wkconvergenceLemma2}
    Let $0<\theta<1$ satisfy $\frac{1}{2}(1-\theta+\sqrt{1+6\theta+\theta^2})\le\gamma$, then
    \begin{equation}
        \rhoin^{\theta}-\overline{\rho^{\theta}}\in L^{\frac{2}{\theta}}(0,T;L^{\frac{2}{\theta}}((0,\rin(t)),r^2dr)).
    \end{equation}
\end{lem}
\begin{proof}
    By Lemma \ref{wkconvergenceLemma1}, we have
    \begin{equation}\label{rhoDiff1}
        \overline{\rho^{\theta+\gamma}}-\overline{\rho^{\theta}}\overline{\rho^{\gamma}}=(\eta+\frac{4}{3}\varepsilon)(\overline{\rho^{\theta}\pr u}-\overline{\rho^{\theta}}\pr\uin).
    \end{equation}
    By convexity\cite[Corollary 3.33]{AI}, we have $\overline{\rho^{\gamma+\theta}}\ge\overline{\rho^{\gamma}}^{\frac{\gamma+\theta}{\gamma}},\ \overline{\rho^{\gamma}}\ge\rhoin^{\gamma},\ \rhoin^{\theta}\ge\overline{\rho^{\theta}}$. Then
    \begin{equation}\label{rhoDiff2}
        \overline{\rho^{\theta+\gamma}}-\overline{\rho^{\theta}}\overline{\rho^{\gamma}}\ge\overline{\rho^{\gamma}}^{\frac{\gamma+\theta}{\gamma}}-\overline{\rho^{\theta}}\overline{\rho^{\gamma}}\ge\overline{\rho^{\gamma}}(\overline{\rho^{\gamma}}^{\frac{\theta}{\gamma}}-\overline{\rho^{\theta}})\ge\rhoin^{\gamma}(\rhoin^{\theta}-\overline{\rho^{\theta}})\ge0.
    \end{equation}
    It follows from the basic energy estimate in Lemma \ref{BasicEnergyEst} that
    \begin{equation}
        (\rhoxj)^{\theta}\pr\uxj,\ \rhoin^{\theta}\pr\uin\in L^{\frac{2\gamma}{\gamma+2\theta}}(0,T;L^{\frac{2\gamma}{\gamma+2\theta}}((0,\rin(t)),r^2dr)),
    \end{equation}
    which implies
    \begin{equation}
        \overline{\rho^{\theta}\pr u},\ \overline{\rho^{\theta}}\pr\uin\in L^{\frac{2\gamma}{\gamma+2\theta}}(0,T;L^{\frac{2\gamma}{\gamma+2\theta}}((0,\rin(t)),r^2dr)).
    \end{equation}
    As a result,
    \begin{equation}
        \rhoin^{\gamma}(\rhoin^{\theta}-\overline{\rho^{\theta}})\in L^{\frac{2\gamma}{\gamma+2\theta}}(0,T;L^{\frac{2\gamma}{\gamma+2\theta}}((0,\rin(t)),r^2dr)).
    \end{equation}
    So
    \begin{equation}
        \begin{aligned}
            (\rhoin^{\theta}-\overline{\rho^{\theta}})^{\frac{2}{\theta}}=&(\rhoin^{\theta}-\overline{\rho^{\theta}})^{\frac{2\gamma}{\gamma+2\theta}}(\rhoin^{\theta}-\overline{\rho^{\theta}})^{\frac{2}{\theta}-\frac{2\gamma}{\gamma+2\theta}}\\
            \le&C(\rhoin^{\theta}-\overline{\rho^{\theta}})^{\frac{2\gamma}{\gamma+2\theta}}\rhoin^{2-\frac{2\gamma\theta}{\gamma+2\theta}}\\
            \le&C(\rhoin^{\theta}-\overline{\rho^{\theta}})^{\frac{2\gamma}{\gamma+2\theta}}(1+\rhoin^{\frac{2\gamma^2}{\gamma+2\theta}})\\
            \le&C(1+\rhoin^{\frac{\gamma^2+2\gamma\theta}{\gamma+2\theta}}+(\rhoin^{\theta}-\overline{\rho^{\theta}})^{\frac{2\gamma}{\gamma+2\theta}}\rhoin^{\frac{2\gamma^2}{\gamma+2\theta}})\\
            =&C(1+\rhoin^{\gamma}+(\rhoin^{\theta}-\overline{\rho^{\theta}})^{\frac{2\gamma}{\gamma+2\theta}}\rhoin^{\frac{2\gamma^2}{\gamma+2\theta}})\in L^1(0,T;L^1((0,\rin(t)),r^2dr)),
        \end{aligned}
    \end{equation}
    where the second "$\le$" relies on $\frac{1}{2}(1-\theta+\sqrt{1+6\theta+\theta^2})\le\gamma$.
\end{proof}
Now, we are ready to prove $\overline{\rho^{\theta}}=\rhoin^{\theta}$ for $0<\theta\le1$. Then, we can apply Lemma \ref{App3.5} to get the following strong convergence of the density.
\begin{lem}
    For any $1\le p<2\gamma$, $\rhoxj\rightarrow\rhoin$ in $L^p_{loc}((0,\rin(t))\times[0,T])$.
\end{lem}
\begin{proof}
    By the first equation in (\ref{WkLimEq1}) and a similar argument as in the proof of \cite[Lemma 2.1, 2.2]{JZ}, we get
    \begin{equation}
        \pt\rhoin^{\theta}+\pr(\uin\rhoin^{\theta})+\frac{2\theta}{r}\uin\rhoin^{\theta}=(1-\theta)\rhoin^{\theta}\pr\uin,
    \end{equation}
    where $\theta$ satisfies the condition in Lemma \ref{wkconvergenceLemma2}. Subtract it from (\ref{WkLimEq2}) to get
    \begin{equation}
        \medmath{\pt(\rhoin^{\theta}-\overline{\rho^{\theta}})+\pr(\uin(\rhoin^{\theta}-\overline{\rho^{\theta}}))+\frac{2\theta}{r}\uin(\rhoin^{\theta}-\overline{\rho^\theta})=(1-\theta)(\rhoin^\theta\pr\uin-\overline{\rho^\theta\pr u})\le(1-\theta)(\rhoin^\theta-\overline{\rho^{\theta}})\pr\uin,}
    \end{equation}
    where the last "$\le$" follows from (\ref{rhoDiff1}), (\ref{rhoDiff2}). Multiply it by $r^{2\theta}$, we get
    \begin{equation}
        \pt f+\theta f\pr\uin+\uin\pr f\le0,
    \end{equation}
    where $f:=r^{2\theta}(\rhoin^\theta-\overline{\rho^\theta})\ge0$. By mollifying and then sending to the limit as in \cite{JZ}, we obtain
    \begin{equation}\label{Estf}
        \pt f^{\frac{1}{\theta}}+\pr(\uin f^{\frac{1}{\theta}})\le0
    \end{equation}
    and
    \begin{equation}
        \pt f^{\frac{1}{\theta}}+\pr(\uin f^{\frac{1}{\theta}})=\frac{1}{\theta}f^{\frac{1}{\theta}-1}(1-\theta)(\overline{\rho^{\theta}}\pr\uin-\overline{\rho^\theta\pr u}),
    \end{equation}
    where 
    \begin{equation}
        \begin{aligned}
            \uin f^{\frac{1}{\theta}}\in L^{\frac{2\gamma}{\gamma+1}}&(0,T;L^{\frac{2\gamma}{\gamma+1}}_{loc}((0,\rin(t)))),\\
            f^{\frac{1}{\theta}-1}\in L^{\frac{2\gamma}{1-\theta}}&(0,T;L^{\frac{2\gamma}{1-\theta}}_{loc}((0,\rin(t)))),\\
            (\overline{\rho^{\theta}}\pr\uin-\overline{\rho^\theta\pr u})\in& L^{\frac{2\gamma}{\gamma+\theta}}(0,T;L^{\frac{2\gamma}{\gamma+\theta}}_{loc}((0,\rin(t)))),
        \end{aligned}
    \end{equation}
    so $\Dt f^{\frac{1}{\theta}}\in L^{\frac{2\gamma}{\gamma+1}}(0,T;W^{-1,\frac{2\gamma}{1+\gamma}}_{loc}((0,\rin(t))))$, which implies $f^{\frac{1}{\theta}}\in C([0,T];W^{-1,\frac{2\gamma}{1+\gamma}}_{loc}((0,\rin(t))))$. Then by Lemma \ref{App3},
    \begin{equation}
        f^{\frac{1}{\theta}}\in C([0,T];L^{\gamma}_{loc}((0,\rin(t)))-w).
    \end{equation}

    To prove to prove $\overline{\rho^{\theta}}=\rhoin^{\theta}$ for $0<\theta\le1$, we first prove $f^{\frac{1}{\theta}}=0\ a.e.$ on $[0,\rin(t))\times[0,T]$ for $\theta>0$ satisfying $\frac{1}{2}(1-\theta+\sqrt{1+6\theta+\theta^2})\le\gamma$. Let $\chi_{\delta}(r,t)=\min\{\delta,(\rin(t)-r)_+\}$, then
    \begin{equation}
        \begin{aligned}
            \pt\chid=&u(\rin(t),t)\One_{\{r\le\rin(t)\le r+\delta\}},\\
            \pr\chid=&-\One_{\{r\le\rin(t)\le r+\delta\}},
        \end{aligned}
    \end{equation}
    where $\delta>0$ is a small number such that $r_{x_1}(t)<\rin(t)-\delta$ for $t\in[0,T]$. Also, we let $\zeta_{\sigma}\in C^{\infty}([0,\rin(t))\times[0,T])$
    \begin{equation}
        \zeta_{\sigma}(r,t)=\begin{cases}
            0,\text{ if }r<\frac{1}{2}\sigma,\\
            1,\text{ if }\sigma\le r\le\rin(t),
        \end{cases}
    \end{equation}
    where $\sigma>0$ is small enough, and $|\pr\zeta_{\sigma}(r)|\le\frac{C}{\sigma}$ for $\frac{1}{2}\sigma<r<\sigma$. So
    \begin{equation}\label{Chidf}
        \pt(\chid^m f^{\frac{1}{\theta}})+\pr(\chid^m\uin f^{\frac{1}{\theta}})=\chid^m(\pt f^{\frac{1}{\theta}}+\pr(\uin f^{\frac{1}{\theta}}))+m\chid^{m-1}f^{\frac{1}{\theta}}(\uin(\rin(t),t)-\uin(r,t)).
    \end{equation}
    Multiply (\ref{Chidf}) by $\zeta_{\sigma}$ and integrate it over $[0,\rin(t))\times[0,T]$, together with (\ref{Estf}), we have
    \begin{equation}\label{IntEstf}
        \begin{aligned}
            &\int_0^t\int_0^{\rin(s)}\partial_s(\chid^mf^{\frac{1}{\theta}})\zeta_\sigma+\pr(\chid^m\uin f^{\frac{1}{\theta}})\zeta_\sigma drds\\
            =&\int_0^t\int_0^{\rin(s)}\chid^m(\pt f^{\frac{1}{\theta}}+\pr(\uin f^{\frac{1}{\theta}}))\zeta_\sigma+m\chid^{m-1}f^{\frac{1}{\theta}}(\uin(\rin(s),s)-\uin(r,s))\zeta_\sigma drds\\
            \le&\int_0^t\int_{\rin(s)-\delta}^{\rin(s)}m\chid^{m-1}f^{\frac{1}{\theta}}|\uin(\rin(s),s)-\uin(r,s)|\zeta_\sigma drds\\
            \le&\int_0^t\int_{\rin(s)-\delta}^{\rin(s)}Cm\delta^{\frac{1}{2}}\chid^{m-1}\zeta_\sigma f^{\frac{1}{\theta}}drds,
        \end{aligned}
    \end{equation}
where we use the regularities of $u_b$ in (\ref{ConvInLag}) to deal with the term $|\uin(\rin(s),s)-\uin(r,s)|$ for $r\in[r_{in}(t)-\delta,r_{in}(t)]$. On the other hand,
    \begin{equation}
    \begin{aligned}
        \int_0^t\int_0^{\rin(s)}\pr(\chid^m\uin f^{\frac{1}{\theta}})\zeta_\sigma drds=&\int_0^t\left(-\int_0^{\sigma}\chid^m\uin f^{\frac{1}{\theta}}\pr\zeta_\sigma dr\right)ds\\
        =&\int_0^t\left(-\int_0^{\sigma}\delta^m\uin f^{\frac{1}{\theta}}\pr\zeta_\sigma dr\right)ds,
    \end{aligned}
    \end{equation}
    which, combining with (\ref{IntEstf}), gives
    \begin{equation}\label{IntEstftwo}
        \int_0^t\int_0^{\rin(s)}\partial_s(\chid^mf^{\frac{1}{\theta}})\zeta_\sigma drds\le\int_0^t\int_0^{\sigma}\delta^m|\uin| f^{\frac{1}{\theta}}|\pr\zeta_\sigma| drds+\int_0^t\int_{\rin(s)-\delta}^{\rin(s)}Cm\delta^{\frac{1}{2}}\chid^{m-1} f^{\frac{1}{\theta}}drds.
    \end{equation}
    Making use of the fact that $f(r,0)=0$, left hand side of (\ref{IntEstftwo}) can be written as $\int_0^{\rin(t)}\chid^mf^{\frac{1}{\theta}}\zeta_\sigma dr$. So
  \begin{equation}
      \int_0^{\rin(t)}\chid^mf^{\frac{1}{\theta}}\zeta_\sigma dr\le\int_0^t\int_0^{\sigma}\delta^m|\uin| f^{\frac{1}{\theta}}|\pr\zeta_\sigma| drds+\int_0^t\int_{\rin(s)-\delta}^{\rin(s)}Cm\delta^{\frac{1}{2}}\chid^{m-1} f^{\frac{1}{\theta}}drds.
  \end{equation}
  Notice that $\frac{1}{r}\uin  f^{\frac{1}{\theta}}=r\uin(\rhoin^{\theta}-\overline{\rho^{\theta}})^{\frac{1}{\theta}}\in L^1([0,\rin(t))\times[0,T])$ and $f^{\frac{1}{\theta}}\in L^2(0,T;L^2((0,\rin(t)),r^2dr))$,
\begin{equation}
\medmath{\begin{aligned}
    \int_0^t\int_0^{\sigma}\delta^m|\uin| f^{\frac{1}{\theta}}|\pr\zeta_\sigma| drds\le& C\delta^m\int_0^t\int_0^{\sigma}\frac{1}{r}|\uin|  f^{\frac{1}{\theta}}drds;\\
    \int_0^t\int_{\rin(s)-\delta}^{\rin(s)}Cm\delta^{\frac{1}{2}}\chid^{m-1} f^{\frac{1}{\theta}}drds\le&C\delta^{m-\frac{1}{2}}\int_0^t\int_{\rin(s)-\delta}^{\rin(s)}f^{\frac{1}{\theta}}drds\\
    \le&C\delta^{m-\frac{1}{2}}\int_0^t\delta^{\frac{1}{2}}\left(\int_{\rin(s)-\delta}^{\rin(s)}f^{\frac{2}{\theta}}dr\right)^{\frac{1}{2}}ds\le C\delta^m\sqrt{T}\left(\int_0^t\int_{\rin(t)-\delta}^{\rin(t)}f^{\frac{2}{\theta}}drds\right)^{\frac{1}{2}},
\end{aligned}}
\end{equation}
then we have
\begin{equation}
    \int_0^{\rin(t)}\delta^{-m}\chid^mf^{\frac{1}{\theta}}\zeta_\sigma dr\le C\int_0^t\int_0^{\sigma}\frac{1}{r}|\uin|  f^{\frac{1}{\theta}}drds+C\sqrt{T}\left(\int_0^t\int_{\rin(t)-\delta}^{\rin(t)}f^{\frac{2}{\theta}}drds\right)^{\frac{1}{2}}.
\end{equation}
Let $\delta\rightarrow0$, we get
\begin{equation}
    \int_0^{\rin(t)}f^{\frac{1}{\theta}}\zeta_\sigma dr\le C\int_0^t\int_0^{\sigma}\frac{1}{r}|\uin|  f^{\frac{1}{\theta}}drds.
\end{equation}
Let $\sigma\rightarrow0$, we have $f^{\frac{1}{\theta}}=0\ a.e.$ on $[0,\rin(t))\times[0,T]$.

Now, notice that for those $\theta>0$ within a small neighborhood of 0, there holds $\frac{1}{2}(1-\theta+\sqrt{1+6\theta+\theta^2})\le\gamma$. Fix such a $\theta_0$, and consider any $\theta_0<\theta\le1$, by convexity,
\begin{equation}
    \overline{\rho^{\theta}}\le\rhoin^{\theta}=\rhoin^{\theta_0\cdot\frac{\theta}{\theta_0}}=\overline{\rho^{\theta_0}}^{\frac{\theta}{\theta_0}}\le\overline{\rho^{\theta}},
\end{equation}
which means $\rhoin^{\theta}=\overline{\rho^\theta}\ a.e.$ holds for any $0<\theta\le1$, i.e.,
\begin{equation}
    (\rhoxj)^{\theta}\rightharpoonup\rhoin^{\theta}\text{ weakly in }L^{\frac{2\gamma}{\theta}}(0,T;L^\frac{2\gamma}{\theta}_{loc}((0,\rin(t)))).
\end{equation}

Suppose $0<\theta\le\frac{1}{2}$, and the Young's measure of $\{(\rhoxj)^\theta\}$ is $\mu((r,t),d\lambda)$, then
\begin{equation}
    \sigma((r,t))=\int\lambda^2\mu((r,t),d\lambda)-\left(\int\lambda\mu((r,t),d\lambda)\right)^2=\overline{\rho^{2\theta}}-(\overline{\rho^\theta})^2=0,
\end{equation}
which means the Young's measure associated with $\{(\rhoxj)^\theta\}$ is the Dirac measure. By Lemma \ref{App4}, 
\begin{equation}
    (\rhoxj)^{\theta}\rightarrow\rhoin^{\theta}\text{ strongly in }L^{p}(0,T;L^p_{loc}((0,\rin(t))))\ (1\le p<\frac{2\gamma}{\theta}).
\end{equation}
Now, for any $1\le p<2\gamma$,
\begin{equation}
    \medmath{\begin{aligned}
        \|\rhoxj-\rhoin\|_{L^p(0,T;L^p((s_1,\rin(t))))}=&\|((\rhoxj)^{\frac{1}{2}}-\rhoin^{\frac{1}{2}})((\rhoxj)^{\frac{1}{2}}+\rhoin^{\frac{1}{2}})\|_{L^p(0,T;L^p((s_1,\rin(t))))}\\
        \le&\|(\rhoxj)^{\frac{1}{2}}-\rhoin^{\frac{1}{2}}\|_{L^{2p}(0,T;L^{2p}((s_1,\rin(t))))}\|(\rhoxj)^{\frac{1}{2}}+\rhoin^{\frac{1}{2}}\|_{L^{2p}(0,T;L^{2p}((s_1,\rin(t))))}\\
        \le&C\|(\rhoxj)^{\frac{1}{2}}-\rhoin^{\frac{1}{2}}\|_{L^{2p}(0,T;L^{2p}((s_1,\rin(t))))}\rightarrow0\text{ as }\xi_j\rightarrow0,
    \end{aligned}}
\end{equation}
as desired.
\end{proof}
In particular, we have
\begin{equation}
\begin{aligned}
     &\rhoxj\rightarrow\rhoin\text{ in }L^1(0,T;L^1((0,\rin(t)),r^2dr)),\\
     &\rhoxj\uxj\rightharpoonup\rhoin\uin\text{ in }\Linf(0,T;L^{\frac{2\gamma}{\gamma+1}}((0,\rin(t)),r^2dr)).
\end{aligned}
\end{equation}

\subsubsection{Verification of weak solutions}
Finally, let
\begin{equation}\label{solnRad1}
    (\rho, u)=\begin{cases}
        (\rho_b,u_b)(r,t),\ r_{x_1}(t)\le r\le a(t),\ t\in[0,T],\\
        
        (\rhoin,\uin)(r,t),\ 0\le r\le\rin(t),\ t\in[0,T],
    \end{cases}
\end{equation}
where by choosing a proper subsequence, we can make 
\begin{equation}\label{solnRad2}
    (\rho_b,u_b)=(\rhoin,\uin)\ a.e.\ (r,t)\in[r_{x_1}(t),\rin(t)]\times[0,T],
\end{equation}
and we let
\begin{equation}\label{soln}
    (\rho,\u)(\bx,t)=(\rho,u\frac{\bx}{|\bx|})(|\bx|,t),
\end{equation}
we directly verify that (\ref{soln}) is indeed a weak solution as in Remark \ref{WeakSolution3D}, which will automatically imply (\ref{solnRad1}) is a weak solution satisfying Definition \ref{WeakSolutionRadSymmetry}.

For the mass equation, choose any $\varphi(\bx,t)\in C_c^{\infty}(\Omega_t\times[0,T])$, let 
$$\zeta=\frac{1}{r^2}\int_{\partial B(0;r)}\varphi(\bx,t)dS_r(x)=\int_{\partial B(0;1)}\varphi(r\boldsymbol{y},t)dS_1(y),$$
then $\zeta(r,t)\in C^{\infty}_c([0,a(t))\times[0,T])$. Let $\zeta=\zeta_1+\zeta_2$, where $\zeta_1,\ \zeta_2\in C^{\infty}_c([0,a(t))\times[0,T])$, and
\begin{equation}\label{zeta12}
    \supp\zeta_1\subset[0,\rin(t))\times[0,T],\ \supp\zeta_2\subset(r_{x_1}(t),a(t))\times[0,T].
\end{equation}
Since $(\rhoxj,\uxj)$ satisfies the mass equation in (\ref{nsp3}) on $[\xi_j,a^{\xi_j}(t))\times[0,T]$, we multiply it by $r^2\zeta_1$ and integrate by parts, making use of the fact that $(\rhoxj,\uxj)=(0,0)$ on $[0,\xi_j)\times[0,T]$ to get
\begin{equation}\label{solnRadin}
    \int_0^{\rin(t)}\rhoxj\zeta_1 r^2dr|_{t_1}^{t_2}-\int_{t_1}^{t_2}\int_0^{\rin(t)}(\rhoxj\pt\zeta_1+\rhoxj\uxj\pr\zeta_1)r^2drdt=0.
\end{equation}
Since $\Dt\rhoxj\in\Linf(0,T;W^{-1,\frac{2\gamma}{\gamma+1}}_{loc}((0,r_{in}(t))))$, by Lemma \ref{App3}, 
\begin{equation}
    \lim_{\xi_j\rightarrow0} \int_0^{\rin(t)}\rhoxj\zeta_1 r^2dr= \int_0^{\rin(t)}\rhoin\zeta_1 r^2dr\in C([0,T]).
\end{equation}
Send $\xi_j\rightarrow0$ to get
\begin{equation}
    \int_0^{\rin(t)}\rhoin\zeta_1 r^2dr|_{t_1}^{t_2}-\int_{t_1}^{t_2}\int_0^{\rin(t)}(\rhoin\pt\zeta_1+\rhoin\uin\pr\zeta_1)r^2drdt=0.
\end{equation}
On the other hand, notice that
\begin{equation}
    \ptau\rhoxj+4\pi(\rhoxj)^2(r^{\xi_j})^2\px \uxj+\frac{2\rhoxj\uxj}{r^{\xi_j}}=0,
\end{equation}
using the coordinates change $\ptau=\pt+u_b\partial_{r},\ \px=\frac{1}{4\pi\rho_br^2}\pr$, where $x=M-4\pi\int_r^{a(t)}\rho_b(s,t)s^2ds$, we get
\begin{equation}
    \pt\rhoxj+u_b\pr\rhoxj+\frac{(\rhoxj)^2(r^{\xi_j})^2}{\rho_b r^2}\pr\uxj+\frac{2\rhoxj\uxj}{r^{\xi_j}}=0,
\end{equation}
which implies
\begin{equation}\label{SolnBdary1}
    r^2\pt\rhoxj+\pr(r^2\rhoxj\uxj)-r^2(\frac{2\rhoxj\uxj}{r}-\frac{2\rhoxj\uxj}{r^{\xi_j}})-r^2(\uxj-u_b)\pr\rhoxj-\pr\uxj(\rhoxj r^2-\frac{(\rhoxj)^2(r^{\xi_j})^2}{\rho_b})=0.
\end{equation}
Multiply (\ref{SolnBdary1}) by $\zeta_2$, integrate it by parts to get
\begin{equation}\label{SolnBdary2}
        \medmath{\begin{aligned}
            &\int_{r_{x_1}(t)}^{a(t)}\rhoxj\zeta_2r^2drdt|_{t_1}^{t_2}-\int_{t_1}^{t_2}\int_{r_{x_1}(t)}^{a(t)}(\rhoxj\pt\zeta_2+\rhoxj\uxj\pr\zeta_2)r^2drdt\\
            +&\int_{t_1}^{t_2}\int_{r_{x_1}(t)}^{a(t)}-r^2(\frac{2\rhoxj\uxj}{r}-\frac{2\rhoxj\uxj}{r^{\xi_j}})\zeta_2-r^2(\uxj-u_b)\pr\rhoxj\zeta_2-\pr\uxj(\rhoxj r^2-\frac{(\rhoxj)^2(r^{\xi_j})^2}{\rho_b})\zeta_2drdt=0.
        \end{aligned}}
\end{equation}
When $\xi_j\rightarrow0$, the first term in the second line of (\ref{SolnBdary2}) goes to 0. Also, since
\begin{equation}
    \begin{aligned}
        &\left|\int_{t_1}^{t_2}\int_{r_{x_1}(t)}^{a(t)}\zeta_2r^2\pr\rhoxj(u_b-\uxj)drdt\right|\\
        =&\left|\int_{t_1}^{t_2}\int_{x_1}^M\zeta_2r^2\px\rhoxj(u_b-\uxj)dxd\tau\right|\\
        \le& C\int_{t_1}^{t_2}\left(\int_{x_1}^M(\px\rhoxj)^2dx\right)^{\frac{1}{2}}\left(\int_{x_1}^M|u_b-\uxj|^2dx\right)^{\frac{1}{2}}d\tau\rightarrow0,
    \end{aligned}
\end{equation}
\begin{equation}
    \begin{aligned}
        &\left|\int_{t_1}^{t_2}\int_{r_{x_1}(t)}^{a(t)}\zeta_2\pr\uxj\left(\frac{(\rhoxj)^2(r_{\xi_j})^2}{\rho_b}-r^2\rhoxj\right)drdt\right|\\
        =&\left|\int_{t_1}^{t_2}\int_{x_1}^{M}\zeta_2\px\uxj\left(\frac{(\rhoxj)^2(r_{\xi_j})^2}{\rho_b}-r^2\rhoxj\right)dxd\tau\right|\\
        \le&\int_{t_1}^{t_2}\left(\int_{x_1}^M\rhoxj(r^{\xi_j})^4(\px\uxj)^2dx\right)^{\frac{1}{2}}\left(\int_{x_1}^M\zeta_2^2\left(\frac{(\rhoxj)^{\frac{3}{2}}}{\rho_b}-\frac{r^2}{(r^{\xi_j})^2}(\rhoxj)^{\frac{1}{2}}\right)^2dx\right)^{\frac{1}{2}}d\tau\rightarrow0,    \end{aligned}
\end{equation}
where we use the strong convergence of $(\rhoxj,\uxj)$ and the fact that $\zeta_2$ is compactly supported in $[r_{x_1}(t),a(t))$ for each fixed $t\in[0,T]$. As a result, we obtain
\begin{equation}\label{solnRadBdry}
    \int_{r_{x_1}(t)}^{a(t)}\rho_b\zeta_2r^2drdt|_{t_1}^{t_2}-\int_{t_1}^{t_2}\int_{r_{x_1}(t)}^{a(t)}(\rho_b\pt\zeta_2+\rho_bu_b\pr\zeta_2)r^2drdt=0.
\end{equation}
Combining (\ref{solnRadin}) and (\ref{solnRadBdry}), we get
\begin{equation}\label{solnRadmass}
    \int_{0}^{a(t)}\rho\zeta r^2drdt|_{t_1}^{t_2}-\int_{t_1}^{t_2}\int_{0}^{a(t)}(\rho\pt\zeta+\rho u\pr\zeta)r^2drdt=0.
\end{equation}
Thus we have
\begin{equation}\label{FinalRs1}
    \int_{\Omega_t}\rho\varphi\dbx|_{t_1}^{t_2}=\int_{t_1}^{t_2}\int_{\Omega_t}\rho\partial_t\varphi+\rho\u\cdot\nabla_x\varphi\dbx dt.
\end{equation}
By a density argument, (\ref{FinalRs1}) holds for any $\varphi\in C^1_c([0,T]\times{\Omega_t})$.

For the momentum equation, similarly, we choose any test function $\boldsymbol{\psi}(\bx,t)\in(C^{\infty}_c([0,a(t))\times[0,T]))^3$ and let
$$\zeta(r,t)=\frac{1}{r^2}\int_{\partial B(0;r)}\boldsymbol{\psi}(\bx,t)\cdot\boldsymbol{\omega}dS_r(x)=\int_{\partial B(0;1)}\boldsymbol{\psi}(r\boldsymbol{y},t)\cdot\boldsymbol{y}dS_1(y),$$
then $\zeta(r,t)\in C^{\infty}_c([0,a(t))\times[0,T])$. By divergence theorem,
\begin{equation}
    |\zeta(r,t)|=\left|\frac{1}{r^2}\int_{B(0;r)}\dv\boldsymbol{\psi}(\bx,t)\dbx\right|\le Cr\|\boldsymbol{\psi}\|_{C^1}.
\end{equation}
Write $\zeta=\zeta_1+\zeta_2$, where $\zeta_1,\ \zeta_2$ satisfy (\ref{zeta12}). Since $(\rhoxj,\uxj)$ satisfies the momentum equation in (\ref{nsp3}) on $[\xi_j,a^{\xi_j}(t))\times[0,T]$, we multiply the momentum equation by $r^2\zeta_1$, integrate by parts and making use of the fact that $(\rhoxj,\uxj)=(0,0)$ on $[0,\xi_j)\times[0,T]$ to get
\begin{equation}
    \medmath{\begin{aligned}
        &\int_0^{\rin(t)}\rhoxj\uxj r^2\zeta_1dr|_{t_1}^{t_2}-\int_{t_1}^{t_2} \int_0^{\rin(t)}\rhoxj\uxj\pt\zeta_1r^2drdt\\&-\int_{t_1}^{t_2} \int_0^{\rin(t)}(\rhoxj)^\gamma(\pr\zeta_1+\frac{2\zeta_1}{r})r^2drdt-\int_{t_1}^{t_2} \int_0^{\rin(t)}\rhoxj(\uxj)^2\pr\zeta_1r^2drdt\\
        =&-\int_{t_1}^{t_2} \int_0^{\rin(t)}4\pi\rhoxj\zeta_1\int_0^r\rhoxj(s,t)s^2dsdrdt-\int_{t_1}^{t_2} \int_0^{\rin(t)}(\eta+\frac{4}{3}\varepsilon)(\pr\uxj+\frac{2\uxj}{r})(\pr\zeta_1+\frac{2\zeta_1}{r})r^2drdt.
    \end{aligned}}
\end{equation}
Integrating by parts, we have
\begin{equation}
\medmath{\begin{aligned}
    &\int_{t_1}^{t_2} \int_0^{\rin(t)}(\eta+\frac{4}{3}\varepsilon)(\pr\uxj+\frac{2\uxj}{r})(\pr\zeta_1+\frac{2\zeta_1}{r})r^2drdt\\
    =&\int_{t_1}^{t_2} \int_0^{\rin(t)}(\eta-\frac{2}{3}\varepsilon)(\pr\uxj+\frac{2\uxj}{r})(\pr\zeta_1+\frac{2\zeta_1}{r})r^2drdt+\int_{t_1}^{t_2} \int_0^{\rin(t)}2\varepsilon(\pr\uxj+\frac{2\uxj}{r})(\pr\zeta_1+\frac{2\zeta_1}{r})r^2drdt\\
    =&\int_{t_1}^{t_2} \int_0^{\rin(t)}(\eta-\frac{2}{3}\varepsilon)(\pr\uxj+\frac{2\uxj}{r})(\pr\zeta_1+\frac{2\zeta_1}{r})r^2drdt+\int_{t_1}^{t_2} \int_0^{\rin(t)}2\varepsilon(\pr\uxj\pr\zeta_1+\frac{2\uxj\zeta_1}{r^2})r^2dr,
\end{aligned}}
\end{equation}
thus
\begin{equation}\label{approxMomentum}
    \medmath{\begin{aligned}
        &\int_0^{\rin(t)}\rhoxj\uxj r^2\zeta_1dr|_{t_1}^{t_2}-\int_{t_1}^{t_2} \int_0^{\rin(t)}\rhoxj\uxj\pt\zeta_1r^2drdt\\&-\int_{t_1}^{t_2} \int_0^{\rin(t)}(\rhoxj)^\gamma(\pr\zeta_1+\frac{2\zeta_1}{r})r^2drdt-\int_{t_1}^{t_2} \int_0^{\rin(t)}\rhoxj(\uxj)^2\pr\zeta_1r^2drdt\\
        =&-\int_{t_1}^{t_2} \int_0^{\rin(t)}4\pi\rhoxj\zeta_1\int_0^r\rhoxj(s,t)s^2dsdrdt-\int_{t_1}^{t_2} \int_0^{\rin(t)}(\eta-\frac{2}{3}\varepsilon)(\pr\uxj+\frac{2\uxj}{r})(\pr\zeta_1+\frac{2\zeta_1}{r})r^2drdt\\
        &-\int_{t_1}^{t_2} \int_0^{\rin(t)}2\varepsilon(\pr\uxj\pr\zeta_1+\frac{2\uxj\zeta_1}{r^2})r^2dr.
    \end{aligned}}
\end{equation}
Since $\Dt(\rhoxj\uxj)\in L^2(0,T;W^{-1,1}_{loc}((0,r_{in}(t))))$, by Lemma \ref{App3},
\begin{equation}
    \lim_{\xi_j\rightarrow0}\int_0^{\rin(t)}\rhoxj\uxj\zeta_1r^2dr=\int_0^{\rin(t)}\rhoin u_{in}\zeta_1r^2dr\in C([0,T]).
\end{equation}
To send $\xi_j\rightarrow0$, we still need to check the gravitational term $\int_{t_1}^{t_2} \int_0^{\rin(t)}4\pi\rhoxj\zeta_1\int_0^r\rhoxj(s,t)s^2dsdrdt$. Notice the fact that $\rhoxj\rightarrow\rho_{in}$ strongly in $L^1(0,T;L^1((0,\rin(t)),r^2dr))$, the same convergence also holds for $\int_0^r\rhoxj(s,t)s^2ds$. Thus we can find a subsequence such that
\begin{equation}
    \begin{aligned}
        \rhoxj\rightarrow\rho_{in}\ a.e.\text{ in }&[0,\rin(t))\times[0,T],\\
        \int_0^r\rhoxj(s,t)s^2ds\rightarrow\int_0^r\rhoin(s,t)&s^2ds\ a.e.\text{ in }[0,\rin(t))\times[0,T].
    \end{aligned}
\end{equation}
Since
\begin{equation}
\begin{aligned}
    &\int_{t_1}^{t_2} \int_0^{\rin(t)}4\pi\rhoxj\zeta_1\int_0^r\rhoxj(s,t)s^2dsdrdt\\
   \le& C \int_{t_1}^{t_2} \int_0^{\rin(t)}4\pi\rhoxj r\int_0^r\rhoxj(s,t)s^2dsdrdt\le C(E_0,M,T),
\end{aligned}
\end{equation}
by dominated convergence theorem,
\begin{equation}
    \int_{t_1}^{t_2} \int_0^{\rin(t)}4\pi\rhoxj\zeta_1\int_0^r\rhoxj(s,t)s^2dsdrdt\rightarrow \int_{t_1}^{t_2} \int_0^{\rin(t)}4\pi\rhoin\zeta_1\int_0^r\rhoin(s,t)s^2dsdrdt.
\end{equation}
Now we send $\xi_j\rightarrow0$ in (\ref{approxMomentum}), we have
\begin{equation}\label{SolnIn}
    \medmath{\begin{aligned}
        &\int_0^{\rin(t)}\rhoin\uin r^2\zeta_1dr|_{t_1}^{t_2}-\int_{t_1}^{t_2} \int_0^{\rin(t)}\rhoin\uin\pt\zeta_1r^2drdt\\&-\int_{t_1}^{t_2} \int_0^{\rin(t)}(\rhoin)^\gamma(\pr\zeta_1+\frac{2\zeta_1}{r})r^2drdt-\int_{t_1}^{t_2} \int_0^{\rin(t)}\rhoin(\uin)^2\pr\zeta_1r^2drdt\\
        =&-\int_{t_1}^{t_2} \int_0^{\rin(t)}4\pi\rhoin\zeta_1\int_0^r\rhoin(s,t)s^2dsdrdt-\int_{t_1}^{t_2} \int_0^{\rin(t)}(\eta-\frac{2}{3}\varepsilon)(\pr\uin+\frac{2\uin}{r})(\pr\zeta_1+\frac{2\zeta_1}{r})r^2drdt\\
        &-\int_{t_1}^{t_2} \int_0^{\rin(t)}2\varepsilon(\pr\uin\pr\zeta_1+\frac{2\uin\zeta_1}{r^2})r^2drdt.
    \end{aligned}}
\end{equation}
On the other hand, notice that
\begin{equation}
    \ptau\uxj+4\pi(r^{\xi_j})^2\px(\rhoxj)^{\gamma}+\frac{x}{(r^{\xi_j})^2}=16\pi^2(r^{\xi_j})^2\px\left((\eta+\frac{4}{3}\varepsilon)\rhoxj\px((r^{\xi_j})^2\uxj)\right),
\end{equation}
using the coordinates change $\ptau=\pt+u_b\partial_{r},\ \px=\frac{1}{4\pi\rho_br^2}\pr$, where $x=M-4\pi\int_r^{a(t)}\rho_b(s,t)s^2ds$, we get
\begin{equation}\label{SolnBdary3}
\begin{aligned}
    &\pt(\rhoxj\uxj)+\pr\left(\rhoxj(\uxj)^2+(\rhoxj)^{\gamma}\right)+\frac{2\rhoxj(\uxj)^2}{r^{\xi_j}}+\frac{\rho_b}{(r^{\xi_j})^2}\left(M-\int_r^{a(t)}4\pi\rho_bs^2ds\right)\\
    &+(\rho_b-\rhoxj)\pt\uxj+(\rho_bu_b-\rhoxj\uxj)\pr\uxj=\frac{(r^{\xi_j})^2}{r^2}\pr\left((\eta+\frac{4}{3}\varepsilon)\frac{\rhoxj}{\rho_br^2}\pr((r^{\xi_j})^2\uxj)\right),
\end{aligned}
\end{equation}
multiply (\ref{SolnBdary3}) by $r^2\zeta_2$ and integrate by parts to get
\begin{equation}\label{SolnBdary4}
    \medmath{\begin{aligned}
        &\int_{r_{x_1}(t)}^{a(t)}\rhoxj\uxj r^2\zeta_2|_{t_1}^{t_2}
        -\int_{t_1}^{t_2}\int_{r_{x_1}(t)}^{a(t)}\rhoxj\uxj r^2\pt\zeta_2drdt\\
        &-\int_{t_1}^{t_2}\int_{r_{x_1}(t)}^{a(t)}(\rhoxj)^{\gamma}(\pr\zeta_2+\frac{2\zeta_2}{r})r^2+\rhoxj(\uxj)^2\pr\zeta_2r^2+\rhoxj(\uxj)^2(2r\zeta_2-\frac{2r^2\zeta_2}{r^{\xi_j}})drdt\\
        &+\int_{t_1}^{t_2}\int_{r_{x_1}(t)}^{a(t)}r^2\zeta_2(\rho_b-\rhoxj)\pt\uxj+r^2\zeta_2(\rho_bu_b-\rhoxj\uxj)\pr\uxj drdt\\
        =&-\int_{t_1}^{t_2}\int_{r_{x_1}(t)}^{a(t)}\frac{r^2\rho_b}{(r^{\xi_j})^2}\zeta_2\left(M-\int_r^{a(t)}4\pi\rho_bs^2ds\right)drdt\\
        &-\int_{t_1}^{t_2}\int_{r_{x_1}(t)}^{a(t)}(\eta+\frac{4}{3}\varepsilon)\frac{\rhoxj}{\rho_b}\left(\frac{(r^{\xi_j})^2}{r^2}\pr\uxj+\frac{2\rho_b}{\rhoxj r^{\xi_j}}\uxj\right)\left(\frac{(r^{\xi_j})^2}{r^2}\pr\zeta_2+\frac{2\rho_b}{\rhoxj r^{\xi_j}}\zeta_2\right)r^2drdt.
    \end{aligned}}
\end{equation}
For the third line in (\ref{SolnBdary4}), we can estimate as follow
\begin{equation}
    \begin{aligned}
        &\left|\int_{t_1}^{t_2}\int_{r_{x_1}(t)}^{a(t)}r^2\zeta_2(\rho_b-\rhoxj)\pt\uxj+r^2\zeta_2(\rho_bu_b-\rhoxj\uxj)\pr\uxj drdt\right|\\
        =&\left|\int_{t_1}^{t_2}\int_{x_1}^M\frac{1}{4\pi}\zeta_2(1-\frac{\rhoxj}{\rho_b})\ptau\uxj+r^2\zeta_2\rhoxj\px\uxj(u_b-\uxj)dxd\tau\right|\\
        \le&C\|\zeta_2(1-\frac{\rhoxj}{\rho_b})\|_{C([x_1,M]\times[0,T])}\int_0^T\int_{x_1}^M|\ptau\uxj|dxd\tau\\&+C\left(\int_0^T\int_{x_1}^M\rhoxj r^4(\px\uxj)^2dxd\tau\right)^{\frac{1}{2}}\|u_b-\uxj\|_{C([0,T];L^2([x_1,M]))}\rightarrow0\text{ as $\xi_j\rightarrow0$}.
    \end{aligned}
\end{equation}

Now, making use of the strong convergence of $\rhoxj$ and $\uxj$ in $[x_1,M]\times[0,T]$ and (\ref{ConvInLag}), sending $\xi_j\rightarrow0$, we have
\begin{equation}\label{SolnBdary5}
    \medmath{\begin{aligned}
        &\int_{r_{x_1}(t)}^{a(t)}\rho_bu_b r^2\zeta_2|_{t_1}^{t_2}
        -\int_{t_1}^{t_2}\int_{r_{x_1}(t)}^{a(t)}\rho_bu_b r^2\pt\zeta_2drdt-\int_{t_1}^{t_2}\int_{r_{x_1}(t)}^{a(t)}\rho_b^{\gamma}(\pr\zeta_2+\frac{2\zeta_2}{r})r^2+\rho_bu_b^2\pr\zeta_2r^2drdt\\
        =&-\int_{t_1}^{t_2}\int_{r_{x_1}(t)}^{a(t)}\rho_b\zeta_2\left(M-\int_r^{a(t)}4\pi\rho_bs^2ds\right)drdt\\
        &-\int_{t_1}^{t_2}\int_{r_{x_1}(t)}^{a(t)}(\eta+\frac{4}{3}\varepsilon)\left(\pr u_b+\frac{2u_b}{r}\right)\left(\pr\zeta_2+\frac{2\zeta_2}{r}\right)r^2drdt,
    \end{aligned}}
\end{equation}
integrate by parts, we get
\begin{equation}\label{SolnBdary6}
    \medmath{\begin{aligned}
        &\int_{r_{x_1}(t)}^{a(t)}\rho_bu_b r^2\zeta_2|_{t_1}^{t_2}
        -\int_{t_1}^{t_2}\int_{r_{x_1}(t)}^{a(t)}\rho_bu_b r^2\pt\zeta_2drdt-\int_{t_1}^{t_2}\int_{r_{x_1}(t)}^{a(t)}\rho_b^{\gamma}(\pr\zeta_2+\frac{2\zeta_2}{r})r^2+\rho_bu_b^2\pr\zeta_2r^2drdt\\
        =&-\int_{t_1}^{t_2}\int_{r_{x_1}(t)}^{a(t)}\rho_b\zeta_2\left(M-\int_r^{a(t)}4\pi\rho_bs^2ds\right)drdt-\int_{t_1}^{t_2}\int_{r_{x_1}(t)}^{a(t)}(\eta-\frac{2}{3}\varepsilon)\left(\pr u_b+\frac{2u_b}{r}\right)\left(\pr\zeta_2+\frac{2\zeta_2}{r}\right)r^2drdt\\
        &-\int_{t_1}^{t_2}\int_{r_{x_1}(t)}^{a(t)}2\varepsilon\left(\pr u_b\pr\zeta_2+\frac{2u_b\zeta_2}{r^2}\right)r^2drdt.
    \end{aligned}}
\end{equation}
Combining (\ref{SolnIn}) and (\ref{SolnBdary6}), we have
\begin{equation}
    \medmath{\begin{aligned}
        &\int_0^{a(t)}\rho u r^2\zeta dr|_{t_1}^{t_2}-\int_{t_1}^{t_2} \int_0^{a(t)}\rho u\pt\zeta r^2drdt\\&-\int_{t_1}^{t_2} \int_0^{a(t)}\rho^\gamma(\pr\zeta+\frac{2\zeta}{r})r^2drdt-\int_{t_1}^{t_2} \int_0^{a(t)}\rho u^2\pr\zeta r^2drdt\\
        =&-\int_{t_1}^{t_2} \int_0^{a(t)}4\pi\rho\zeta\int_0^r\rho(s,t)s^2dsdrdt-\int_{t_1}^{t_2} \int_0^{a(t)}(\eta-\frac{2}{3}\varepsilon)(\pr u+\frac{2u}{r})(\pr\zeta+\frac{2\zeta}{r})r^2drdt\\
        &-\int_{t_1}^{t_2} \int_0^{a(t)}2\varepsilon(\pr u\pr\zeta+\frac{2u\zeta}{r^2})r^2drdt,
    \end{aligned}}
\end{equation}
which implies
\begin{equation}\label{FinalRs2}    
        \begin{aligned}                 &\int_{\Omega_t}\rho\u\cdot\boldsymbol{\psi}\dbx|_{t_1}^{t_2}-\int_{t_1}^{t_2}\int_{\Omega_t}\rho\u\cdot\boldsymbol{\psi}+\rho\u\otimes\u:\nabla\boldsymbol{\psi}+\rho^{\gamma}div\boldsymbol{\psi}\dbx dt \\
        =&-\int_{t_1}^{t_2}\int_{\Omega_t}\varepsilon(\nabla\u+\nabla\u^T):\nabla\boldsymbol{\psi}+(\eta-\frac{2}{3}\varepsilon)div\u\cdot div\boldsymbol{\psi}\dbx dt-\int_{t_1}^{t_2}\int_{\Omega_t}\rho\nabla\varPhi\cdot\boldsymbol{\psi}\dbx dt.
        \end{aligned}        
    \end{equation}
    By a density argument, (\ref{FinalRs2}) holds for any $\boldsymbol{\psi}=(\psi_1,\psi_2,\psi_3)\in (C^1_c(\Omega_t\times[0,T]))^3$.
\subsection{Proof of Theroem \ref{Thm1: ExistenceOfGlobalWkSoln} and Theorem \ref{Thm2}}
    \begin{proof}[Proof of Theorem \ref{Thm1: ExistenceOfGlobalWkSoln}]
       Let $T>0$ and $\gamma\in(\frac{6}{5},\frac{4}{3}]$.
    Suppose $(\rho_0(r),u_0(r))$ are initial data satisfying (\ref{conditionsI}) and Assumption \ref{cond11}, we can construct a radially symmetric weak solution $(\rho,u,a)$ to the equations (\ref{nsp3})-(\ref{boundary conditions}) as (\ref{solnRad1}), which satisfies the equations in the sense of Definition \ref{WeakSolutionRadSymmetry}.
    \end{proof}
    \begin{proof}[Proof of Theorem \ref{Thm2}]
         Suppose $(\rho,u,a)$ is a weak solution to (\ref{nsp3})-(\ref{boundary conditions}) with initial data $(\rho_0,u_0)$ satisfying the conditions in Theorem \ref{Thm1: ExistenceOfGlobalWkSoln}, then (\ref{Thm2: 1}) follows from Lemma \ref{HigherEnergyEst3}, \ref{HigherEnergyEst4}, \ref{StrConvOfrho}, \ref{StrConvOfu} and Remark \ref{HigherEnEst5}, \ref{StrConvOfu: Int}; (\ref{Thm2: 2}) and (\ref{Thm2: 3}) follows from Lemma \ref{HigherEnergyEst2}, \ref{HigherEnergyEst3}, \ref{HigherEnergyEst4}.
    \end{proof}
    \section{Expanding rate of the free boundary}
\label{Expanding rate of the free boundary}
In this section, the goal is to prove the algebraic expanding rate of the free boundary.

We assume $(\rho,u,a)$ is a radially symmetric strong solution to the following radially symmetric 3D Navier-Stokes-Poisson equations
\begin{equation}\label{nspViscDen}
    \begin{cases}
        \partial_t\rho+\partial_r(\rho u)+\frac{2\rho u}{r}=0,\\
        \partial_t(\rho u)+\partial_r(\rho u^2+\rho^{\gamma})+\frac{2\rho u^2}{r}=\partial_r((\eta+\frac{4}{3}\varepsilon))\rho^{\alpha}(\partial_r u+\frac{2u}{r}))-2\varepsilon\pr\rho^{\alpha}\cdot\frac{2u}{r}-\frac{4\pi\rho}{r^2}\int_0^r\rho s^2 ds
    \end{cases}
\end{equation}
 under the assumptions of Corollary \ref{cor1}, with $0\le\alpha\le\gamma$ and $\gamma\in(\frac{6}{5},\frac{4}{3})$. In particular, when $\alpha=0$, the equations are just (\ref{nsp3}). In this case, we also discuss the situation for $\gamma=\frac{4}{3}$.

Let 
\begin{equation}
    \sigma(a(t),t)=0,\ \rho(a(t),t)=0,\ u(0,t)=0,
\end{equation}
where $\sigma:=\rho^{\gamma}-(\eta+\frac{4}{3}\varepsilon)\rho^{\alpha}(\partial_ru+\frac{2u}{r})$.

Suppose $\rho_0\in L^1([0,a(t)),r^2dr)\cap L^{\gamma}([0,a(t)),r^2dr)$, $\sqrt{\rho_0}u_0\in L^2([0,a(t)),r^2dr)$ and $(\rho_0,u_0)\in\mathcal{I}$, then, for $\alpha>0$, we have the following basic energy estimate
\begin{equation}\label{BasicEnEstVsc}
    \begin{aligned}
        \dt E(t)=&\dt\left(4\pi\int_0^{a(t)}\frac{1}{2}\rho u^2 r^2+\frac{1}{\gamma-1}\rho^{\gamma} r^2dr-4\pi\int_0^{a(t)}4\pi\rho r\left(\int_0^r\rho s^2ds\right)dr\right)\\
        =&-4\pi \int_0^{a(t)}\eta\rho^{\alpha}(\partial_ru+\frac{2u}{r})^2r^2+\frac{4}{3}\varepsilon\rho^{\alpha}(\partial_ru-\frac{u}{r})^2 r^2dr.
    \end{aligned}
\end{equation}
Also, similar to Lemma \ref{LinLemma} and Corollary \ref{cor1}, we can control the gravitational energy by the potential energy of the basic energy. And under the flow of the equations (\ref{nspViscDen}), $Q(\rho(t))\ge\Lambda$ for any $t>0$ for some $\Lambda>0$.

To derive the expanding rate of the free boundary, consider
\begin{equation}
    H(t)=\frac{1}{2}\intb\rho|\bx|^2\dbx=2\pi\int_0^{a(t)}\rho r^4dr.
\end{equation}
Take the time derivative, we have
\begin{equation}
\begin{aligned}
    H'(t)&=2\pi\partial_t\int_0^{a(t)}\rho r^4dr=2\pi\int_0^{a(t)}\partial_t\rho r^4dr\\&=2\pi\int_0^{a(t)}-\partial_r(\rho u)r^4-\frac{2\rho u}{r}r^4dr\\
    &=2\pi\int_0^{a(t)}\rho u\cdot4r^3-2\rho ur^3dr=4\pi\int_0^{a(t)}\rho ur^3dr.
\end{aligned}
\end{equation}
Take one more time derivative, for $\alpha=0$, we have
\begin{equation}
    \medmath{\begin{aligned}
        H''(t)=&4\pi\partial_t\int_0^{a(t)}\rho ur^3dr=4\pi\int_0^{a(t)}\partial_t(\rho u)r^3dr\\
        =&4\pi\int_0^{a(t)}\left(-\partial_r(\rho u^2+\rho^{\gamma})r^3-2\rho u^2r^2+r^3\partial_r((\eta+\frac{4}{3}\varepsilon)(\partial_ru+\frac{2u}{r}))-4\pi\rho r\int_0^r\rho s^2ds\right)dr\\
        =&4\pi\int_0^{a(t)}\left(3(\rho u^2+\rho^{\gamma})r^2-2\rho u^2r^2-3r^2(\eta+\frac{4}{3}\varepsilon)(\partial_ru+\frac{2u}{r})-4\pi\rho r\int_0^r\rho s^2ds\right)dr\\
        =&4\pi\int_0^{a(t)}\rho u^2r^2dr+\left(4\pi\int_0^{a(t)}3\rho^{\gamma}r^2dr-4\pi\int_0^{a(t)}4\pi\rho r\int_0^r\rho s^2dsdr\right)-12\pi\int_0^{a(t)}(\eta+\frac{4}{3}\varepsilon)\partial_r(r^2u)dr,
        \end{aligned}}
       \end{equation}
and for $\alpha>0$, we have
\begin{equation}
    \medmath{\begin{aligned}
        H''(t)=&4\pi\int_0^{a(t)}\left(-\partial_r(\rho u^2+\rho^{\gamma})r^3-2\rho u^2r^2+r^3\partial_r((\eta+\frac{4}{3}\varepsilon)\rho^{\alpha}(\partial_ru+\frac{2u}{r}))-2\varepsilon\pr\rho^{\alpha}\cdot2ur^2-4\pi\rho r\int_0^r\rho s^2ds\right)dr\\
        =&4\pi\int_0^{a(t)}\left(3(\rho u^2+\rho^{\gamma})r^2-2\rho u^2r^2-3r^2(\eta+\frac{4}{3}\varepsilon)\rho^{\alpha}(\partial_ru+\frac{2u}{r})+4\varepsilon\rho^{\alpha}(\pr u+\frac{2u}{r})r^2-4\pi\rho r\int_0^r\rho s^2ds\right)dr\\
        =&4\pi\int_0^{a(t)}\rho u^2r^2dr+\left(4\pi\int_0^{a(t)}3\rho^{\gamma}r^2dr-4\pi\int_0^{a(t)}4\pi\rho r\int_0^r\rho s^2dsdr\right)-12\pi\int_0^{a(t)}\eta\rho^{\alpha}\partial_r(r^2u)dr,
        \end{aligned}}
       \end{equation}
       which gives
       \begin{equation}\label{H''}
           H''(t)=\begin{cases}
                    4\pi\int_0^{a(t)}\rho u^2r^2dr+Q(\rho(t))-12\pi\int_0^{a(t)}(\eta+\frac{4}{3}\varepsilon)\partial_r(r^2u)dr\ (\alpha=0),
                         \\
                    4\pi\int_0^{a(t)}\rho u^2r^2dr+Q(\rho(t))-12\pi\int_0^{a(t)}\eta\rho^{\alpha}\partial_r(r^2u)dr\ (\alpha>0).
                \end{cases}
        \end{equation}
\subsection{Proof of Theorem \ref{Thm3}}
\begin{proof}[Proof of Theorem \ref{Thm3}]
    Suppose $\alpha=0$. When $\gamma\in(\frac{6}{5},\frac{4}{3})$, we have
\begin{equation}
    \Lambda\le4\pi\int_0^{a(t)}\rho u^2r^2dr+Q(\rho(t))= H''(t)+12\pi(\eta+\frac{4}{3}\varepsilon) a^2(t)a'(t),
\end{equation}
integrate it over $[0,t]$, we have
\begin{equation}
\begin{aligned}
    \Lambda t+C_0\le& H'(t)+4\pi(\eta+\frac{4}{3}\varepsilon)(a(t))^3\\
    \le&4\pi\left(\int_0^{a(t)}\rho r^4dr\right)^{\frac{1}{2}}\left(\int_0^{a(t)}\rho u^2r^2dr\right)^{\frac{1}{2}}+4\pi(\eta+\frac{4}{3}\varepsilon)(a(t))^3\\
    \le&4\pi (M(\rho))^{\frac{1}{2}}E_0a(t)+4\pi(\eta+\frac{4}{3}\varepsilon)(a(t))^3,
\end{aligned}
\end{equation}
which implies $a(t)\gtrsim(\frac{t}{\eta+\frac{4}{3}\varepsilon})^{\frac{1}{3}}$. On the other hand,
\begin{equation}
    H''(t)+12\pi(\eta+\frac{4}{3}\varepsilon) a^2(t)a'(t)\le 2E_0,
\end{equation}
which gives
\begin{equation}
    -4\pi(M(\rho))^{\frac{1}{2}}E_0a(t)+4\pi(\eta+\frac{4}{3}\varepsilon)(a(t))^3\le 2E_0t,
\end{equation}
which implies $a(t)\lesssim(\frac{t}{\eta+\frac{4}{3}\varepsilon})^{\frac{1}{3}}$. As a result, $a(t)\approx(\frac{t}{\eta+\frac{4}{3}\varepsilon})^{\frac{1}{3}}$.

When $\gamma=\frac{4}{3}$, if $M(\rho)<M_{ch}:=\left(\frac{6}{C_{min}}\right)^{3/2}$, then we have:
\begin{equation}\label{Qrho4/3}
\begin{aligned}
    Q(\rho(t))=&\intb 3\rho^{\frac{4}{3}}\dbx-\frac{1}{2}\iint_{\mathbb{R}^3\times\mathbb{R}^3}\frac{\rho(x)\rho(y)}{|x-y|}\dbx\dby\\
    \ge&3\left(1-\frac{\left(M(\rho)\right)^{2/3}}{\left(M_{ch}\right)^{2/3}}\right)\intb \rho^{\frac{4}{3}}\dbx.
\end{aligned}
\end{equation}
Consider the functional (as in \cite{KL}):
\begin{equation}
\begin{aligned}
    I(t)=&\frac{1}{2}\int_0^{a(t)}4\pi\rho(r-(1+t)u)^2 r^2dr+3(1+t)^2\int_0^{a(t)}4\pi \rho^{\frac{4}{3}}r^2dr\\
    &-(1+t)^24\pi\int_0^{a(t)}4\pi\rho r\int_0^r\rho s^2dsdr\\
    =&H(t)-(1+t)H'(t)+(1+t)^2E(t),
\end{aligned}
\end{equation}
then:
\begin{equation}
\begin{aligned}
    I'(t)=&-(1+t)H''(t)+2(1+t)E(t)+(1+t)^2E'(t)\\
    \le&-(1+t)H''(t)+2(1+t)E(t)\\
    =&(1+t)\left(3\int_0^{a(t)}4\pi\rho^{\frac{4}{3}} r^2dr-4\pi\int_0^{a(t)}4\pi\rho r\int_0^r\rho s^2dsdr+12\pi\int_0^{a(t)}(\eta+\frac{4}{3}\varepsilon)\pr(r^2u)dr\right)\\
    \le&\frac{1}{1+t}I(t)+(1+t)4(\eta+\frac{4}{3}\varepsilon)\pi(a^3(t))',
\end{aligned}
\end{equation}
which implies
\begin{equation*}
    \left(\frac{I(t)}{1+t}\right)'\le4(\eta+\frac{4}{3}\varepsilon)\pi(a^3(t))',
\end{equation*}
so:
\begin{equation*}
    I(t)\lesssim(\eta+\frac{4}{3}\varepsilon)(1+t)a^3(t)+I(0)(1+t).
\end{equation*}
Combining with (\ref{Qrho4/3}), we have:
\begin{equation}\label{internal energy est}
    \int_0^{a(t)}\rho^{\frac{4}{3}}r^2dr\lesssim(\eta+\frac{4}{3}\varepsilon)(1+t)^{-1}a^3(t)+I(0)(1+t)^{-1}.
\end{equation}
By H\"older's Inequality and (\ref{internal energy est}),
\begin{equation*}
    \begin{aligned}
        M(\rho)=&\int_0^{a(t)}4\pi\rho r^2dr\lesssim\left(\int_0^{a(t)}\rho^{\frac{4}{3}}r^2dr\right)^{\frac{3}{4}}a^{\frac{3}{4}}(t)\\
        \lesssim&\left((\eta+\frac{4}{3}\varepsilon)(1+t)^{-1}a^3(t)+I(0)(1+t)^{-1}\right)^{\frac{3}{4}}a^{\frac{3}{4}}(t),
    \end{aligned}
\end{equation*}
then we get $a(t)\gtrsim (\frac{1+t}{\eta+\frac{4}{3}\varepsilon})^{\frac{1}{4}}$. By a similar computation as in $\gamma\in(\frac{6}{5},\frac{4}{3})$, we get $a(t)\lesssim(\frac{t}{\eta+\frac{4}{3}\varepsilon})^{\frac{1}{3}}$.
\end{proof}
\subsection{Proof of Remark \ref{DenDependVisc}}\label{Proof of expanding rate for density-dependent viscosity}
In this subsection, we prove the claims in Remark \ref{DenDependVisc}.

\textbf{Case 1: $\eta=0$}

It follows from (\ref{H''}) that
\begin{equation}
    \Lambda\le H''(t),
\end{equation}
integrate $H''(t)$ twice, we get
\begin{equation}
    \frac{1}{2}\Lambda t^2+Ct+D\lesssim2\pi\int_0^{a(t)}\rho r^4 dr\lesssim M(\rho)a^2(t),
\end{equation}
which indicates $a(t)\gtrsim t$.

\textbf{Case 2: $\eta\not=0$}

First, we consider $0<\alpha<1$. For the third term on the right hand side of (\ref{H''}), we substitute $\partial_r(r^2u)$ by $-\rho^{-1}r^2(\partial_t\rho+u\partial_r\rho)$ and integrate by parts
\begin{equation}
    \begin{aligned}
    -12\pi\int_0^{a(t)}\eta\rho^{\alpha}\partial_r(r^2u)dr&=12\pi\int_0^{a(t)}\eta\rho^{\alpha-1}r^2(\partial_t\rho+u\partial_r\rho)dr\\
    &=\frac{12\pi\eta}{\alpha}\partial_t\int_0^{a(t)}\rho^{\alpha}r^2dr+\frac{12\pi\eta}{\alpha}\int_0^{a(t)}\partial_r\rho^{\alpha}(r^2u)dr\\
    &=\frac{12\pi\eta}{\alpha}\partial_t\int_0^{a(t)}\rho^{\alpha}r^2dr-\frac{12\pi\eta}{\alpha}\int_0^{a(t)}\rho^{\alpha}\partial_r(r^2u)dr,
    \end{aligned}
\end{equation}
which implies
\begin{equation}
    12\pi\eta(\frac{1}{\alpha}-1)\int_0^{a(t)}\rho^{\alpha}\partial_r(r^2u)dr=\frac{12\pi\eta}{\alpha}\partial_t\int_0^{a(t)}\rho^{\alpha}r^2dr.
\end{equation}
So we have
\begin{equation}
    H''(t)+12\pi\eta\alpha^{-1}(\frac{1}{\alpha}-1)^{-1}\partial_t\int_0^{a(t)}\rho^{\alpha}r^2dr=4\pi\int_0^{a(t)}\rho u^2r^2dr+Q(\rho(t)),
\end{equation}
which gives
\begin{equation}
    \Lambda\le H''(t)+12\pi\eta\alpha^{-1}(\frac{1}{\alpha}-1)^{-1}\partial_t\int_0^{a(t)}\rho^{\alpha}r^2dr.
\end{equation}
Integrate it in $t$ and apply H\"older to $H'(t)$ and $\int_0^{a(t)}\rho^{\alpha}r^2dr$,
\begin{equation}
\begin{aligned}
    \Lambda t+C_0&\le H'(t)+12\pi\eta\alpha^{-1}(\frac{1}{\alpha}-1)^{-1}\int_0^{a(t)}\rho^{\alpha}r^2dr\\
    &\le4\pi\left(\int_0^{a(t)}\rho r^4dr\right)^{\frac{1}{2}}\left(\int_0^{a(t)}\rho u^2r^2dr\right)^{\frac{1}{2}}+\frac{12\pi\eta}{1-\alpha}\left(\int_0^{a(t)}\rho r^2dr\right)^{\alpha}\left(\int_0^{a(t)}r^2dr\right)^{1-\alpha}\\
    &\le4\pi (M(\rho))^{\frac{1}{2}}E_0a(t)+\frac{12\pi\eta}{1-\alpha}(M(\rho))^{\alpha}a^{3(1-\alpha)}(t).
\end{aligned}
\end{equation}
As a result, when $3(1-\alpha)\ge1$, i.e., $0\le\alpha\le\frac{2}{3}$, $a(t)\gtrsim (\frac{t}{\eta})^{\frac{1}{3(1-\alpha)}}$; when $3(1-\alpha)<1$, i.e., $\frac{2}{3}<\alpha<1$, $a(t)\gtrsim t$.

For $1\le\alpha\le\gamma$, it follows from (\ref{H''}) that
\begin{equation}
    H''(t)=4\pi\int_0^{a(t)}\rho u^2r^2dr+Q(\rho(t))-12\pi\int_0^{a(t)}\eta\rho^{\alpha}(\partial_ru+\frac{2u}{r})r^2dr,
\end{equation}
which implies
\begin{equation}
\begin{aligned}
    \Lambda\le& H''(t)+12\pi\int_0^{a(t)}\eta\rho^{\alpha}(\partial_ru+\frac{2u}{r})r^2dr\\
    \le&H''(t)+\frac{\Lambda}{2D}\int_0^{a(t)}4\pi\rho^{\alpha}  r^2dr+CD\int_0^{a(t)}\rho^{\alpha} (\partial_ru+\frac{2u}{r})^2r^2dr\\
    \le&H''(t)+\frac{\Lambda}{2D}\left(\int_0^{a(t)}4\pi\rho  r^2dr\right)^{\alpha\frac{\alpha-1}{\gamma-1}}\left(\int_0^{a(t)}4\pi\rho^{\gamma}  r^2dr\right)^{\alpha\frac{\gamma-\alpha}{\gamma-1}}+CD\int_0^{a(t)}\rho^{\alpha} (\partial_ru+\frac{2u}{r})^2r^2dr,
    \end{aligned}
\end{equation}
where $D>0$ can be chosen properly so that
\begin{equation}
    \frac{1}{2}\Lambda\le H''(t)+C\int_0^{a(t)}\rho^{\alpha} (\partial_ru+\frac{2u}{r})^2r^2dr,
\end{equation}
integrate it with respect to $t$ and apply H\"older inequality, we get
\begin{equation}
    \frac{1}{2}\Lambda t+C_0\le4\pi(M(\rho))^{\frac{1}{2}}E_0a(t)+CE_0,
\end{equation}
which gives $a(t)\gtrsim t$.  

    \section*{Acknowledgments}
    The author was supported in part by the NSF grant DMS-2306910.
    \section*{Appendix}
   \appendix
    \section{Functional analysis}
\label{Functional analysis}
\begin{mydef}[Compatible pair]\label{App0}\cite{ACDE}
    Consider a family of Banach spaces $\{X(t)\}_{t\in[0,T]}\ (X_0:=X(t))$ with the bounded linear, invertible map
\begin{equation}
    \varphi_t:X_0\rightarrow X(t)
\end{equation}
whose inverse is
\begin{equation}
    \varphi_{-t}:X(t)\rightarrow X_0.
\end{equation}
Assume further that 
\begin{enumerate}
    \item[I] $\varphi_0=id_{X_0}$;
    \item[II] there exists a constant $C_X$ independent of $t\in[0,T]$ such that
\begin{equation*}
    \|\varphi_tu\|_{X(t)}\le C_X\|u\|_{X_0},\ \forall u\in X_0;\ \|\varphi_{-t}u\|_{X_0}\le C_X\|u\|_{X(t)},\ \forall u\in X(t);
\end{equation*}
    \item[III] for all $u\in X_0$, the map $t\rightarrow\|\varphi_tu\|_{X(t)}$ is measurable.
\end{enumerate}
We say the pair $(X(t),\varphi_t)_t$ is compatible.

\end{mydef}
\begin{mydef}[The space $L^p_X$]\cite{ACDE}For $1\le p\le\infty$, define
\begin{equation}
    L^p_X:=\{u:[0,T]\rightarrow X(t),\ t\mapsto(u(t),t):\ \varphi_{-(\cdot)}u(\cdot)\in L^p(0,T;X_0)\}
\end{equation}
    with norm 
    \begin{equation}
        \|u\|_{L^p_X}:=\begin{cases}
            \left(\int_0^T\|u(t)\|^p_{X(t)}dt\right)^{\frac{1}{p}}\ (1\le p<\infty),\\
            ess\sup_{t\in[0,T]}\|u(t)\|_{X(t)}.
        \end{cases}
    \end{equation}
\end{mydef}
\begin{mydef}[The space $\mathbb{W}^{p,q}(X,Y)$]\cite{ACDE}Suppose $(X(t),\varphi^X_t)_t$ and $(Y(t),\varphi^Y_t)_t$ are compatible pairs satisfy $X(t)\hookrightarrow Y(t)$ continuously for all $t\in[0,T]$. Consider the space
\begin{equation}
    W^{p,q}(X_0,Y_0):=\{u\in L^p(0,T;X_0):u'\in L^q(0,T;Y_0)\},
\end{equation}
where $u'$ is the weak derivative with respective to $t$. We define the following space
\begin{equation}
    \mathbb{W}^{p,q}(X,Y):=\{u\in L^p_X:\mathcal{D}_tu\in L^q_Y\}
\end{equation}
with norm $\|u\|_{\mathbb{W}^{p,q}(X,Y)}:=\|u\|_{L^p_x}+\|\mathcal{D}_tu\|_{L^q_Y}$, where $\mathcal{D}_tu:=\varphi_t^X((\varphi_{-t}^Xu)')$. We say that the spaces $W^{p,q}(X_0,Y_0)$ and $\mathbb{W}^{p,q}(X,Y)$ are equivalent if
\begin{equation}
    v\in \mathbb{W}^{p,q}(X,Y)\text{ iff }\varphi^X_{-(\cdot)}v(\cdot)\in W^{p,q}(X_0,Y_0),
\end{equation}
and
\begin{equation}
    C_1\|\varphi^X_{-(\cdot)}v(\cdot)\|_{W^{p,q}(X_0,Y_0)}\le\|v\|_{\mathbb{W}^{p,q}(X,Y)}\le C_2\|\varphi^X_{-(\cdot)}v(\cdot)\|_{W^{p,q}(X_0,Y_0)}.
\end{equation}
\end{mydef}
\begin{lem}[Aubin-Lions lemma in evolving spaces]\label{App1}\cite[Theorem 5.2]{ACDE}Suppose $(X(t),\varphi^X_t)_t$, $(Y(t),\varphi^Y_t)_t$ and $(Z(t),\varphi^Z_t)_t$are compatible pairs with $X(t)\hookrightarrow Y(t)$, $X_0\hookrightarrow\hookrightarrow Z_0\hookrightarrow Y_0$ and $\varphi_t^Z|_{X_0}=\varphi_t^X$. Suppose $W^{p,q}(X_0,Y_0)$ and $\mathbb{W}^{p,q}(X,Y)$ are equivalent. Then for any $p\in(1,\infty)$ and $q\in[1,\infty]$, $\mathbb{W}^{p,q}(X,Y)\hookrightarrow\hookrightarrow L^p_Z$.
\end{lem}
\begin{lem}\label{App2}\cite{LionsII}
    Let $g_n,\ h_n$ converge weakly to $g,\ h$ respectively in $L^{p_1}(0,T;L^{p_2}(\Omega))$, $L^{q_1}(0,T;L^{q_2}(\Omega))$, where $1\le p_1,p_2\le+\infty$,
    $$\frac{1}{p_1}+\frac{1}{q_1}=\frac{1}{p_2}+\frac{1}{q_2}=1.$$
    We assume in addition that $\frac{\partial g_n}{\partial t}$ is bounded in $L^1(0,T;W^{-m,1}(\Omega))$ for some $m\ge0$ independent of $n$, and 
    \begin{equation*}
        \|h_n-h_n(\cdot+\xi,t)\|_{L^{q_1}(0,T;L^{q_2}(\Omega))}\rightarrow0
    \end{equation*}
    as $|\xi|\rightarrow0$ uniformly in $n$. Then $g_nh_n\rightharpoonup gh$ in the sense of distributions on $\Omega\times(0,T)$.
\end{lem}
\begin{lem}\label{App3}\cite{LionsI}Let $X$ be a separable reflexive Banach space and $\{f_n\}$ be bounded in $\Linf(0,T;X)$ for some $T>0$. Assume $f_n\in C([0,T];Y)$ where $Y$ is a Banach space such that $X\hookrightarrow Y$, $Y'$ is separable and dense in $X'$. Furthermore, we assume that for all $\varphi\in Y'$, $\langle\varphi,f_n(t)\rangle_{Y'\times Y}$ is uniformly continuous in $t\in[0,T]$ uniformly in $n\ge1$. Then $\{f_n\}$ is relatively compact in $C([0,T];X-w)$.
\end{lem}
\begin{lem}[Young's measure]\label{App3.5}\cite[Proposition 3.1.1]{JMR}Given a bounded family $\{u^{\varepsilon}\}\subset L^p(\mathcal{O})$ of $\mathbb{R}^N$-valued functions, there exists a subsequence and a measurable family of probability measures on $\mathbb{R}^N$, $\{\mu(y,\cdot):y\in\mathcal{O}\}$, such that for all continuous functions $f(\lambda)$ on $\mathbb{R}^N$ which tend to 0 at infinity and for all $\varphi\in C_0(\mathcal{O})$,
\begin{equation}
    \int_{\mathcal{O}}\varphi(y)f(u^{\varepsilon}(y))dy\rightarrow\int_{\mathcal{O}}\int\varphi(y)f(\lambda)\mu(y,d\lambda)dy.
\end{equation}
\end{lem}
\begin{rmk}\cite[Remark 3.1.6]{JMR}
    Let $u^{\varepsilon}$ be a pure, real-valued and bounded family in $L^p(\mathcal{O})$, $p>1$. Then the mean value of $\mu$,
    \begin{equation}
        m(y):=\int\lambda\mu(y,d\lambda)\ a.e.,
    \end{equation}
    is the weak limit in $L^p(\mathcal{O})$ of $u^{\varepsilon}$. The variance
    \begin{equation}
        \sigma(y):=\int|\lambda-m(y)|^2\mu(y,d\lambda)=\int\lambda^2\mu(y,d\lambda)-(m(y))^2\ a.e.
    \end{equation}
    is defined and belongs to $L^{\frac{p}{2}}(\mathcal{O})$, provided that $p\ge2$. Then, saying that $\sigma\equiv0$ on an open set $\mathcal{O}_1\subset\mathcal{O}$ means that, for almost all $y\in\mathcal{O}_1$, $\mu(y,d\lambda)$ is the probability measure concentrated at $\lambda=m(y)$, i.e., that $\mu(y,d\lambda)$ is the Dirac mass at $\lambda=m(y)$.
\end{rmk}
\begin{lem}\label{App4}\cite[Proposition 3.1.7]{JMR}Suppose that $u^{\varepsilon}$ is a pure bounded, family in $L^p(\mathcal{O})$, $p\ge2$, with Young measure $\mu$. Then the variance $\sigma$ vanishes on the bounded subset $\mathcal{O}_1\subset\mathcal{O}$, if and only if $u^{\varepsilon}\rightarrow m$ strongly in $L^q(\mathcal{O}_1)$ for all $q<p$.
\end{lem}
    \section{Derivation of (\ref{KLcondition})}\label{Derivation of (KLcondition)}
In this section, we briefly discuss the approach in \cite{KL} and derive (\ref{KLcondition}).

We apply H\"older inequality to the gravitational term:
  \begin{equation}
     \begin{aligned}
         \frac{1}{8\pi}\intr|\nabla\varPhi|^2\dbx&=\frac{1}{2}\intr\rho\varPhi\dbx\le\frac{1}{2}\|\rho\|_{L^p(B(0;a(t)))}\|\varPhi\|_{L^{p'}(\mathbb{R}^3)}\\
         &\le\frac{1}{2}\|\rho\|_{L^1(B(0;a(t)))}^{\theta}\|\rho\|^{1-\theta}_{L^{\gamma}(B(0;a(t)))}\|\varPhi\|_{L^{p'}(\mathbb{R}^3)},
     \end{aligned}
 \end{equation}
 where $\frac{1}{p}=\frac{\theta}{\gamma}+\frac{1-\theta}{1}$. By Hardy-Littlewood-Sobolev inequality,
 \begin{equation}
     \|\varPhi\|_{L^{p'}(\mathbb{R})^3}\le A_{\gamma}\|\rho\|_{L^{\gamma}(B(0;a(t)))},
 \end{equation}
 where $\frac{1}{p'}=\frac{1}{\gamma}-\frac{2}{3}$. So
 \begin{equation}
     \frac{1}{8\pi}\intr|\nabla\varPhi|^2\dbx\le\frac{1}{2}A_{\gamma}\|\rho\|_{L^1(B(0;a(t)))}^{1+\theta}\|\rho\|^{1-\theta}_{L^{\gamma}(B(0;a(t)))},
 \end{equation}
where $1+\theta=\frac{\gamma}{3(\gamma-1)}$, $1-\theta=\frac{5\gamma-6}{3(\gamma-1)}$. Since $\frac{1}{8\pi}\intr|\nabla\varPhi|^2\dbx=\frac{1}{2}\iint_{\mathbb{R}^3\times\mathbb{R}^3}\frac{\rho(\bx)\rho(\boldsymbol{y})}{|\bx-\boldsymbol{y}|}\dbx\dby$, we have
\begin{equation}
\begin{aligned}
    &\frac{1}{\gamma-1}\intb\rho^{\gamma}\dbx-\frac{1}{2}\iint_{\mathbb{R}^3\times\mathbb{R}^3}\frac{\rho(\bx)\rho(\boldsymbol{y})}{|\bx-\boldsymbol{y}|}\dbx\dby\\\ge&\frac{1}{\gamma-1}\intb\rho^{\gamma}\dbx-BM^{\frac{5\gamma-6}{3(\gamma-1)}}\left(\intb\rho^{\gamma}\dbx\right)^{\frac{1}{3(\gamma-1)}}
\end{aligned}
\end{equation}
for some positive constant $B$.

Now recall
\begin{equation*}
    \tilde{E_0}=4\pi\int_0^{a_0}\frac{1}{2}\rho_0u_0^2 r^2+\frac{1}{\gamma-1}\rho_0^{\gamma} r^2dr,
\end{equation*}
consider the function
\begin{equation}
    f(s)=\frac{1}{\gamma-1}s-BM^{\frac{5\gamma-6}{3(\gamma-1)}}s^{\frac{1}{3(\gamma-1)}}\ \ (s>0),
\end{equation}
we can easily find the maximum of the function is attained at $s^*=(\frac{B}{3})^{-\frac{3(\gamma-1)}{4-3\gamma}}M^{-\frac{5\gamma-6}{4-3\gamma}}$,
\begin{equation}
    f(s^*)=\frac{4-3\gamma}{\gamma-1}(\frac{B}{3})^{-\frac{3(\gamma-1)}{4-3\gamma}}M^{-\frac{5\gamma-6}{4-3\gamma}},
\end{equation}
and $f(s)$ is strictly increasing in $(0,s^*)$. Recall that in \cite{KL}, the critical mass is defined as follows
\begin{equation*}
    M<M_c=\left(\frac{4-3\gamma}{\gamma-1}(\frac{B}{3})^{-\frac{3(\gamma-1)}{4-3\gamma}}\right)^{\frac{4-3\gamma}{5\gamma-6}}(\frac{\tilde{E_0}}{4\pi})^{-\frac{4-3\gamma}{5\gamma-6}}.
\end{equation*}
By direct computation,
\begin{equation}
    s^*=(\frac{B}{3})^{-\frac{3(\gamma-1)}{4-3\gamma}}M^{-\frac{5\gamma-6}{4-3\gamma}}>(\frac{B}{3})^{-\frac{3(\gamma-1)}{4-3\gamma}}M^{-\frac{5\gamma-6}{4-3\gamma}}_c=\frac{\gamma-1}{4-3\gamma}\frac{\tilde{E_0}}{4\pi}>2(\gamma-1)\frac{\tilde{E_0}}{4\pi},
\end{equation}
and
\begin{equation}
    f(s^*)=\frac{4-3\gamma}{\gamma-1}(\frac{B}{3})^{-\frac{3(\gamma-1)}{4-3\gamma}}M^{-\frac{5\gamma-6}{4-3\gamma}}>\frac{1}{4\pi}\tilde{E_0}.
\end{equation}
On the other hand, it follows from the basic energy estimate that
\begin{equation}
    f\left(\int_0^{a(t)}\rho^{\gamma}r^2dr\right)\le\frac{1}{4\pi}\tilde{E_0}.
\end{equation}
Since $4\pi\int_0^{a_0}\rho_0^{\gamma}r^2dr\le(\gamma-1)\tilde{E_0}<4\pi s^*$, and $4\pi\int_0^{a(t)}\rho^{\gamma}r^2dr$ is changing continuously with respect to $t$, we conclude that $4\pi\int_0^{a(t)}\rho^{\gamma}r^2dr\le(\gamma-1)\tilde{E_0}<4\pi s^*$ for $t\ge0$. The following shows how the gravitational term is controlled
\begin{equation}\label{ZAMPcontrol}
    \begin{aligned}
        &3\int_0^{a(t)}\rho^{\gamma}r^2dr-4\pi\int_0^{a(t)}\rho r\int_0^r\rho s^2dsdr
        \\\ge&BM^{\frac{5\gamma-6}{3(\gamma-1)}}(s^*)^{\frac{4-3\gamma}{3(\gamma-1)}}\int_0^{a(t)}\rho^{\gamma}r^2dr-BM^{\frac{5\gamma-6}{3(\gamma-1)}}\left(\int_0^{a(t)}\rho^{\gamma}r^2dr\right)^{\frac{1}{3(\gamma-1)}}\\
        =&BM^{\frac{5\gamma-6}{3(\gamma-1)}}\left((s^*)^{\frac{4-3\gamma}{3(\gamma-1)}}-\left(\int_0^{a(t)}\rho^{\gamma}r^2dr\right)^{\frac{4-3\gamma}{3(\gamma-1)}}\right)\int_0^{a(t)}\rho^{\gamma}r^2dr,
    \end{aligned}
\end{equation}
in particular,
\begin{equation}
\begin{aligned}
    &3\int_0^{a_0}\rho^{\gamma}_0r^2dr-4\pi\int_0^{a_0}\rho_0 r\int_0^r\rho_0 s^2dsdr\\
    \ge&BM^{\frac{5\gamma-6}{3(\gamma-1)}}\left(\int_0^{a_0}\rho^{\gamma}_0r^2dr\right)\left((2(\gamma-1))^{\frac{4-3\gamma}{3(\gamma-1)}}-(\gamma-1)^{\frac{4-3\gamma}{3(\gamma-1)}}\right)(\frac{\tilde{E_0}}{4\pi})^{\frac{4-3\gamma}{3(\gamma-1)}}\\
        \ge& C_0M^{\frac{5\gamma-6}{3(\gamma-1)}}\left(\int_0^{a_0}\rho^{\gamma}_0r^2dr\right)^{\frac{1}{3(\gamma-1)}}       
\end{aligned}
\end{equation}
for some positive constant $C_0$.

    \bibliographystyle{abbrv}
    \bibliography{Lookup}

\end{document}